\documentclass[12pt]{article}

\usepackage{amsmath,amssymb,amsthm,mathtools,mathrsfs}
\usepackage{graphicx}
\usepackage{xcolor}
\usepackage{tikz}
\usepackage{newclude}
\usepackage{imakeidx}
\usepackage{bm}
\usepackage{cancel}

\usepackage[colorlinks=true,
            linkcolor=blue,
            citecolor=blue,
            urlcolor=red,
            bookmarksopen=true,
            pdfborder={0 0 0}
           ]{hyperref}
\makeindex

\allowdisplaybreaks

 \numberwithin{equation}{section}

\newtheorem{thm}{Theorem}[section]
\newtheorem{corl}{Corollary}[section]
\newtheorem{lema}{Lemma}[section]
\newtheorem{prop}{Proposition}[section]
\newtheorem{defi}{Definition}[section]
\newtheorem{rmk}{Remark}[section]

\renewcommand{\Pi}{\mathcal{P}}

\renewcommand{\Xi}{\mathcal{X}}

\let\div\relax
\DeclareMathOperator*{\div}{div}

\DeclarePairedDelimiter\abs{\lvert}{\rvert}
\DeclarePairedDelimiter\norm{\lVert}{\rVert}
\makeatletter
\let\oldabs\abs
\def\abs{\@ifstar{\oldabs}{\oldabs*}}
\let\oldnorm\norm
\def\norm{\@ifstar{\oldnorm}{\oldnorm*}}
\makeatother

\usepackage{scalerel}
\usepackage[usestackEOL]{stackengine}
\def\dashint{\,\ThisStyle{\ensurestackMath{
\stackinset{c}{.2\LMpt}{c}{.5\LMpt}{\SavedStyle-}{\SavedStyle\phantom{\int}}
}
\setbox0=\hbox{$\SavedStyle\int\,$}\kern-\wd0}\int}

\usepackage{adjustbox}
\usepackage{wrapfig}
\usepackage{enumitem}
\usepackage{protosem}

\title{Global existence of classical solutions for the multi-dimensional compressible Navier-Stokes-Poisson equations on solid balls for arbitrary spherically symmetric large initial data}
\date{}
\author{
\bf\large  Jie Fan$^{a}$, Xiangdi Huang$^{a,*}$, Muxi Lei$^{a}$\\
\small a. State Key Laboratory of Mathematical Sciences, Academy of Mathematics and Systems Science\\
\small Chinese Academy of Sciences, Beijing 100910, P.R.China;\\
\footnote{*Email addresses: fj0828@outlook.com\,\, (J. Fan); xdhuang@amss.ac.cn\,\, (X. Huang);  leimuxi25@mails.ucas.ac.cn\,\,  (M. X. Lei).  } }
\begin{document}
\maketitle
\begin{abstract}
Whether the 3D compressible Navier–Stokes–Poisson equations admit global classical solutions for general large initial data has long been a challenging open problem. In this paper, we provide an affirmative answer to this question under spherical symmetry on solid balls . Specifically, we consider the initial-boundary value problem for the multi-dimensional compressible equations with density-dependent viscosity coefficients satisfying the BD-type entropy equality, namely, assuming $\mu=\rho^{\alpha},\ \lambda=(\alpha-1)\rho^{\alpha}$ with $N=2, \alpha\in (\frac{1}{2},1]$ and $N=3, \alpha\in (\frac{5}{6},1]$, we establish the global existence of spherically symmetric classical solutions to the compressible Navier–Stokes–Poisson equations for both gaseous stars and plasmas with arbitrarily large initial data on solid balls. Our key observation lies in successfully handling the singularity at the center of the ball. By controlling the growth orders of the density and the gravitational potential at the central singularity, leveraging the structural advantages of the BD entropy and spherical symmetry, and fully exploiting the coupling between the effective velocity and the velocity, we establish $L^\infty$ estimates for the key quantities, which in turn yield upper and lower bound estimates for the density. This can be regarded as the first result on the existence of global classical solutions for arbitrarily large initial data to the compressible Navier-Stokes-Poisson equations in a truly multi-dimensional domain with high-dimensional features.
 \\[4mm]
\end{abstract}

\tableofcontents
\vspace{1cm}

\section{Introduction}
The compressible Navier--Stokes--Poisson (NSP) equations in $\mathbb{R}^N$ with $N=2,3$ govern the motion of compressible gaseous stars or plasmas in a self-consistent gravitational or electric field, given by
\begin{equation} 
\label{eq:NSP-original}
\left\{
\begin{array}{ll}
\rho_t + \mathrm{div}(\rho \mathbf{u}) = 0, \\[4pt]
\rho (\mathbf{u}_t + \mathbf{u}\cdot \nabla \mathbf{u}) + \nabla P(\rho) = -\rho \nabla \Phi + \mathrm{div}(\mu(\rho)\mathcal{D}\mathbf{u}) +\nabla(\lambda(\rho)\mathrm{div}\mathbf{u}), \\[4pt]
\Delta \Phi =\kappa( \rho - \bar{\rho}).
\end{array}
\right.
\end{equation}
Here, $\rho$, $\overline{\rho}$, $P$, and $\Phi$ denote the density, the background state density, the pressure, and a potential, respectively. The potential $\Phi$ represents the gravitational potential for gaseous stars when $\kappa>0$, and the electric field potential for plasmas when $\kappa<0$. $\mathcal{D}\mathbf{u}=\frac{1}{2}\left(\nabla\mathbf{u}+(\nabla\mathbf{u})^{T}\right)$ is the deformation tensor. The density-dependent Lam\'e viscosity coefficients, namely the shear coefficient $\mu(\rho)$ and the bulk coefficient $\lambda(\rho)$, they satisfy
\[
\mu(\rho) \ge 0, \qquad \mu(\rho) + N\,\lambda(\rho) \ge 0.
\]
The system is supplemented with the prescribed initial data
\begin{equation}
(\rho,\rho\mathbf{u},\Phi)|_{t=0}=(\rho_0,\rho\mathbf{u}_0,\Phi_0),
\end{equation}
and the boundary condition
\begin{equation}
\rho \mathbf{u}|_{\partial \Omega} = 0,\,\, \frac{\partial\Phi}{\partial n}|_{\partial\Omega}=0.
\end{equation}
The uniqueness condition of \(\Phi\) is given by
\begin{equation}\label{uniqueness}
\int_{\Omega} \Phi \, d\mathbf{x} = 0. 
\end{equation}
Under small perturbations, Zhang and Fang \cite{Zhang2009} established the global existence and uniqueness of weak solutions for the spherically symmetric vacuum free boundary problem without a hard core, with $\gamma > 1$ and the initial data being a small perturbation of some steady state. Moreover, they also demonstrated the stability of such a system. For the three-dimensional problem with $\kappa = 1$ and $\gamma \in (\frac{6}{5}, \frac{4}{3}]$, Duan and Li \cite{Duan2015} proved the global existence of spherically symmetric weak solutions under the stress-free boundary condition, with nonzero initial density and arbitrarily large initial data. Luo-Xin-Zeng \cite{LuoXinZeng2016a,LuoXinZeng2016b} established the existence and long-time stability of spherically symmetric smooth solutions for the three-dimensional viscous problem (with $\kappa = 1$) when $\gamma \in (\frac{4}{3}, 2)$, under small perturbations of the Lane-Emden solution.

Another important variable-viscosity model is the BD entropy model, which satisfies $\lambda(\rho) = \rho\mu'(\rho) - \mu(\rho)$. This entropy estimate was first introduced by Bresch-Desjardins-Lin \cite{Bresch2003a} and Bresch-Desjardins \cite{Bresch2003b}, and later used by Mellet-Vasseur \cite{Mellet2007} to study stability via new a priori bounds. A key breakthrough was made independently by Li-Xin \cite{Li2015} and Vasseur-Yu \cite{Vasseur2016}, who constructed global weak solutions for the degenerate case $\mu(\rho) = \rho$, $\lambda(\rho) = 0$, allowing vacuum and arbitrarily large data. Their results cover $N = 2$, $\gamma > 1$ and $N = 3$, $1 < \gamma < 3$. Subsequently, Bresch-Vasseur-Yu \cite{Bresch2022} generalized the result to three dimensions with a physically symmetric viscous stress tensor. For spherically symmetric flows, Huang-Meng-Zhang \cite{huangmengzhang} proved the global existence of strong solutions with large initial data and without vacuum, under the BD relation $\mu(\rho) = \rho^{\alpha}$, $\lambda(\rho) = (\alpha - 1)\rho^{\alpha}$.

In this paper, we consider the shear and bulk viscosity coefficients satisfying the BD relation:
\[
\mu(\rho)=\rho^\alpha,\quad \lambda(\rho)=(\alpha-1)\rho^\alpha.
\]
Under this assumption, we investigate the global existence of large classical solutions away from the vacuum. In this paper, we consider radially symmetric solutions to the initial-boundary value problem \eqref{eq:NSP-original}-\eqref{uniqueness} with large radially symmetric initial data. On a bounded domain $\Omega =B_{R}\subset \mathbb{R}^N$ with $N = 2, 3$, the radially symmetric formulation of \eqref{eq:NSP-original}-\eqref{uniqueness} reads:
\begin{equation} \label{eq:radial-ansatz}
\rho(\mathbf{x}, t) = \rho(r,t), \quad 
\mathbf{u}(\mathbf{x},t) = u(r,t) \frac{\mathbf{x}}{r}, \quad 
\Phi(\mathbf{x},t) = \Phi(r,t) \quad \text{for } r = |\mathbf{x}|.
\end{equation}
In spherically symmetric coordinates, \eqref{eq:NSP-original} becomes
\begin{equation} \label{eq:NSP-radial-eulerian}
\left\{
\begin{array}{ll}
\rho_t + (\rho u)_r + \dfrac{N-1}{r}\rho u = 0, \\[8pt]
\rho (u_t + u u_r) + (\rho^{\gamma})_r = { \left(\alpha\rho^{\alpha}(u_r+\frac{N-1}{r}u)\right)_r }- \frac{N-1}{r}u{(\rho^{\alpha})_r}-\rho\Phi_r , \\[8pt]
\Phi_{rr} + \dfrac{N-1}{r}\Phi_r = \kappa(\rho - \bar{\rho}),
\end{array}
\right.
\end{equation}
with initial data
\begin{equation} 
\rho(r,0)=\rho_0(r),\,\, \rho u(r,0)=m_0(r),
\end{equation}
and the boundary conditions
\begin{equation} \label{eq:radial-bc}
\rho u(0,t) =\rho u(R,t)= 0, \quad \Phi_r(0,t)=\Phi_r(R,t) = 0, \quad t \geq 0.
\end{equation}

\section{Main results}
\begin{defi}[Global classical solution]A radially symmetric
 triple $(\rho, \mathbf{u}, \Phi)$ with $\rho > 0$ is said to be a global classical solution to the initial-boundary-value problem \eqref{eq:NSP-original}-\eqref{uniqueness} provided that, for any $0 < \tau < T$, the following conditions are satisfied:
\begin{equation} \label{eq:global-regularity}
    \begin{cases}
        \rho \, \in C([0, T]; H^3(\Omega)), \quad \rho_t \,  \in C([0, T]; H^2(\Omega)), \\
        \rho_{tt}\, \in L^\infty(0, T; L^2(\Omega)) \cap L^2(0, T; H^1(\Omega)),\,\nabla \Phi \in C([0, T]; H^4(\Omega)),\\
        \mathbf{u} \in C([0, T]; H_0^1(\Omega) \cap H^3(\Omega)) \cap L^2(0, T; H^4(\Omega)) \cap L^\infty(\tau, T; H^4(\Omega)), \\
        \mathbf{u}_t \in L^\infty(0, T; H_0^1(\Omega)) \cap L^2(0, T; H^2(\Omega)) \cap L^\infty(\tau, T; H^2(\Omega)), \\
        \mathbf{u}_{tt} \in L^2(0, T; L^2(\Omega)) \cap L^\infty(\tau, T; L^2(\Omega)) \cap L^2(\tau, T; H^1(\Omega)),\\
        \nabla\Phi_t\in L^{\infty}(0,T;H^3(\Omega)),\,\,\nabla\Phi_{tt}\in L^{\infty}(0,T;H^1(\Omega)).
    \end{cases}
\end{equation}
\end{defi}

\begin{defi}
 Define the admissible set $A_{set}$ as follows
\[
A_{\text{set}} := \left\{1 + \frac{s}{2k+1} \;\middle|\; s\in \mathbb{N}^+,\; k\in \mathbb{N}\right\}.
\]
One can easily verify that $A_{\text{set}}$ is dense in $(1,\infty)$.

\end{defi}

\begin{defi}
For any $n>1$, we introduce two functions
\[
\begin{aligned}
\alpha_{2,-}(n) &:= 1 - \frac{n\sqrt{2n-1} - 2n + 1}{n^2 - 2n + 1},\\[4pt]
\alpha_{2,+}(n) &:= 1 + \frac{n\sqrt{2n-1} + 2n - 1}{n^2 - 2n + 1}.
\end{aligned}
\]
The function $\alpha_{2,-}(\cdot)$ is strictly increasing on $(1,\infty)$ with
\[
\lim_{n\to1^+}\alpha_{2,-}(n)=\frac12,\qquad \lim_{n\to\infty}\alpha_{2,-}(n)=1,
\]
while $\alpha_{2,+}(\cdot)$ is strictly decreasing on $(1,\infty)$ with
\[
\lim_{n\to1^+}\alpha_{2,+}(n)=\infty,\qquad \lim_{n\to\infty}\alpha_{2,+}(n)=1.
\]
\end{defi}

\begin{defi}
We define the function $n_2(\cdot):(\tfrac12,\infty)\to(1,\infty]$ as follows:
\begin{itemize}
  \item If $\alpha\in(\tfrac12,1)$, let $n_2(\alpha)$ be the unique number satisfying $\alpha_{2,-}\bigl(n_2(\alpha)\bigr)=\alpha$;
  \item If $\alpha=1$, set $n_2(1):=+\infty$;
  \item If $\alpha\in(1,\infty)$, let $n_2(\alpha)$ be the unique number satisfying $\alpha_{2,+}\bigl(n_2(\alpha)\bigr)=\alpha$.
\end{itemize}
\end{defi}

\begin{defi}
For $n>1$ we define
\[
\begin{aligned}
\alpha_{3,-}(n) &:= 1 - \frac{\sqrt{4n(4n^2-n-1)+1} - 6n + 3}{4n^2 - 8n + 4},\\[4pt]
\alpha_{3,+}(n) &:= 1 + \frac{\sqrt{4n(4n^2-n-1)+1} + 6n - 3}{4n^2 - 8n + 4}.
\end{aligned}
\]
The function $\alpha_{3,-}(\cdot)$ is strictly increasing on $(1,\infty)$ with
\[
\lim_{n\to1^+}\alpha_{3,-}(n)=\frac23,\qquad \lim_{n\to\infty}\alpha_{3,-}(n)=1,
\]
whereas $\alpha_{3,+}(\cdot)$ is strictly decreasing on $(1,\infty)$ with
\[
\lim_{n\to1^+}\alpha_{3,+}(n)=\infty,\qquad \lim_{n\to\infty}\alpha_{3,+}(n)=1.
\]
\end{defi}

\begin{defi}
We define $n_3(\cdot):(\tfrac23,\infty)\to(1,\infty]$ by the following rules:
\begin{itemize}
  \item For $\alpha\in(\tfrac23,1)$, $n_3(\alpha)$ is the unique solution of $\alpha_{3,-}\bigl(n_3(\alpha)\bigr)=\alpha$;
  \item For $\alpha=1$, set $n_3(1):=+\infty$;
  \item For $\alpha\in(1,\infty)$, $n_3(\alpha)$ is the unique solution of $\alpha_{3,+}\bigl(n_3(\alpha)\bigr)=\alpha$.
\end{itemize}
\end{defi}
\subsection{Global classical solutions away from vacuum $\frac{N-1}{N}<\alpha<1$}
The first result regarding the global existence of classical solutions for large initial data far from vacuum can be stated as follows.
\begin{thm}\label{thm:global-classical-2dalpha1}
   (Global classical solutions for \(N = 2\))
Assume that
\begin{equation}\label{n2laphasmaal1}
\frac{1}{2}<\alpha<1,\,\,\gamma>2-\alpha,
\end{equation}
and the spherically symmetric initial data $(\rho_0, \mathbf{u}_0, \Phi_0)$ satisfies
\begin{equation}
0 < {\rho_{*}} \leq \rho_0 \leq \rho^{*}, 
\end{equation}
and
\begin{equation}
(\rho_0, \mathbf{u}_0) \in H^3(\Omega),
\mathbf{u}_0|_{\partial\Omega} = 0,
\end{equation}
where $\rho_{*}$ and $\rho^{*}$ are two constants. Then, there exists a unique global spherically symmetric classical solution to the initial-boundary value problem \eqref{eq:NSP-original}-\eqref{uniqueness} such that, for any \((\mathbf{x},t) \in \Omega \times [0,T]\),
\begin{equation}
(C(T))^{-1} \leq \rho(\mathbf{x},t) \leq C(T), \end{equation}
where $C(T)>0$ is a constant depending on the initial data and $T$.
\end{thm}

\begin{thm}\label{thm:global-classical-3dalpha1}
   (Global classical solutions for $N = 3$) Assume that
   \begin{equation}
    \begin{cases}\label{plasmas3dalphagamma}
        \frac{5}{6}<\alpha<1,\\
        \frac{4}{3}<\gamma<4\alpha-1-\alpha^2;
    \end{cases}\end{equation}
and the spherically symmetric initial data $(\rho_0, \mathbf{u}_0, \Phi_0)$ satisfy
\begin{equation}
0 < \rho_{*} \leq \rho_0 \leq \rho^{*}, 
\end{equation}
and
\begin{equation}
(\rho_0, \mathbf{u}_0) \in H^3(\Omega), \quad 
\mathbf{u}_0|_{\partial\Omega} = 0.
\end{equation}
Then, there exists a unique global spherically symmetric classical solution to the initial-boundary value problem \eqref{eq:NSP-original}-\eqref{uniqueness} such that, for any \((\mathbf{x},t) \in \Omega \times [0,T]\),
\begin{equation}
(C(T))^{-1} \leq \rho(\mathbf{x},t) \leq C(T), \end{equation}
where $C(T)>0$ is a constant depending on the initial data and $T$.
\end{thm}
The aim of this paper is to establish the existence of global classical solutions. With the local well-posedness of classical solutions and the blow-up criterion established, the key to extending the local solution globally lies in establishing uniform upper and lower bounds for the density. Consequently, deriving these bounds constitutes crucial step of our proof, the main ideas of which are outlined below.

\begin{itemize}
    \item \textbf{The Two-Dimensional Case}
    
    In two dimensions, the derivation of the upper bound for the density relies on establishing higher integrability estimates for the effective velocity $w = u + r(\rho^{\alpha})_x$ and the velocity $u$, with the final step relying on the Sobolev embedding inequality. The proof is outlined step by step in the following.

    First, we establish the basic energy estimate and the BD entropy estimate, which yield the $L^2$ estimates for $u$ and $w$. In this process, a singularity arises at the origin, reflected by the term $\frac{x}{r}$ in the momentum equation (corresponding to the Poisson term in Eulerian coordinates). To address this singularity, we make full use of the weighted estimates for the density near the origin derived from the BD entropy and the basic energy estimate, namely the control of $\|\rho^{\alpha-\frac{1}{2}} r^\xi\|_{L^\infty(0,R)}$ in two dimensions, which allows us to convert singular integrals such as $\int_0^1 \frac{x^2}{r^2} dx$ into quantities controlled by $\|u\|^2_{L^2(0,1)}$ and $\|w\|^2_{L^2(0,1)}$. This weighted estimate plays a key role in all subsequent estimates.
    
    Next, after obtaining the $L^2$ integrability estimates for the effective velocity $w$ and the velocity $u$, we return to the weighted estimates for the density near the origin and rigorously establish its boundedness. This boundedness was previously used as a tool to handle the singularity. Meanwhile, in the two-dimensional case, we also obtain higher integrability estimates for the density. Building upon the higher integrability estimates for the density, we establish a relationship between $x$ and $r$, which plays a crucial role in the following  higher integrability estimates for $w$ and $u$---it eliminates the singularity near the origin caused by the Poisson term.
    
    Third, for any $\alpha$ and $\gamma$ satisfying \eqref{n2laphasmaal1}, we can choose a constant $1 < k \in A_{\text{set}}$ depending on $\alpha$ such that

\begin{equation}\nonumber
1-\frac{k\sqrt{2k-1}-2k+1}{k^2-2k+1}<\alpha<1-\frac{1}{2k}. 
\end{equation}
The choice of $k$ above ensures the high integrability of $u$ and $w$, which further yields the upper and lower bounds of the density.
\end{itemize}
\begin{itemize}
\item \textbf{The Three-Dimensional Case}

In the three-dimensional case, the Poisson term $\frac{x}{r^2}$ exhibits a stronger singularity near the origin due to the higher power of $r^2$. The weighted estimates for the density near the origin derived from the basic energy estimate and the BD entropy estimate are therefore insufficient to close the estimates within the range $\gamma>1$; indeed, closing the estimates would require $\gamma>\frac{4}{3}$, which is a more restrictive condition.
To overcome this obstacle, we carry out the following estimates.

\textbf{Step 1:} In carrying out the basic energy estimate, careful treatment is required for the term $\int_0^1 \frac{x}{r} dx$. In the case $\kappa = 1$ (gaseous star), when this integral appears on the right-hand side, it can ultimately be controlled by $\|\rho-1\|^2_{L^{\frac{6}{5}}(\Omega)}$, and absorbing it into the $\|\rho\|^{\gamma}_{L^\gamma(\Omega)}$ requires $\gamma > \frac{4}{3}$. In the case $\kappa = -1$ (plasma), when it appears on the left-hand side, no additional treatment is needed. When performing the BD entropy estimate, the Poisson term yields a term of the form $\int_{0}^{\tau}\|\rho\|^{1+\alpha}_{L^{1+\alpha}(\Omega)}ds$, which must be absorbed by the dissipation term $\int_{0}^{\tau}\|\nabla \rho^{\frac{\gamma+\alpha-1}{2}}\|^2_{L^2(\Omega)}ds$ to close the $L^2$ estimate for the effective velocity $w$.

\begin{sloppypar}\textbf{Step 2:} The weighted estimate for the density near the origin takes the form \(
\|\rho^{\alpha-\frac{1}{2}} r^{\frac{1}{2}+\xi}\|_{L^\infty(0,R)}
\), which differs from the two-dimensional case in the power of the weight. Similar to the two-dimensional case, the basic energy estimate and the BD entropy estimate allow us to establish a relationship between $x$ and $r$, which  is the key to the subsequent higher integrability estimates by eliminating the singularity near the origin caused by the Poisson term. 
\end{sloppypar}

\textbf{Step 3:} Similar to the two-dimensional case, when $\alpha$ and $\gamma$ satisfy \eqref{plasmas3dalphagamma}, we can choose a constant $1 < k \in A_{\text{set}}$ depending on $\alpha$ and $\gamma$ such that
\begin{equation}\label{3dalphagamma-1}
\frac{4k+3}{4k+6}< \alpha < 1 - \frac{1}{2k} ,
\end{equation}
and
\begin{equation}\label{3dalphagamma1-11}
\frac{4}{3} < \gamma < 3\alpha -1 + \frac{\alpha}{2k}. 
\end{equation}

\eqref{3dalphagamma-1},  \eqref{3dalphagamma1-11} guarantee the high integrability of the velocity $u$. However, due to the stronger singularity arising from the Poisson term, namely the appearance of the integral $\int_0^1 (\frac{x}{r^2})^{2k} dx$, this leads to a more restrictive condition on $\alpha$. On the other hand, the construction of the set $A_{\text{set}}$ still ensures that the effective velocity $w$ also possesses high integrability.

\textbf{Step 4:} Applying the one-dimensional  Sobolev embedding inequality yields an inequality for the density: $$R_T^\beta \leq C(T) R_T^{\sigma} + C(T).$$ By choosing appropriate parameters $\beta$ and $\sigma$ satisfying $\beta > \sigma$, this inequality can be closed, leading to a uniform upper bound for the density. 
\end{itemize}

\subsection{Global classical solutions away from vacuum $\alpha=1$}
\begin{thm}\label{thm:global-classical-2d}
   (Global classical solutions for \(N = 2\))
Assume that
\begin{equation}\label{gamma1alpha=1}
\gamma>1,
\end{equation}
and the spherically symmetric initial data $(\rho_0, \mathbf{u}_0, \Phi_0)$ satisfy
\begin{equation}
0 < {\rho}_* \leq \rho_0 \leq {\rho}^{*}, 
\end{equation}

\begin{equation}
(\rho_0, \mathbf{u}_0) \in H^3(\Omega),
\mathbf{u}_0|_{\partial\Omega} = 0.
\end{equation}
Then, there exists a unique global spherically symmetric classical solution to the initial-boundary value problem \eqref{eq:NSP-original}-\eqref{uniqueness} such that, for any \((\mathbf{x},t) \in \Omega \times [0,T]\),
\begin{equation}
(C(T))^{-1} \leq \rho(\mathbf{x},t) \leq C(T), \end{equation}
where $C(T)>0$ is a constant depending on the initial data and $T$.
\end{thm}

\begin{thm}\label{thm:global-classical-3d}
   (Global classical solutions for $N = 3$)
Assume that 
\begin{equation}\label{assumptiongamma1}
    \frac{4}{3}<\gamma<3,
\end{equation}
the spherically symmetric initial data $(\rho_0, \mathbf{u}_0, \Phi_0)$ satisfy
\begin{equation}
0 < {\rho}_* \leq \rho_0 \leq {\rho}^{*}, 
\end{equation}
and
\begin{equation}
(\rho_0, \mathbf{u}_0) \in H^3(\Omega), \quad 
\mathbf{u}_0|_{\partial\Omega} = 0.
\end{equation}
Then the initial-boundary value problem \eqref{eq:NSP-original}-\eqref{uniqueness} admits a unique global spherically symmetric classical solution. 
Moreover, for every $(\mathbf{x},t) \in \Omega \times [0,T]$,
\begin{equation}
(C(T))^{-1} \leq \rho(\mathbf{x},t) \leq C(T),
\end{equation}
where $C(T)>0$ is a constant depending on the initial data and $T$.
\end{thm}

\begin{rmk}
By fully utilizing the weighted estimates near the center established through the BD entropy, this paper investigates the global existence of spherically symmetric classical solutions with large initial data for the compressible Navier–Stokes–Poisson equations with density-dependent viscosity coefficients satisfying the BD entropy relation on a bounded domain containing the origin.
\end{rmk}

\begin{rmk}
The singularity near the origin is stronger in three dimensions than in two dimensions. Therefore, the two cases $\kappa = 1$ and $\kappa = -1$ must be treated separately. For $\kappa = 1$, which corresponds to gaseous stars, a careful treatment is required to close the basic energy estimate, which necessitates the condition $\gamma > \frac{4}{3}$.
\end{rmk}
The analysis for $\alpha = 1$ follows the same lines as the case $\alpha < 1$. For the two-dimensional problem, after obtaining $L^2$ estimates for both the velocity $u$ and the effective velocity $w$, we further derive a weighted estimate for the density near the origin. This estimate supports the subsequent derivation of $L^4$ bounds for $u$ and $w$. By the two-dimensional Sobolev embedding inequality, we obtain the $L^\infty$ estimate for the density. In the crucial step of establishing the lower bound for the density, an estimate for $\|\nabla\rho^{\frac{1}{2n}}\|_{L^\infty(\Omega)}$ is required, which in turn depends on $L^\infty$ estimates for both $u$ and $w$. To this end, we derive a coupled Gronwall-type inequality governing the $L^\infty$ norms of $w$ and $u$. When estimating the Poisson term, $\int_{0}^{1}\frac{x}{r}dx$ in the 2D case ensures that no singularity arises. For $\gamma > 1$, we are ultimately able to close the a priori estimates and thereby obtain a lower bound for the density.

In the three-dimensional spherically symmetric case, the analysis of the Navier-Stokes-Poisson system is substantially complicated by the strong singularity of the Poisson term $\frac{x}{r^2}$ near the origin, unlike in two dimensions, so a more delicate treatment of estimates is needed. The basic energy estimate behaves differently depending on the sign of the coupling parameter $\kappa$. When $\kappa = 1$ (gaseous star), the Poisson term appears on the right-hand side. The singular integral $\int_0^1 \frac{x}{r} dx$ is handled by the $L^{6/5}$-norm of the density and then absorbed by the potential energy $\|\rho\|_{L^\gamma(\Omega)}^\gamma$, which forces the condition $\gamma > 4/3$. When $\kappa = -1$ (plasma), this term moves to the left-hand side and actually helps the estimate, so no restriction on $\gamma$ arises from this step. The BD entropy estimate gives a weighted bound of the form $\|\rho^{1/2} r^{1/2+\xi}\|_{L^\infty(0,R)}$, where the power of the radial weight is different from the two-dimensional case. It also provides a key relation between $x$ and $r$, which relies on the uniform bound $\sup_{0\leq t\leq T} \|\rho(t)\|_{L^3(\Omega)} \leq C$ and is essential for removing the singularity caused by the Poisson term. Owing to the specific structure of these weighted estimates, the standard Sobolev embedding theorems fail to yield direct $L^4$ bounds for the velocities $u$ and $w$ when the adiabatic exponent lies within the intermediate range $1 < \gamma < 2$; instead, a refined higher-order integrability scheme is employed to establish that the quantities $\int_0^R \rho u^{2n} r^2 dr$ and $\int_0^R |\partial_r \rho^{1/(2n)}|^{2n} r^2 dr$ are bounded by $R_T^{2n\sigma}$, where $R_T = \sup_{0\leq t\leq T} \|\rho(t)\|_{L^\infty} + 1.$  In this process, the Poisson term gives rise to a singular term $\int_0^1 (\frac{x}{r^2})^{2n} dx$; by means of the density integrability established above, a proper relation between $x$ and $r$ is derived, which allows this term to be controlled. Applying a one-dimensional Sobolev embedding inequality subsequently yields an estimate of the form $R_T^\beta \leq C(T) R_T^\sigma + C(T)$; by  suitably selecting exponents satisfying $\beta > \sigma$, this inequality is closed to provide a uniform upper bound for the density, a method that remains valid for the full range $1 < \gamma < 3$. Finally, once uniform $L^\infty$ estimates for the velocity fields $w$ and $u$ are obtained, by introducing the specific volume $v = 1/\rho$ and applying a one-dimensional Sobolev embedding to the quantity $v^\beta$, together with the established bound on $\|r^2 \rho_x\|_{L^\infty}$, an upper bound for $v$ is derived, which yields a strict positive lower bound for the density.

\section{Preliminary}
\label{sec-2}

The local existence theory for classical solutions stated below can be established using standard methods similar to those in \cite{tanzhang}, so we omit the proof.

\begin{lema} \label{lem:local-existence}
    Assume that the spherically symmetric initial data $(\rho_0,\mathbf{u}_0,\Phi_0)$ and $\gamma$ satisfy the conditions of Theorem \eqref{thm:global-classical-2d} and Theorem \eqref{thm:global-classical-3d}. Then there exists a small time $T_0 > 0$ such that the initial-boundary value problem \eqref{eq:NSP-original}-\eqref{uniqueness} has a unique symmetric classical solution $(\rho,\mathbf{u},\Phi)$ on $\Omega \times [0,T_0]$ with $\rho,s > 0$, and for every $0 < \tau < T_0$, satisfying
    
    \begin{equation} \label{eq:local-regularity}
    \begin{cases}
        \rho \, \in C([0, T_0]; H^3(\Omega)), \quad \rho_t \,  \in C([0, T_0]; H^2(\Omega)), \\
        \rho_{tt}\, \in L^\infty(0, T_0; L^2(\Omega)) \cap L^2(0, T_0; H^1(\Omega)),\,\nabla \Phi \in C([0, T_0]; H^4(\Omega)),\\
        \mathbf{u} \in C([0, T_0]; H_0^1(\Omega) \cap H^3(\Omega)) \cap L^2(0, T_0; H^4(\Omega)) \cap L^\infty(\tau, T_0; H^4(\Omega)), \\
        \mathbf{u}_t \in L^\infty(0, T_0; H_0^1(\Omega)) \cap L^2(0, T_0; H^2(\Omega)) \cap L^\infty(\tau, T_0; H^2(\Omega)), \\
        \mathbf{u}_{tt} \in L^2(0, T_0; L^2(\Omega)) \cap L^\infty(\tau, T_0; L^2(\Omega)) \cap L^2(\tau, T_0; H^1(\Omega)).
    \end{cases}
\end{equation}
\end{lema}

As mentioned in \cite{Bourguignon}, we now recall certain estimates for elliptic equations.
\begin{lema}
      \label{lem:poisson-neumann}
Given an $f \in H^k(\Omega)$  $(k \in \mathbb{N})$ and a $g \in H^{k+1-1/2}(\partial\Omega)$ such that
\begin{equation}
\int_{\Omega} f \, d\mathbf{x} = \int_{\partial\Omega} g \, dS,
\end{equation}
then there exists a $v \in H^{k+2}(\Omega)$  satisfying
\begin{equation}
\begin{cases} 
\Delta v = f, & \text{in } \Omega, \\[4pt]
\displaystyle \frac{\partial v}{\partial n} = g, & \text{on } \partial\Omega,
\end{cases}
\end{equation}
and
\begin{equation}\label{psiestimate}
\|\nabla v\|_{H^{k+1}} \leq C\bigl(\|f\|_{H^k} + \|g\|_{H^{k+1-1/2}}\bigr).
\end{equation}    
\end{lema}
Next, we recall the following lemma from \cite{huangmengzhang}.
\begin{lema}\label{mdense}
    Let \(k \in \mathbb{N}\) and \(s \in \mathbb{N}^{+}\). Then there exists \(\varepsilon > 0\) such that
\[
a(a + b)^{1 + \frac{2s}{2k + 1}} \geq \varepsilon |a|^{2 + \frac{2s}{2k + 1}} - |b|^{2 + \frac{2s}{2k + 1}}, \quad \forall a, b \in \mathbb{R}.
\]
\end{lema}

\section{Global classical solutions away from vacuum ($\alpha<1$)}
Lemma \ref{lem:local-existence} guarantees the existence of a unique local classical solution to the initial-boundary-value problem \eqref{eq:NSP-original}-\eqref{uniqueness} on $\Omega \times [0, T_0]$.
Let $T^*$ denote the maximal existence time of this classical solution.
We now fix any $T$ with $0 < T < T^*$.
For classical solutions, it is advantageous to rewrite the IBVP \eqref{eq:NSP-original}-\eqref{uniqueness} in Lagrangian coordinates.
For $r \in [0, R]$ and $t \in [0, T]$, we define the Lagrangian coordinates $ (x,\tau)$  as
\[
x(r, t) = \int_0^r \rho(s,t)s^{N-1} ds, \quad \tau(r,t) = t,
\]
which translates domain \([0, R] \times [0, T]\) into a fixed domain \([0, 1] \times [0, T]\). A direct calculation shows that
\[
\nabla_{(r,t)}x = (\rho r^{N-1},-\rho ur^{N-1}), \quad \nabla_{(r,t)}\tau = (0, 1),
\]
\[
\nabla_{(x,\tau)}r = (\rho^{-1}r^{1-N}, u), \quad \nabla_{(x,\tau)}t = (0,1).
\]
Without loss of generality, we assume $\bar{\rho} = 1$ in the following. In Lagrangian coordinates, \eqref{eq:NSP-radial-eulerian}-\eqref{eq:radial-bc} transforms into
\begin{equation} \label{eq:NSP-lagrangian-rhoalpha}
\left\{
\begin{array}{ll}
\rho_{\tau} + \rho^2 (r^{N-1}u)_x = 0, \\[6pt]
u_{\tau} + r^{N-1} (\rho^{\gamma})_x =  \alpha r^{N-1} \left( \rho^{1+\alpha} (r^{N-1}u)_x \right)_x - (N-1)u r^{N-2} (\rho^{\alpha})_x - {\kappa}\dfrac{x}{r^{N-1}}+{\kappa}\dfrac{r}{N}, \\[10pt]
\left( \rho r^{2N-2} \Phi_x \right)_x =\kappa (1-\rho^{-1}),
\end{array}
\right.
\end{equation}
with the initial-boundary conditions 
\begin{equation} \label{eq:lagrangian-ic-bc}
\left\{
\begin{array}{ll}
\rho(x,0) = \rho_0(x), \quad u(x,0) = u_0(x), &\quad x \in [0, 1], \\[4pt]
u(0,\tau)=u(1,\tau) = 0, \quad \Phi_x(0,\tau)=\Phi_x(1,\tau) = 0, &\quad \tau \geq 0.
\end{array}
\right.
\end{equation}
Let the effective velocity be defined as \begin{equation}
    w = u + r^{N-1}(\rho^\alpha)_x.
\end{equation}
By the definition of the effective velocity and the momentum equation $\eqref{eq:NSP-lagrangian-rhoalpha}_1,$ we deduce that
\begin{equation}\label{bd2_alpha<1}
w_\tau + (\rho^\gamma)_x r^{N-1} = -\kappa\frac{x}{r^{N-1}}+{\kappa}\frac{r}{N}.
\end{equation}
\subsection{A priori estimates for 
$N=2$: lower-order estimates}

We first provide a basic energy estimate, but due to the appearance of the Poisson term, there is a constraint on $\gamma$.
\begin{lema}\label{2d-u2}
Under the assumption of \eqref{n2laphasmaal1}, there exists a constant $C(T)>0$ such that
\begin{equation}\label{basic2_alpha<1}
       \sup_{0\leq t\leq T}\int_{0}^{1} \left( u^{2}+\rho^{\gamma-1}  \right) dx + \int_{0}^{T}  \int_{0}^{1} \left( \rho^{\alpha-1}\frac{u^{2}}{r^{2}} +\rho^{1+\alpha}  r^2u_x^2\right) dx dt \le C(T) .
   \end{equation} 
   
\end{lema}
\begin{proof}
Setting $N=2$ in $\eqref{eq:NSP-lagrangian-rhoalpha}_1$, we take the $L^2$ inner product of $\eqref{eq:NSP-lagrangian-rhoalpha}_2$ with $u$ on $[0,1]$. Then, integration by parts and the use of $\eqref{eq:NSP-lagrangian-rhoalpha}_1$ give
 \begin{equation}
 \begin{aligned}
     \int^{1}_{0}u _{\tau}udx+ \int^{1}_{0} r  \left(\rho^{\gamma}\right) _{x}udx =& \alpha\int^{1}_{0}ru\left(\rho^{1+\alpha} \left(r  u\right) _{x}\right) _{x}dx \\
     &-\int^{1}_{0}(\rho^{\alpha})_{x}u^{2} dx-\kappa\int^{1}_{0} \frac {x}{r }udx+\frac{\kappa}{2}\int^{1}_{0}urdx.
    \end{aligned}\end{equation}
    That is
\begin{equation}\begin{aligned}
    \label{u1}
 &\frac{d}{d\tau} \int_0^1 \left( \frac{1}{2}u^2 + \frac{1}{\gamma-1} \rho^{\gamma-1} \right) dx +  \int_0^1 \left[ \alpha\rho^{\alpha-1} \frac{u^2}{r^2} + {2(\alpha-1)\rho^\alpha u u_x} + \alpha \rho^{1+\alpha} r^2 u_x^2 \right] dx\\
=&-\kappa\int^{1}_{0} \frac {x}{r }udx+\frac{\kappa}{2}\int^{1}_{0}urdx.\\
\end{aligned}\end{equation}
Note that 
\begin{equation}
   \int_0^1 \left[ \alpha \rho^{\alpha-1}\frac{u^2}{r^2} + {2(\alpha-1)\rho^\alpha u u_x} + \alpha \rho^{1+\alpha} r^2 u_x^2 \right] dx\ge (2\alpha-1) \int_0^1 \left( \rho^{\alpha-1}\frac{u^2}{r^2} + \rho^{1+\alpha} r^2 u_x^2 \right) dx,
\end{equation}
where we have used the fact that $2\alpha - 1 > 0$ due to assumtion \eqref{n2laphasmaal1}.

By Young's inequality, we have 
\begin{equation}\begin{aligned}
    \label{u1_alpha<1}
 &\frac{d}{d\tau} \int_0^1 \left( \frac{1}{2}u^2 + \frac{1}{\gamma-1} \rho^{\gamma-1} \right) dx 
+ (2\alpha-1)\int_0^1 \Big(\rho^{\alpha-1}\frac{u^2}{r^2} +\rho^{1+\alpha}  r^2u_x^2\Big) dx\\
\leq &\frac{2\alpha-1}{2}\int_{0}^{1}\rho^{\alpha-1}\frac{u^2}{r^2}dx+C\int_{0}^{1}\rho^{1-\alpha}x^2dx+C\int_{0}^{1}\rho^{1-\alpha}r^4dx\\
\leq &\frac{2\alpha-1}{2}\int_{0}^{1}\rho^{\alpha-1}\frac{u^2}{r^2}dx+C\int_{0}^{1}\rho^{1-\alpha}dx\\
\leq &\frac{2\alpha-1}{2}\int_{0}^{1}\rho^{\alpha-1}\frac{u^2}{r^2}dx+C\int_{0}^{1}\rho^{\gamma-1}dx+C,\\
\end{aligned}\end{equation}
where we have used the assumption that $\gamma>2-\alpha.$
Using Gronwall's inequality, we complete the proof of the Lemma \ref{2d-u2}.
\end{proof}
Handling the Poisson term requires the relationship between $x$ and $r$. We need to establish this relationship based on Lemma~\ref{2d-u2} for further estimation.
\begin{lema}   
Under the assumption of \eqref{n2laphasmaal1}, there exists a constant $C(T)>0$ such that
\begin{equation}\label{xrRelation1_alpha<1}
x \leq C(T)\, r(x,\tau)^{\frac{2(\gamma-1)}{\gamma}}, \quad (x,\tau) \in [0,1] \times [0,T].
\end{equation}
\end{lema} 
\begin{proof}
For any $x\in[0,1]$,  together with \eqref{basic2_alpha<1}
 we have,
\begin{equation}
\begin{aligned}
x = \int_{0}^{r(x,\tau)} \rho s\, ds \leq& C\left( \int_{0}^{r(x,\tau)} \rho^\gamma s\, ds \right)^{\frac{1}{\gamma}} \left( \int_{0}^{r(x,\tau)} s \,ds \right)^{\frac{\gamma-1}{\gamma}}\\
\leq& C(T) r^{\frac{2(\gamma-1)}{\gamma}}.
\end{aligned}\end{equation}
This completes the proof of the lemma.
\end{proof}
Next, the $r$-weighted $L^\infty-$ estimate of the density is controlled by the velocity $u$ and the effective velocity $w$. With the BD entropy estimate at hand, this further yields an $r$-weighted $L^\infty$ estimate of the density. Moreover, such a weighted estimate is also crucial for obtaining the upper and lower bounds of the density.
\begin{lema}
For any sufficiently small $\xi > 0$, there exists a positive constant $C$ independent of $T$ such that
\begin{equation}\label{weighted2_alpha<1}
\sup_{0\leq t\leq T}\left\|\rho^{\alpha-\frac{1}{2}}r^{\xi}\right\|_{L^\infty(0,R)}\leq C(1+\left\|\nabla\rho^{\alpha-\frac{1}{2}}\right\|_{L^2(\Omega)}).
\end{equation}
\end{lema}
\begin{proof}
Combining the one-dimensional Sobolev embedding inequality with H\"{o}lder's inequality, \eqref{basic2_alpha<1}, we obtain
\begin{equation}
\begin{aligned}
\left\| \rho^{\alpha-\frac{1}{2}}r^{\xi}\right\|_{L^{\infty}(0,R)} 
&\leq C\int_{0}^{R}\rho^{\alpha-\frac{1}{2}}r^{\xi}\,dr + C\int_{0}^{R}\left|\partial_{r}(\rho^{\alpha-\frac{1}{2}}r^{\xi})\right|\,dr \\
&\leq C\int_{0}^{R}\rho^{\alpha-\frac{1}{2}}r^{\xi}\,dr + C\int_{0}^{R}\left|\partial_{r}\rho^{\alpha-\frac{1}{2}}\right|r^{\xi}\,dr+C\int_{0}^{R}\rho^{\alpha-\frac{1}{2}}r^{\xi-1}\,dr \\
&\leq C\Big(\int_{0}^{R}\rho rdr\Big)^{\alpha-\frac{1}{2}}\Big(\int_{0}^{R}r^{\frac{2\xi-2\alpha+1}{3-2\alpha}}dr\Big)^{\frac{3}{2}-\alpha} \\
&\quad+ C\left(\int_{0}^{R}|\partial_{r} \rho^{\alpha-\frac{1}{2}}|^{2}r\,dr\right)^{\frac{1}{2}}\left(\int_{0}^{R}r^{2\xi -1}\,dr\right)^{\frac{1}{2}} \\
&\quad +C\left(\int_{0}^{R}\rho^{(\alpha-\frac{1}{2})\frac{4}{\xi}}r\,dr\right)^{\frac{\xi}{4}}\left(\int_{0}^{R}r^{\frac{3\xi-4}{4-\xi}}\,dr\right)^{\frac{4-\xi}{4}} \\
&\leq C\left(1 + \left\| \nabla \rho^{\alpha-\frac{1}{2}}\right\|_{L^{2}(\Omega)}+\left\| \rho^{\alpha-\frac{1}{2}}\right\|_{L^{2}(\Omega)}\right)\\
&\leq C\left(1 + \left\| \nabla \rho^{\alpha-\frac{1}{2}}\right\|_{L^{2}(\Omega)}\right),\\
\end{aligned}
\end{equation}
where we have used the fact that $\gamma>\alpha-\frac{1}{2}.$
Thus the proof of the lemma is completed.
\end{proof}
In the following theorem, we obtain an $L^2$ estimate of the effective velocity $w$, which in turn yields an $r$-weighted $L^\infty$ estimate of the density near the origin.
\begin{lema}
Under the assumption of \eqref{n2laphasmaal1}, there exists a constant $C(T)>0$ such that
    \begin{equation}\label{2w2_alpha<1}
\sup_{0\leq t\leq T}\int^{1}_{0}w^{2}dx+\int^{T}_{0}\int^{1}_{0}\left( \left( \rho^{{\frac{\gamma+\alpha}{2}} }\right)_{x}r \right)^{2}dxdt\leq C(T).
    \end{equation}
\end{lema}
\begin{proof}
    
Multiplying \eqref{bd2_alpha<1} by $w$ and integrating the resulting equation over $[0, 1]$, followed by an integration by parts, yields
\begin{equation}
\begin{aligned}\label{2dwbdestimate}
\frac{d}{d\tau} \left(\frac{1}{2}\int_0^1 w^2 dx + \frac{1}{\gamma - 1} \int_0^1 \rho^{\gamma - 1} dx \right) 
&+ \frac{4\alpha\gamma}{(\gamma + \alpha)^2} \int_0^1 \left( \bigl(\rho^{\frac{\gamma + \alpha}{2}}\bigr)_x r \right)^2 dx \\
&= -\kappa \int_0^1 \frac{x}{r} w \, dx + \frac{\kappa}{2} \int_0^1 r w \, dx \\
&\triangleq I_1 + I_2.
\end{aligned}
\end{equation}
Note that $$\begin{aligned}\label{rhoxuw}
\left\lVert \nabla \rho^{\alpha- \frac{1}{2}} \right\rVert _{L^{2}(\Omega)}^{2}&= 2\pi \int^{R}_{0}\left| \partial_{r}\left( \rho^{\alpha- \frac{1}{2}} \right) \right| ^{2}{r}dr\\
&= 2\pi \left( \frac{\alpha - \frac{1}{2}}{\alpha} \right)^2\int_{0}^{1}(w-u)^{2}dx\\
&\le C\left( \|w\|_{L^2(0,1)}^2 + \|u\|_{L^2(0,1)}^2 \right).
\end{aligned}$$
By H\"{o}lder's inequality, \eqref{rhoxuw}, \eqref{xrRelation1_alpha<1} and \eqref{weighted2_alpha<1}, we obtain
\begin{equation}
\begin{aligned}\label{w1_alpha<10}
I_1\leq& C\int_{0}^{1}w^2dx+C\int_{0}^{1}\frac{x^2}{r^2}dx\\
\leq&C\int_{0}^{1}w^2dx+C\int_{0}^{1}r^{2-\frac{4}{\gamma}}dx\\
\leq&C\int_{0}^{1}w^2dx+C\int_{0}^{R}\rho r^{3-\frac{4}{\gamma}}dr\\
\leq&C\int_{0}^{1}w^2dx+C\left\|\rho^{\alpha-\frac{1}{2}}r^{\xi}\right\|^{2}_{L^\infty(0,R)}\int_{0}^{R}\rho^{2-2\alpha} r^{3-\frac{4}{\gamma}-2\xi}dr\\
\leq&C\int_{0}^{1}w^2dx\\
&+C\Big(1+\left\|\nabla\rho^{\alpha-\frac{1}{2}}\right\|^2_{L^2(\Omega)}\Big)\Big(\int_{0}^{R}\rho^{\gamma}rdr\Big)^{\frac{2-2\alpha}{\gamma}}\Big(\int_{0}^{R}r^{(3-\frac{4}{\gamma}-2\xi-\frac{2-2\alpha}{\gamma})(\frac{\gamma}{\gamma-2+2\alpha})}dr\Big)^{\frac{\gamma-2+2\alpha}{\gamma}}\\
\leq&C\Big(1+\left\|u\right\|^2_{L^2(0,1)}+\left\|w\right\|^2_{L^2(0,1)}\Big),
\end{aligned}
\end{equation}
where we take $\xi>0$ to be sufficiently small such that
$$\xi<\frac{4(\gamma+\alpha)-8}{2\gamma},$$ 
provided that
$\gamma>2-\alpha.$

Similarly, we obtain
\begin{equation}
\begin{aligned}\label{w2_alpha<1}
I_2\leq& C\|w\|^2_{L^2(0,1)}+C\int_{0}^{1}r^2dx\\
\leq&C\|w\|^2_{L^2(0,1)}+C.\\
\end{aligned}
\end{equation}
Combining \eqref{2dwbdestimate} and \eqref{w1_alpha<10}-\eqref{w2_alpha<1}, we obtain

\begin{equation}
\begin{aligned}
\frac{d}{d\tau}\int_0^1\left(\frac{1}{2} w^2 + \frac{\rho^{\gamma - 1}}{\gamma - 1} \right)dx 
+  \frac{4\alpha\gamma}{(\gamma + \alpha)^2} \int_0^1 \left( \bigl(\rho^{\frac{\gamma + \alpha}{2}}\bigr)_x r \right)^2 dx \\
\leq C\left(1 + \| u\|_{L^2(0,1)}^2 + \| w\|_{L^2(0,1)}^2\right).
\end{aligned}
\end{equation}
By applying Gronwall's inequality and \eqref{basic2_alpha<1}, we finally prove equation \eqref{2w2_alpha<1}.
\end{proof}
\begin{corl}
    For any sufficiently small $\xi > 0$, there exists a positive constant $C$ independent of $T$ such that
\begin{equation}\label{weighted2_alpha<1}
\sup_{0\leq t\leq T}\left\|\rho^{\alpha-\frac{1}{2}}r^{\xi}\right\|_{L^\infty(0,R)}\leq C(T).
\end{equation}
\end{corl}
\begin{proof}
By \eqref{weighted2_alpha<1}, \eqref{rhoxuw}, \eqref{2w2_alpha<1} we obtain 
\begin{equation}\label{1weighted2_alpha<1}
\sup_{0\leq t\leq T}\left\|\rho^{\alpha-\frac{1}{2}}r^{\xi}\right\|_{L^\infty(0,R)}\leq C\Big(1 + \|w\|_{L^2}+ \|u\|_{L^2}\Big)\leq C(T).
\end{equation}
\end{proof}

  Given the estimates of the BD entropy and the basic energy, together with the definition of the effective velocity $w$, an improved relationship between $x$ and $r$ can be obtained.  \begin{lema}  
  Under the assumption of \eqref{n2laphasmaal1}. For any $1 \leq q < \infty$, there exists a constant $C(T,q) > 0$ such that
\begin{equation}\label{xr1_alpha<1}
x \leq C(T,q)\, r(x,\tau)^{\frac{2(q-1)}{q}}, \quad (x,\tau) \in [0,1] \times [0,T].
\end{equation}
\end{lema} 
\begin{proof}
Combining \eqref{basic2_alpha<1},\eqref{2w2_alpha<1} and Sobolev embedding inequality, we deduce that $\rho^{\alpha-\frac{1}{2}} \in W^{1,2}(\Omega) \hookrightarrow L^{q}(\Omega),$ for any $1 \leq q < \infty$. Thus for any $x\in[0,1]$, it holds that
\begin{equation}
\begin{aligned}
x &= \int_{0}^{r(x,\tau)} \rho s\, ds  \leq C\left( \int_{0}^{r(x,\tau)} \rho^q s\, ds \right)^{\frac{1}{q}} \left( \int_{0}^{r(x,\tau)} s \,ds \right)^{\frac{q-1}{q}}\\
&\leq C(T,q) r^{\frac{2(q-1)}{q}}.
\end{aligned}\end{equation}

\end{proof}
Based on the basic energy estimate and the BD-entropy estimate established above, we further exploit the dissipative structure of the system by deriving higher integrability estimates for the velocity field $u$. By carefully handling the source term, we transform the higher integrability estimates for $u$ into a weighted estimate for the density $\rho$.
\begin{lema}\label{2duhighnot}
Under the assumption of \eqref{n2laphasmaal1} and the condition $1< n< n_{2}(\alpha)$ with $n\in \mathbb{R}$ and for any $1 \leq q < \infty$, the following estimates hold 
\begin{equation}\begin{aligned}\label{2du4}
&{\sup_{0\leq t\leq T}\int_{0}^{1}u^{2n}dx + \int_{0}^{T}\int_{0}^{1}\Big(\rho^{\alpha-1}\frac{u^{2n}}{r^{2}} +\rho^{1+\alpha}r^{2}u_x^{2}u^{2n-2}\Big)dx dt} \\
\leq &C(T)+C\int_{0}^{T}\int_{0}^{1}\rho^{2n(\gamma-\alpha)+\alpha-1}r^{2n-2}dxdt+C(T)\int_{0}^{T}\int_{0}^{1}\rho^{(2n-1)(1-\alpha)}r^{6n-\frac{4n}{q}-2}dxdt.
\end{aligned}
\end{equation}

\end{lema}
\begin{proof}   
Multiplying \(\eqref{eq:NSP-lagrangian-rhoalpha}_2\) by \(u^{2n-1}\) and integrating over \((0,1)\) gives
\begin{equation}
\begin{aligned}
\frac{1}{2n}\frac{d}{d\tau}\int_0^1 u^{2n}dx +\int_0^1 (\rho^\gamma)_x ru^{2n-1}dx =&-\alpha \int_0^1 \rho^{1+\alpha} (ru)_x (ru^{2n-1})_xdx \\
&- \int_0^1 (\rho^\alpha)_x u^{2n}dx -\kappa\int^{1}_{0} \frac{x}{r}u^{2n-1}dx\\
&+ \frac{\kappa}{2}\int^{1}_{0}u^{2n-1}rdx.
\end{aligned}
\end{equation}
A direct computation leads to
\begin{equation}
\begin{aligned}\label{2d-uhighestimate}
&\frac{1}{2n}\frac{d}{d\tau}\int_0^1 u^{2n}dx + \int_0^1 \Big( \alpha\rho^{\alpha-1} \frac{u^{2n}}{r^2} + (2n-1)\alpha \rho^{1+\alpha} u_x^2 u^{2n-2} r^{2} + 2n(\alpha-1) \rho^\alpha u^{2n-1}  u_x \Big) dx \\
=& \int_0^1 \Big(  \rho^{\gamma-1} \frac{u^{2n-1}}{r} +(2n-1) \rho^\gamma u^{2n-2} u_x r \Big) dx -\kappa\int^{1}_{0} \frac{x}{r}u^{2n-1}dx + \frac{\kappa}{2}\int^{1}_{0}u^{2n-1}rdx.
\end{aligned}
\end{equation}
Applying Young's inequality, we obtain
\begin{equation}\label{ufix1}
\begin{aligned}
& 2n(\alpha-1) \rho^\alpha u^{2n-1}  u_x \\
\ge& -(1-2\varepsilon)(2n-1)\alpha \rho^{1+\alpha} u_x^2 u^{2n-2} r^{2} - \frac{\big(2n(\alpha-1)\big)^2}{4(1-2\varepsilon)(2n-1)\alpha} \rho^{\alpha-1} \frac{u^{2n}}{r^2}\\
\ge&-(1-2\varepsilon)\Big((2n-1)\alpha\rho^{1+\alpha}u_x^2u^{2n-2}r^2+\alpha\rho^{\alpha{-1}}\frac{u^{2n}}{r^2}\Big).
\end{aligned}
\end{equation}
Taking \(\varepsilon\) sufficiently small such that
\[
\frac{\big(2n(\alpha-1)\big)^2}{4(1-2\varepsilon)(2n-1)\alpha} <\alpha (1-2\varepsilon).
\]
Such an \(\varepsilon\) exists because \(n<n_2(\alpha)\). Substituting \eqref{ufix1} into \eqref{2d-uhighestimate} yields
\begin{equation}
\begin{aligned}\label{2duhighfirst}
&\frac{1}{2n}\frac{d}{d\tau}\int_0^1 u^{2n}dx +2\varepsilon \int_0^1 \Big(  \alpha\rho^{\alpha-1} \frac{u^{2n}}{r^2} + (2n-1)\alpha \rho^{1+\alpha} u_x^2 u^{2n-2} r^{2} \Big) dx \\
\le& C \int_0^1 \Big( (2n-1)\rho^\gamma u^{2n-2}  u_x r + \rho^{\gamma-1} \frac{u^{2n-1}}{r} \Big) dx-\kappa\int^{1}_{0} \frac{x}{r}u^{2n-1}dx+ \frac{\kappa}{2}\int^{1}_{0}u^{2n-1}rdx\\
\triangleq& J_1+J_2+J_3.
\end{aligned}
\end{equation}
Applying Young's inequality again, we obtain
\begin{equation}\begin{aligned}\label{3uhigh21}
J_{1}&\leq C\left( \int^{1}_{0}\rho^{\alpha-1} \frac{u^{2n}}{r^{2}}dx \right)^{\frac{2n-1}{2n}}\left( \int^{1}_{0}\rho^{2n(\gamma-\alpha)+\alpha-1}r^{2(n-1)}dx \right)^{\frac{1}{2n}}\\
&\quad+C\left( \int^{1}_{0} \rho^{1+\alpha}u^{2n-2}r^{2}u_{x}^{2}dx \right)^{\frac{1}{2}}\left( \int^{1}_{0} \rho^{\alpha-1}\frac{u^{2n}}{r^{2}}dx \right)^{\frac{n-1}{2n}}\left( \int^{1}_{0}\rho^{2n(\gamma-\alpha)+\alpha-1}r^{2(n-1)} dx\right)^{\frac{1}{2n}}\\
&\leq \frac{\alpha\varepsilon}{3}\int^{1}_{0}\rho^{\alpha-1} \frac{u^{2n}}{r^{2}}dx +C(\varepsilon)\int^{1}_{0}\rho^{2n(\gamma-\alpha)+\alpha-1}r^{2(n-1)}dx\\
&\quad+(2n-1)\alpha\varepsilon\int^{1}_{0}\rho^{1+\alpha}u^{2n-2}r^{2}u_{x}^{2}dx.
\end{aligned}\end{equation}
Using Young's inequality combined with \eqref{xr1_alpha<1}, we obtain
yields
\begin{equation}
\begin{aligned}\label{m2}
J_2 &\le \frac{\alpha\varepsilon}{3} \int_{0}^{1}\rho^{\alpha-1}\frac{u^{2n}}{r^2}dx + C\int_{0}^{1}\rho^{(2n-1)(1-\alpha)}x^{2n}r^{2n-2}dx \\
&\le \frac{\alpha\varepsilon}{3} \int_{0}^{1}\rho^{\alpha-1}\frac{u^{2n}}{r^2}dx + C(T)\int_{0}^{1}\rho^{(2n-1)(1-\alpha)}r^{6n-\frac{4n}{q}-2}dx. \\
\end{aligned}
\end{equation}
Using Young's inequality, we obtain
\begin{equation}
\begin{aligned}\label{M3}
J_3 &\le C\int_{0}^{1}u^{2n}dx+C.
\end{aligned}
\end{equation}
Plugging \eqref{3uhigh21}--\eqref{M3} into \eqref{2duhighfirst} and then applying Gronwall's inequality, we obtain  \eqref{2du4}.

\end{proof}
\begin{lema}
Under the assumption of \eqref{n2laphasmaal1} and the condition $1< m\in A_{\text{set}}$ and for any $1 \leq q < \infty$, the following estimates hold 
\begin{equation}\begin{aligned}\label{2uw4}
&\sup_{0\leq t\leq T}\int_{0}^{1}w^{2m}dx + \int_{0}^{T}\int_0^1 \rho^{\gamma-\alpha} \big(r (\rho^\alpha)_x\big)^{2m} dxdt\\ \leq& C+C\int_{0}^{T}\int_{0}^{1}\rho^{\gamma-\alpha}u^{2m}dxdt+C\int_{0}^{T}\int_{0}^{1}(r^{1-\frac{2}{q}})^{{2m}}dxdt +C\int_{0}^{T}\int^{1}_{0}r^{2m}dxdt.
\end{aligned}\end{equation}
\end{lema}
\begin{proof}
Multiplying \eqref{bd2_alpha<1} by $w^{2m-1}$ and integrating over \([0,1]\) gives
\begin{equation}
\begin{aligned}
&\frac{1}{2m}\frac{d}{d\tau}\int_0^1 w^{2m} dx + \int_0^1 \gamma \rho^{\gamma-1} \rho_x \big(u + r(\rho^\alpha)_x\big)^{2m-1} r dx \\
=&-\kappa\int^{1}_{0} \frac{x}{r}w^{2m-1}dx + { \frac{ \kappa}{2}\int^{1}_{0}w^{2m-1}rdx }.
\end{aligned}\end{equation}
Using Lemma \ref{mdense}, we have
\begin{equation}
\begin{aligned}
\gamma \rho^{\gamma-1}  \rho_x \big(u + r (\rho^\alpha)_x\big)^{2m-1} r &= r (\rho^\alpha)_x \big(u + r (\rho^\alpha)_x\big)^{2m-1} \frac{\gamma}{\alpha} \rho^{\gamma-\alpha} \\
&\ge \varepsilon \rho^{\gamma-\alpha} \big(r (\rho^\alpha)_x\big)^{2m} - C \rho^{\gamma-\alpha} u^{2m},
\end{aligned}\end{equation}
where \(\varepsilon>0\) is a small constant depending only on \(m,\gamma,\alpha\). Therefore, combining  Young's inequality and \eqref{xr1_alpha<1}, we arrive at
\begin{equation}
\begin{aligned} \label{2dufinna}
&\frac{1}{2m}\frac{d}{d\tau}\int_0^1 w^{2m} dx + \varepsilon \int_0^1 \rho^{\gamma-\alpha} \big(r (\rho^\alpha)_x\big)^{2m} dx \\\le& C \int_0^1 \rho^{\gamma-\alpha} u^{2m} dx-\kappa\int_{0}^{1}\frac{x}{r}w^{2m-1}dx+\frac{\kappa}{2}\int_{0}^{1}w^{2m-1}rdx\\
\leq&C\int_{0}^{1}w^{2m}dx +C\int_0^1 \rho^{\gamma-\alpha} u^{2m} dx+ C\int_{0}^{1}r^{2m(1-\frac{2}{q})}dx+ C \int_0^1 r^{2m}dx.\\
\end{aligned}\end{equation}
Combining \eqref{2dufinna} with Gronwall's inequality, we obtain \eqref{2uw4}.

\end{proof}
\begin{lema}\label{lem: def k}
Let \(N=2\). For any \((\alpha,\gamma)\) satisfying \eqref{n2laphasmaal1}, there exists an \(\alpha\)-dependent constant \(k\in A_{\text{set}}\) with \(k>1\) such that
\begin{equation}\label{dense2d}
1-\frac{k\sqrt{2k-1}-2k+1}{k^2-2k+1}<\alpha<1-\frac{1}{2k}. 
\end{equation}
\end{lema}

\begin{proof}
To obtain the desired \(k\), we first consider the auxiliary function
\begin{equation}\label{auxiliary}
h(a)=\frac{a\sqrt{2a-1}-2a+1}{(a-1)^2}-\frac{1}{2a},\qquad a>1.
\end{equation}
A direct computation shows that \(h(a)>0\) for all \(a>1\).

Since \(\alpha\) satisfies condition \eqref{n2laphasmaal1}, we set
\[
k_0:=\frac{1}{2(1-\alpha)}.
\]
By virtue of \eqref{auxiliary}, we have $h(k_0)>0$, which implies
\[
1-\frac{k_0\sqrt{2k_0-1}-2k_0+1}{k_0^2-2k_0+1}<1-\frac{1}{2k_0}=\alpha.
\]

Moreover, the set \(A_{\text{set}}\) is dense in \((1,\infty)\). Hence we can choose \(k\in A_{\text{set}}\) with \(k>k_0\) sufficiently close to \(k_0\) such that \eqref{dense2d} hold. \end{proof}

Next, we shall establish the upper and lower bounds of the density. The key step is to obtain the high integrability estimate of the density gradient, which is derived from the high integrability of $u$ and the effective velocity $w$. Meanwhile, the high integrability of $u$ and $w$ can be achieved by the $r$-weighted estimate of the density.
\begin{lema}\label{rhouw2dhigh}
Under the assumption of \eqref{n2laphasmaal1}, and let $k$ be defined as in Lemma \ref{lem: def k}. The following estimates hold
\begin{equation}\label{2drhouhigh}
\sup_{0\leq t\leq T}\int_{0}^{R}\rho u^{2k}rdr + \int_{0}^{T}\int_{0}^{R}\Big(\rho^{\alpha}\frac{u^{2k}}{r} + \rho^{\alpha}(u_r)^{2}u^{2k-2}r\Big)drdt \leq C(T),
\end{equation}
and
\begin{equation}\label{w-uhigh12}
\sup_{0\leq t\leq T}\int_{0}^{R}\Big|(\rho^{\alpha -1 + \frac{1}{2k}})_r\Big|^{2k}rdr + \int_{0}^{T}\int_{0}^{R}\Big|(\rho^{\alpha -1 + \frac{\gamma - \alpha + 1}{2k}})_r\Big|^{2k}rdrdt \leq C(T).
\end{equation}
\end{lema}
\begin{proof}
Using Proposition \ref{2duhighnot} and \eqref{weighted2_alpha<1}, and rewriting the estimates in Eulerian coordinates, we obtain that
\[
\begin{aligned}
&\sup_{0\leq t\leq T}\int_{0}^{R}\rho u^{2k}r\,dr + \int_{0}^{T}\int_{0}^{R}\left(\rho^{\alpha}\frac{u^{2k}}{r} +\rho^{\alpha}u_r^{2}u^{2k-2}r\right)drdt \\
\leq& C + C\int_{0}^{T}\int_{0}^{R}\rho^{2k(\gamma-\alpha)+\alpha}r^{2k-1}drdt \\
&+C(T)\int_{0}^{T}\int_{0}^{R}\rho^{(2k-1)(1-\alpha)+1}r^{6k-\frac{4k}{q}-1}drdt \\
\leq& C + C\sup_{0\leq t\leq T}\left\| \rho^{\alpha-\frac{1}{2}}r^{\xi}\right\|_{L^{\infty}(0,R)}^{\frac{2k(\gamma-\alpha)+\alpha}{\alpha-1/2}}
\int_{0}^{T}\int_{0}^{R}r^{2k-1-\frac{2k(\gamma-\alpha)+\alpha}{\alpha-1/2}\xi}drdt \\
&+C(T)\sup_{0\leq t\leq T}\left\|\rho^{\alpha-\frac{1}{2}}r^{\xi}\right\|^{\frac{(2k-1)(1-\alpha)+1}{\alpha-\frac{1}{2}}}_{L^\infty(0,R)}\int_{0}^{T}\int_{0}^{R} r^{6k-\frac{4k}{q}-1-\xi\frac{(2k-1)(1-\alpha)+1}{\alpha-\frac{1}{2}}}drdt \\
\leq& C(T),
\end{aligned}
\]
where $\xi>0$ is chosen sufficiently small so that
\begin{equation}
\begin{cases}
2k-1-\dfrac{2k(\gamma-\alpha)+\alpha}{\alpha-\frac{1}{2}}\,\xi > -1, \\[1.2em]
6k-\frac{4k}{q}-1-\xi\frac{(2k-1)(1-\alpha)+1}{\alpha-\frac{1}{2}}>-1.
\end{cases}
\end{equation}
Such $\xi$ exists provided $q$ is sufficiently large and \eqref{n2laphasmaal1} holds.

We next prove \eqref{w-uhigh12}. From \eqref{2uw4}, coupled with \eqref{weighted2_alpha<1}, \eqref{2drhouhigh}, and the definition of the effective velocity, we obtain
\begin{equation}
\begin{aligned}
&\sup_{0\leq t\leq T}\int_{0}^{R}\Big|(\rho^{\alpha-1+\frac{1}{2k}})_r\Big|^{2k}r\,dr + \int_{0}^{T}\int_{0}^{R}\Big|(\rho^{\alpha-1+\frac{\gamma-\alpha+1}{2k}})_r\Big|^{2k}r\,drdt \\
&\leq C\sup_{0\leq t\leq T}\int_{0}^{R}\rho u^{2k}r\,dr + C\sup_{0\leq t\leq T}\int_{0}^{R}\rho w^{2k}r\,dr \\
&\quad + C\int_{0}^{T}\int_{0}^{R}\Big|(\rho^{\alpha-1+\frac{\gamma-\alpha+1}{2k}})_r\Big|^{2k}r\,drdt\\
&\leq C(T) + C\int_{0}^{T}\int_{0}^{R}\rho^{\gamma-\alpha+1}u^{2k}r drdt +C\int_{0}^{T}\int_{0}^{R}\rho r^{2k\left(1-\frac{2}{q}\right)+1}drdt \\
&\quad+C\int_{0}^{T}\int^{R}_{0}\rho r^{2k+1}drdt\\
&\leq C(T) + C\int_{0}^{T}\int_{0}^{R}\rho^{\gamma-2\alpha+1}r^{2}\rho^{\alpha}\frac{u^{2k}}{r}drdt+C\sup_{0\leq t\leq T}\left\|\rho^{\alpha-\frac{1}{2}}r^{\xi}\right\|^{\frac{1}{\alpha-\frac{1}{2}}}_{L^\infty}\int_{0}^{T}\int_{0}^{R}r^{2k\left(1-\frac{2}{q}\right)+1-\frac{\xi}{\alpha-\frac{1}{2}}}drdt \\
&\quad+C\sup_{0\leq t\leq T}\left\|\rho^{\alpha-\frac{1}{2}}r^{\xi}\right\|^{\frac{1}{\alpha-\frac{1}{2}}}_{L^\infty(0,R)}\int_{0}^{T}\int_{0}^{R}r^{2k+1-\frac{\xi}{\alpha-\frac{1}{2}}}drdt\\
&\leq C(T) + C\sup_{0\leq t\leq T}\left\| \rho^{\alpha-\frac{1}{2}}r^{\xi}\right\|_{L^{\infty(0,R)}(0,R)}^{\frac{\gamma-2\alpha+1}{\alpha-\frac{1}{2}}}
\int_{0}^{T}\int_{0}^{R}\rho^{\alpha}\frac{u^{2k}}{r}drdt \\
&\leq C(T),
\end{aligned}
\end{equation}
where we used the fact that
\[
\gamma-2\alpha+1 > 0,
\]
and $\xi$ is taken sufficiently small so that
\begin{equation}
\begin{cases}
\left(\dfrac{\gamma-2\alpha+1}{\alpha-\frac{1}{2}}\right)\xi < 2, \\[1.2em]
2k\left(1-\frac{2}{q}\right)+1-\frac{\xi}{\alpha-\frac{1}{2}}>-1,\\
2k+1-\frac{\xi}{\alpha-\frac{1}{2}}>-1.
\end{cases}
\end{equation}
The existence of such $\xi$ is guaranteed by taking $q$ sufficiently large and by \eqref{n2laphasmaal1}.
This finishes the proof of Proposition \ref{rhouw2dhigh}.
\end{proof}
Let
\begin{equation}
R_T := \sup_{0 \le t \le T} \| \rho(t) \|_{L^\infty(\Omega)} + 1, \qquad
V_T := \sup_{0 \le t \le T} \| \rho^{-1}(t) \|_{L^\infty(\Omega)} + 1.
\end{equation}
\begin{prop}
 Under the assumption of \eqref{n2laphasmaal1}, there exists a constant $C(T)>0$ such that
    \begin{equation}
        R_T\leq C(T).
    \end{equation}
\end{prop}
\begin{proof}

By the Sobolev embedding theorem in one dimension, we have
\begin{equation}
    \begin{aligned}
\left\lVert \rho^{\alpha} \right\rVert _{L^{\infty}(0,R)}&\leq \int^{R}_{0}\rho^{\alpha}dr+\int^{R}_{0}|\partial_{r}\rho^{\alpha}|dr:=G_{1}+G_{2}.
\end{aligned}
\end{equation}
Using  H\"older's inequality yields
\begin{equation}\begin{aligned}\label{rhobeta12}
G_{1}&= \int^{R}_{0}\rho^{\alpha}dr=\int^{R}_{0}\rho^{\alpha}r^{\frac{1}{3}}r^{- \frac{1}{3}}dr\\
&\leq\left( \int^{R}_{0}\rho^{3\alpha}rdr \right)^{\frac{1}{3}}\left( \int^{R}_{0}r^{- \frac{1}{2}}dr \right)^{\frac{2}{3}}\\
&\leq C(T).
\end{aligned}\end{equation}
H\"older's inequality and \eqref{w-uhigh12} gives that
\begin{equation}\begin{aligned}\label{rhobeta213}
G_{2}&= \int^{R}_{0} |C\rho^{\alpha-1}\rho_{r}|dr\leq \alpha \int^{R}_{0}|\rho^{\alpha-1}\rho^{\frac{1}{2k}-1}\rho_{r}\rho^{1- \frac{1}{2k}}| dr\\
&\leq C(T)\Big(\int_{0}^{R}|\partial_r\rho^{\alpha-1+\frac{1}{2k}}|^{2k}rdr\Big)^{\frac{1}{2k}}\Big(\int_{0}^{R}\rho r^{-\frac{1}{2k-1}}dr\Big)^{\frac{2k-1}{2k}}\\
&\leq C(T)\Big(\int_{0}^{R}\rho^{\frac{2k}{k-1}}rdr\Big)^{\frac{k-1}{2k}}\Big(\int_{0}^{R}r^{-\frac{2k^2-k+1}{2k^2+k-1}}dr\Big)^{\frac{k+1}{2k}}.
\end{aligned}\end{equation} 
We deduce from \eqref{rhobeta12} and \eqref{rhobeta213} that
\begin{equation}
    R_{T}^{\alpha}\leq C(T).
\end{equation}
\end{proof}
\begin{prop}\label{2dlowerdensity}
Under the assumption of \eqref{n2laphasmaal1}, there exists a constant $C(T)>0$ such that     
\begin{equation}
    V_T\leq C(T).
\end{equation}
\end{prop}
\begin{proof}
Let \(v(x,\tau) = \frac{1}{\rho(x,\tau)}\). From the continuity equation $\eqref{eq:NSP-lagrangian-rhoalpha}_1$ we deduce \(v_\tau = (r u)_x\). Integrating this relation over \([0,1] \times [0,\tau]\)  leads to
\begin{equation}\label{v01}
\int_0^1 v(x,\tau) \, dx = \int_0^1 v_0(x) \, dx \le C.
\end{equation}

Now choose a sufficiently small \(\xi\) satisfying \(0 < \xi < \frac{k-1}{2k}\). We employ estimates \eqref{w-uhigh12} and \eqref{v01}, together with the one-dimensional Sobolev embedding inequality, to bound \(v(x,\tau)\) as follows:
\[
\begin{aligned}
v(x,\tau) &\le \int_0^1 v \, dx + \int_0^1 | v_x| \, dx \\
&\le C + C \int_0^1 | (\rho^\alpha)_x| \, v^{\alpha+1} \, dx\\
&\le C+\left( \int_0^1 | (\rho^\alpha)_x|^{2k} r^{2k} \, dx \right)^{\frac{1}{2k}}
C\left( \int_0^1 v \, r^{-\frac{2}{2\xi+1}} \, dx \right)^{\frac{1+2\xi}{2}} \\
&\quad \times \left( \int_0^1 v^{(\alpha+\frac12-\xi)\frac{2k}{k-1-2k\xi}} \, dx \right)^{\frac{k-1-2k\xi}{2k}}.
\end{aligned}
\]
We obtain
\begin{equation}\label{vlow}
v(x,\tau) \le C + C(T) \, V_T^{\alpha + \frac{1}{2k}}.
\end{equation}
Finally, the desired estimate \(\eqref{vlow}\) is obtained by taking the supremum over \((x,\tau) \in [0,1] \times [0,T]\) and applying Young's inequality together with \eqref{dense2d}, completing the proof of Proposition \ref{2dlowerdensity}.

\end{proof}

\subsection{A priori estimates for 
$N=3$: lower-order estimates}
\indent  Compared with the two-dimensional case, the structure of the Poisson term changes in three dimensions, and the weight in the $r$-weighted estimate of the density also differs. This leads to more difficulties in the three-dimensional problem, and the two-dimensional methods cannot be directly extended.

We now derive the basic energy estimates. Due to the presence of the Poisson term, we consider two separate cases: $\kappa = -1$ and $\kappa = 1$. Accordingly, the admissible range of $\gamma$ depends on which case is under consideration.
\begin{lema}Under the assumption of \eqref{plasmas3dalphagamma}, consider either of the following cases:
\begin{itemize}
    \item $\kappa = -1$ (plasmas) with $\gamma > 1$;
    \item $\kappa = 1$ (gaseous stars) with $\gamma > \frac{4}{3}$.
\end{itemize}
Then there exists a positive constant $C(T)$ depending on $T$ such that 
   \begin{equation}\label{b100-twocase}
       \sup_{0\leq t\leq T}\int_{0}^{1} \left( u^{2}+\rho^{\gamma-1}  \right) dx + \int_{0}^{T} \int_{0}^{1} \left( \rho^{\alpha-1}\frac{u^{2}}{r^{2}} +\rho^{1+\alpha}(r^2u_x)^2 \right) dx dt \le C (T).
   \end{equation} 
\end{lema}
\begin{proof}
We multiply both sides of \eqref{eq:NSP-lagrangian-rhoalpha} by $u$ and then integrate by parts over $(0,1)\times(0,\tau)$. This yields
   \begin{equation}
\begin{aligned}\label{uinte1}
\frac{d}{d\tau} \left( \frac{1}{2} \int^{1}_{0}u^{2}dx + \frac{1}{\gamma-1} \int^{1}_{0} \rho^{\gamma-1}dx \right)
&+ \int^{1}_{0}\Big( (4\alpha-2)\rho^{\alpha-1}\frac{u^{2}}{r^{2}} + \alpha\rho^{1+\alpha}(r^{2}u_{x})^{2} \Big) dx \\
&- (4-4\alpha) \int^{1}_{0}\rho^{\alpha} u u_x r \, dx \\
&= -\kappa \int^{1}_{0} \frac{x}{r^{2}}u \, dx + \frac{\kappa}{3} \int^{1}_{0} u r \, dx \\
&\triangleq H_{1}+H_{2}.
\end{aligned}
\end{equation}
Note that
\begin{equation}
        \begin{aligned}
&(4\alpha-2)\int_{0}^{1}\rho^{\alpha-1}\frac{u^{2}}{r^{2}}dx+ \alpha\int_{0}^{1}\rho^{1+\alpha}(r^{2}u_{x})^{2}dx-(4-4\alpha) \int_{0}^{1}\rho^{\alpha}uu_xrdx\\\geq& (3\alpha-2)\int_{0}^{1}\rho^{\alpha-1}\frac{u^{2}}{r^{2}}dx+(2\alpha-1)\int_{0}^{1}\rho^{1+\alpha}(r^2u_x)^2dx\\
\geq&(3\alpha-2)\int_{0}^{1}\Big(\rho^{\alpha-1}\frac{u^{2}}{r^{2}}+\rho^{1+\alpha}(r^2u_x)^2\Big)dx,
\end{aligned}
    \end{equation}
where we have used the fact that $\alpha >\frac{5}{6}$.
Noting that $r_\tau = u$, we deduce that
    \begin{equation}
        \begin{aligned}\label{xtau-11}
H_{1}&= -\kappa \int^{1}_{0} \frac{x}{r^{2}}udx=-\kappa \int^{1}_{0} \frac{x}{r^{2}}r_{\tau}dx= \kappa \frac{d}{d\tau}\int^{1}_{0} \frac{x}{r}dx.
\end{aligned}
    \end{equation}
An application of Young's inequality gives    
     \begin{equation}
        \begin{aligned}\label{xtau000}
        H_2\leq C\int_{0}^{1}u^2dx+C\int_{0}^{1}r^2dx.
        \end{aligned}
    \end{equation}
\noindent\textbf{Case 1: $\kappa = -1$ (plasmas)}

In the case $\kappa = -1$, combining \eqref{xtau-11} and \eqref{xtau000} yields
    \begin{equation}
        \begin{aligned}
\frac{d}{d\tau} \left( \frac{1}{2} \int^{1}_{0}u^{2}dx+ \frac{1}{\gamma-1} \int^{1}_{0} \rho^{\gamma-1}dx + \int^{1}_{0} \frac{x}{r}dx\right)&+(3\alpha-2)\int^{1}_{0}\left(\rho^{\alpha-1} \frac{u^{2}}{r^{2}}+ \rho^{1+\alpha}(r^{2}u_{x})^{2} \right)dx\\
&\leq C\int_{0}^{1}u^2dx+C\int^{1}_{0}r^{{2}}dx.\\
\end{aligned}
    \end{equation}
Applying Gronwall's inequality to the above inequality yields \eqref{b100-twocase}.

\noindent\textbf{Case 2: $\kappa = 1$ (gaseous stars)}

When $\kappa =1$, the problem becomes more involved. Combining \eqref{uinte1}--\eqref{xtau000} yields
    \begin{equation}
        \begin{aligned}
&\frac{d}{d\tau} \left( \frac{1}{2} \int^{1}_{0}u^{2}dx+ \frac{1}{\gamma-1} \int^{1}_{0} \rho^{\gamma-1}dx \right)+(3\alpha-2)\int^{1}_{0}\left( \rho^{\alpha-1}\frac{u^{2}}{r^{2}}+ \rho^{1+\alpha}(r^{2}u_{x})^{2} \right)dx\\
\le & \frac{d}{d\tau}\int^{1}_{0} \frac{x}{r}dx+C\int^{1}_{0}u^2dx+C\int^{1}_{0}r^2dx.
\end{aligned}
    \end{equation}   
The following terms require careful treatment. Using the fact that \begin{equation}
    \frac{1}{r^2}(r^2\Phi_r)_r=\rho-1\implies x(r)=\int^r_0
\left[(s^2\Phi_s)_s+s^2\right]ds=r^2\Phi_r+ \frac{r^3}{3},\end{equation} we obtain
\begin{align*}
&\int^{1}_{0} \frac{x}{r}dx= \int^{R}_{0} \frac{x}{r}\rho r^{2}dr=\int^{R}_{0}\rho rxdr\\
&= \int^{R}_{0}r^{3}\left(  1+ \frac{2\Phi_{r}}{r}+ \Phi_{rr} \right) \left( \Phi_{r}+ \frac{r}{3} \right)dr\\
&= \int^{R}_{0}\left( \Phi_{r}r^{3}+ \frac{r^{4}}{3} + \Phi_{rr}\Phi_{r}r^{3}+ \frac{\Phi_{rr}r^{4}}{3}+2\left| \Phi_{r} \right| ^{2}r^{2}+ \frac{2\Phi_{r}r^{3}}{3}\right)dr\\
&= \frac{1}{2}\int^{R}_{0} (\Phi_{r}^{2})_{r}r^{3}dr+ 2\int^{R}_{0} r^{2}\left| \Phi_{r} \right| ^{2}dr+ \frac{1}{3} \int^{R}_{0} \Phi_{rr}r^{4}dr+ \frac{5}{3}\int^{R}_{0} r^{3}\Phi_{r}dr+ \frac{1}{3}\int^{R}_{0}r^{4}dr\\
&=\frac{1}{2}\int^{R}_{0} r^{2}\left| \Phi_{r} \right| ^{2}dr+\frac{1}{3}\int^{R}_{0} r^{3} \Phi_{r} dr+\frac{1}{15}R^5\\
&\leq C\int^{R}_{0} r^{2}\left| \Phi_{r} \right| ^{2}dr+C\\
&\leq C\|\nabla\Phi\|^2_{L^2(\Omega)}+C.
\end{align*}
Next, we need to carefully deal with the term $\|\nabla\Phi\|_{L^2(\Omega)}^2$.
We have that
\begin{equation}
        \begin{aligned}
\int^{R}_{0}|\Phi_{r}|^{2}r^{2}dr&= \frac{1}{4\pi}\int_{\Omega}|\nabla \Phi|^{2}d\mathbf{x}= {\frac{1}{4\pi}}\left\lVert \nabla \Phi \right\rVert ^{2}_{L^{2}(\Omega)}.
\end{aligned}
    \end{equation}
We multiply \eqref{eq:NSP-original} by $\Phi$ and integrate by parts to get
    \begin{equation}
        \begin{aligned}
\int_{\Omega}\Delta \Phi \cdot \Phi d\mathbf{x}=\int_{\Omega}(\rho-1)\Phi d\mathbf{x},
\end{aligned}
    \end{equation}
thus we have
    \begin{equation}\label{est of nabla Phi}
        \begin{aligned}
\left\lVert \nabla \Phi\right\rVert ^{2}_{L^{2}}&= -\int_{\Omega}(\rho-1)\Phi d\mathbf{x}\leq C\left\lVert  \rho-1\right\rVert _{L^{\frac{6}{5}}(\Omega)}\left\lVert \Phi \right\rVert _{L^{6}(\Omega)}\\
&\leq C\left\lVert \rho- 1\right\rVert _{L^{\frac{6}{5}}(\Omega)}\left\lVert \nabla \Phi \right\rVert _{L^{2}(\Omega)}.
\end{aligned}
    \end{equation}
It then follows that
\begin{equation}\label{j11}
        \begin{aligned}
\int^{1}_{0} \frac{x}{r}dx\leq& C  \left\lVert \nabla \Phi\right\rVert ^{2}_{L^{2}}+C\\
\leq&C\left\lVert \rho- 1\right\rVert^2 _{L^{\frac{6}{5}}(\Omega)}+C,\\
\leq&C(1+\|\rho\|^{\frac{5\gamma-6}{3(\gamma-1)}}_{L^1(\Omega)}\|\rho\|^{\frac{\gamma}{3(\gamma-1)}}_{L^{\gamma}(\Omega)})\\
\leq&\frac{1}{2(\gamma-1)}\|\rho\|^{\gamma}_{L^{\gamma}(\Omega)}+C,
        \end{aligned}
    \end{equation}
where we used \eqref{plasmas3dalphagamma}.
Combining \eqref{uinte1}, \eqref{j11}, \eqref{xtau000} and integrating over $(0,\tau)$ yields
 \begin{equation}
        \begin{aligned}\label{qq1}
&\frac{1}{2} \int^{1}_{0}u^{2}dx+ \frac{1}{2(\gamma-1)} \int^{1}_{0} \rho^{\gamma-1}dx +(3\alpha-2)\int^{\tau}_{0}\int^{1}_{0}\left( \rho^{1+\alpha}\frac{u^{2}}{r^{2}}+ \rho^{1+\alpha}(r^{2}u_{x})^{2} \right)dxds\\
\leq & C\int^{\tau}_{0}\int^{1}_{0} u^2dxds+ C\int^{\tau}_{0}\int^{1}_{0}r^2 dxds+\frac{1}{2} \int^{1}_{0}u_{0}^{2}dx+ \frac{1}{2(\gamma-1)}\int^{1}_{0}\rho_{0}^{\gamma-1}dx.\\
\end{aligned}
    \end{equation}
    Finally, we apply Gronwall's inequality to \eqref{qq1} and obtain \eqref{b100-twocase}.
\end{proof}
Similarly to the 2D case, we obtain a relation between $x$ and $r$.
\begin{lema}   
Under the assumption of \eqref{plasmas3dalphagamma}, there exists a constant $C(T)>0$ such that
\begin{equation}\label{x1301}
x \leq C(T)\, r(x,\tau)^{\frac{3(\gamma-1)}{\gamma}}, \quad (x,\tau) \in [0,1] \times [0,T].
\end{equation}
\end{lema} 
\begin{proof}
For any $x\in[0,1]$ and \eqref{b100-twocase}
 we have,
\begin{equation}
\begin{aligned}
x = \int_{0}^{r(x,\tau)} \rho s^2\, ds \leq& \left( \int_{0}^{r(x,\tau)} \rho^\gamma s^2\, ds \right)^{\frac{1}{\gamma}} \left( \int_{0}^{r(x,\tau)} s^2 \,ds \right)^{\frac{\gamma-1}{\gamma}}\\
\leq& C(T) r^{\frac{3(\gamma-1)}{\gamma}}.
\end{aligned}\end{equation}
\end{proof}
\begin{lema}For any sufficiently small $\xi > 0$, there exists a positive constant $C$ independent of $T$ such that
     \begin{equation}\label{w4rhoalpha<1}
         \sup_{\tau\in [0, T]}\left\lVert \rho^{\alpha-\frac{1}{2}} r^{\frac{1}{2}+\xi} \right\rVert _{L^\infty(0, R)} \le C\left(\Big\|\nabla\rho^{\alpha-\frac{1}{2}}\Big\|_{L^2(\Omega)}+1\right).
     \end{equation}
\end{lema}
\begin{proof}
By the one-dimensional Sobolev inequality, we have
   \begin{equation}
\begin{aligned}\label{w3rhoalpha<1}
\left\lVert \rho^{\alpha-\frac{1}{2}} r^{\frac{1}{2}+\xi} \right\rVert_{L^{\infty}(0,R)}
&\leq \int^{R}_{0} \rho^{\alpha-\frac{1}{2}} r^{\frac{1}{2}+\xi} \, dr
+ \int^{R}_{0} \Bigl| \bigl( \rho^{\alpha-\frac{1}{2}} r^{\frac{1}{2}+\xi} \bigr)_r \Bigr| \, dr \\
&= M_{1} + M_{2}.
\end{aligned}
\end{equation}
Then, combining H\"older's inequality with the Sobolev inequality yields    
    \begin{equation}
        \begin{aligned}\label{w2rhoalpha<1}
M_{1}&\leq \left( \int^{R}_{0}\left( \rho^{\alpha-\frac{1}{2}}r^{\frac{1}{3}} \right)^{6} dr\right)^{\frac{1}{6}}\left( \int^{R}_{0} r^{\left( \xi+ \frac{1}{6} \right)\cdot \frac{6}{5}} dr\right)^{\frac{5}{6}}\\
&\leq C\left( \int^{R}_{0}\rho^{6\alpha-3}r^{2}dr \right)^{\frac{1}{6}}\left( \int^{R}_{0}r^{\frac{6}{5}\xi+ \frac{1}{5}}dr \right)^{\frac{5}{6}}\\
&\leq C\left(\Big\|\nabla\rho^{\alpha-\frac{1}{2}}\Big\|_{L^2(\Omega)}+\Big\|\rho^{\alpha-\frac{1}{2}}\Big\|_{L^2(\Omega)}\right)\\
&\leq C\left(\Big\|\nabla\rho^{\alpha-\frac{1}{2}}\Big\|_{L^2(\Omega)}+1\right).
\end{aligned}
    \end{equation}
For $M_2$, a similar argument yields
    \begin{equation}
        \begin{aligned}\label{w1rhoalpha<1}
M_{2}&= \int^{R}_{0}\Big |\left( \rho^{\alpha-\frac{1}{2}}r^{\frac{1}{2}+\xi} \right)_r\Big|dr\\
&\leq C\left( \int^{R}_{0}|\partial_{r}\rho^{\alpha-\frac{1}{2}}|r^{2}dr \right)^{\frac{1}{2}}\left( \int^{R}_{0}r^{2\xi-1}dr \right)^{\frac{1}{2}}+C\left( \int^{R}_{0}\rho^{6\alpha-3}r^{2}dr \right)^{\frac{1}{6}}\left( \int^{R}_{0}r^{\frac{6}{5}\xi-1}dr \right)^{\frac{5}{6}}\\
&\leq C\left(\Big\|\nabla\rho^{\alpha-\frac{1}{2}}\Big\|_{L^2(\Omega)}+1\right).
\end{aligned}
    \end{equation}
Substituting \eqref{w2rhoalpha<1}-\eqref{w1rhoalpha<1} into \eqref{w3rhoalpha<1} gives \eqref{w4rhoalpha<1}, which completes the proof.
\end{proof}
\begin{lema}
Under the assumption of \eqref{plasmas3dalphagamma}, there exists a constant $C(T)>0$ such that, for both the plasma case ($\kappa = -1$) and the gaseous star case ($\kappa = 1$) with $\gamma > \frac{4}{3}$, the following estimate holds
    \begin{equation}\label{bdsfinall000}
        \sup_{0\leq \tau\leq T}\int^{1}_{0}w^{2}dx+\int^{T}_{0}\int^{1}_{0}\left( \left( \rho^{\frac{\gamma+\alpha}{2}} \right)_{x}r^{2} \right)^{2}dxds\leq C(T).
    \end{equation}
\end{lema}
\begin{proof}
We multiply the equation \eqref{bd2_alpha<1} by $w$ to obtain
    \begin{equation}\begin{aligned}\label{bd3}
        \frac{1}{2} \frac{d}{d\tau}\int^{1}_{0}w^{2}dx=&-\int^{1}_{0}(\rho^{\gamma})_{x}r^{2}(u+r^{2}(\rho^{\alpha})_x )dx{-\kappa}\int^{1}_{0} \frac{x}{r^{2}}(u+r^{2}(\rho^{\alpha})_x )dx\\
        &+ \frac{{\kappa}}{3}\int^{1}_{0}r(u+r^{2}(\rho^{\alpha})_x )dx\\
        \triangleq& G_1+G_2+G_3.
     \end{aligned}\end{equation}
Based on $\eqref{eq:NSP-lagrangian-rhoalpha}_1$, we obtain
     \begin{equation}
        \begin{aligned}\label{bd31ALPHA<1}
G_{1}&= -\int^{1}_{0}(\rho^{\gamma})_{x}r^{2}(u+r^{2}(\rho^{\alpha})_x)dx\\
&= -\int^{1}_{0}\rho^{\gamma} \frac{\rho_{\tau}}{\rho^{2}}dx- \frac{4{\alpha}\gamma}{(\gamma+\alpha)^{2}}\int^{1}_{0}\left( \left( \rho^{\frac{\gamma+\alpha}{2}} \right)_{x}r^{2} \right)^{2}dx\\
&= -\frac{1}{\gamma-1} \frac{d}{d\tau}\int^{1}_{0}\rho^{\gamma-1}dx- \frac{4{\alpha}\gamma}{(\gamma+\alpha)^{2}}\int^{1}_{0}\left( \left( \rho^{\frac{\gamma+\alpha}{2}} \right)_{x}r^{2} \right)^{2}dx.
\end{aligned}
    \end{equation}
    Using Young's inequality, we deduce that
 \allowdisplaybreaks
\begin{align} \label{1bd31ALPHA<1}
G_3=&\frac{{\kappa}}{3}\int^{1}_{0}r(u+r^{2}(\rho^{\alpha})_{x})dx\nonumber\\
\leq& C\int^{1}_{0} u^2dx+ C\int^{1}_{0}r^2dx+ C\int^{1}_{0}r^{3}(\rho^{\alpha})_{x}dx\nonumber\\
\leq&C\int^{1}_{0} u^2dx+ C\left( \int^{1}_{0}(r^{2}(\rho^{\alpha})_{x})^{2}dx \right)^{\frac{1}{2}}\left( \int^{1}_{0}r^{2}dx \right)^{\frac{1}{2}}+C\nonumber\\
\leq& C\int^{1}_{0} u^2dx+C\Big\|w-u\Big\|^2_{L^2(0,1)}+C\nonumber\\
\leq& C(1+\left\lVert w \right\rVert ^2_{L^{2}(0,1)}+\left\lVert u \right\rVert^2 _{L^{2}(0,1)}).
\end{align}
\noindent\textbf{Case 1: $\kappa = -1$ (plasmas)}

Using integration by parts, we have
\begin{equation}
        \begin{aligned}  \label{bd3230finall}
  G_2=&  \int^{1}_{0} \frac{x}{r^{2}}(u+r^{2}(\rho^{\alpha})_{x})dx\\
 =&-\frac{d}{d\tau} \int^{1}_{0} \frac{x}{r}dx+x\rho^{{\alpha}}|^{1}_{0}-\int^{1}_{0}\rho^{\alpha} dx\\
=&- \frac{d}{d\tau} \int^{1}_{0} \frac{x}{r}dx+\rho^{\alpha}(1,\tau)-\int^{1}_{0}\rho^{\alpha} dx.
\end{aligned}\end{equation} 
To estimate the boundary term $\rho^{\alpha}(1,\tau)$, we employ the following identity
\[
\rho^{\alpha}(1,\tau) = \rho^{\alpha}(x,\tau) + \int_{x}^{1} (\rho^{\alpha})_y \, dy.
\]
Integrating both sides of the above identity with respect to $x$ over the interval $[\frac{1}{2}, 1]$ yields
\begin{equation}
\begin{aligned}\label{boundarytermsfianlly}
\frac{1}{2} \rho^{\alpha}(1,\tau) =& \int_{\frac{1}{2}}^{1} \rho^{\alpha}(x,\tau) \, dx + \int_{\frac{1}{2}}^{1} \int_{x}^{1} (\rho^{\alpha})_y \, dy \, dx\\
=&\int_{\frac{1}{2}}^{1} \rho^{\alpha}(x,\tau) \, dx + \int_{\frac{1}{2}}^{1} \int_{\frac{1}{2}}^{y} (\rho^{\alpha})_y \, dx \, dy
\end{aligned}\end{equation}
From \eqref{x1301}, it follows that when $x > 0$, we also have $r > 0$, which implies the uniform bounds
\begin{equation}\label{rlowupper}
r\geq C_0,
\end{equation}
where $C_0$ depends on $T$.
Using the definition of the effective velocity $w $ together with the bounds in \eqref{rlowupper}, a direct calculation yields
\begin{equation}\begin{aligned}\label{boundaryterms}
\left| \int_{\frac{1}{2}}^1 \int_{\frac{1}{2}}^y (\rho^\alpha)_y dx dy \right| &= \left| \int_{\frac{1}{2}}^1 \left(y - \frac{1}{2}\right) (\rho^\alpha)_y dy \right| \le \frac{1}{2} \int_{\frac{1}{2}}^1 |(\rho^\alpha)_y| dy \le C \int_{\frac{1}{2}}^1 r^2 |(\rho^\alpha)_y| dy \\
&= C \int_{\frac{1}{2}}^1 |w(y) - u(y)| dy \\
&\le C \int_{\frac{1}{2}}^1 w^2 dy + C \int_{\frac{1}{2}}^1 u^2 dy.
\end{aligned}\end{equation}
Combining \eqref{boundaryterms} with \eqref{boundarytermsfianlly}, we finally obtain
\begin{equation}\label{boundaryterms1}
\frac{1}{2} \rho^{\alpha}(1,\tau) \leq C\int_{0}^{1} \rho^{\alpha} \, dx + C \int_{0}^{1} w^2 \, dx + C \int_{0}^{1} u^2 \, dx.
\end{equation}
From \eqref{bd3230finall} and \eqref{boundaryterms1}, we have
    \begin{equation}\label{kapap-1g2}
G_2\leq- \frac{d}{d\tau} \int^{1}_{0} \frac{x}{r}dx+C\int^{1}_{0}\rho^{\alpha} dx+C \int_{0}^{1} w^2 \, dx + C \int_{0}^{1} u^2 \, dx.
\end{equation}
\noindent\textbf{Case 2: $\kappa = 1$ (gaseous stars)}

By integration by parts, we have
  \begin{equation}
        \begin{aligned}  \label{kapap1g2}
  G_2=&  -\int^{1}_{0} \frac{x}{r^{2}}(u+r^{2}(\rho^{\alpha})_{x})dx\\
 =&\frac{d}{d\tau} \int^{1}_{0} \frac{x}{r}dx-x\rho^{{\alpha}}|^{1}_{0}+\int^{1}_{0}\rho^{\alpha} dx\\
\leq& \frac{d}{d\tau} \int^{1}_{0} \frac{x}{r}dx+\int^{1}_{0}\rho^{\alpha} dx.
\end{aligned}\end{equation} 
Putting together \eqref{bd31ALPHA<1}, \eqref{1bd31ALPHA<1}, and \eqref{kapap-1g2}--\eqref{kapap1g2} gives  
    \begin{equation}
        \begin{aligned}
&\frac{1}{2} \frac{d}{d\tau}\int^{1}_{0}w^{2}dx+ \frac{1}{\gamma-1} \frac{d}{d\tau}\int^{1}_{0}\rho^{\gamma-1}dx+ \frac{4\alpha\gamma}{(\gamma+\alpha)^{2}}\int^{1}_{0}\left( \left( \rho^{\frac{\gamma+\alpha}{2}} \right)_{x}r^{2} \right)^{2}dx\\
\leq& \kappa\frac{d}{d\tau}\int^{1}_{0} \frac{x}{r}dx+C\int^{1}_{0}\rho^{\alpha} dx+ C\int^{1}_{0}w^{2}dx+C\int^{1}_{0}u^{2}dx+C.
\end{aligned}
    \end{equation}
Integration of the above inequality over $[0,\tau]$ and using \eqref{j11} leads to
    \begin{equation}
        \begin{aligned}\label{finallalpha<1both}
&\frac{1}{2}\int^{1}_{0}w^{2}dx+ \frac{1}{\gamma-1}\int^{1}_{0}\rho^{\gamma-1}dx+ \frac{4\alpha\gamma}{(\gamma+\alpha)^{2}}\int^{\tau}_{0}\int^{1}_{0}\left( \left( \rho^{\frac{\gamma+\alpha}{2}} \right)_{x}r^{2} \right)^{2}dxds\\
\leq&\frac{1}{2}\int^{1}_{0}w_{0}^{2}dx+ \frac{1}{\gamma-1}\int^{1}_{0}\rho_{0}^{\gamma-1}dx + \frac{1}{2(\gamma-1)}\int^{1}_{0}\rho^{\gamma-1}dx+C(T)\\
&+C\int^{\tau}_{0}\int^{1}_{0}\rho^\alpha dxds+C\int^{\tau}_{0}\int^{1}_{0}w^{2}dxds\\
\leq& C(T)+C\int^{\tau}_{0}\int^{1}_{0}\rho^{\alpha} dxds+C\int^{\tau}_{0}\int^{1}_{0}w^{2}dxds.
\end{aligned}
    \end{equation}
Then it remains to estimate the term $\int^{\tau}_{0}\int^{1}_{0}\rho^{\alpha} dxds.$

It is worth noting that we have
\begin{equation}
    \begin{aligned}
\int^{1}_{0}\rho^{\alpha} dx&= \int^{R}_{0}\rho^{{1+\alpha}}r^{2}dr=\frac{1}{4\pi}\int_{\Omega}\rho^{{1+\alpha}}d\mathbf{x}.
\end{aligned}
\end{equation}
Applying the interpolation inequality gives
\begin{equation}\label{3rho1010}
    \begin{aligned}
\int^{1}_{0}\rho^\alpha dx = \frac{1}{4\pi}\left\lVert \rho \right\rVert^{\alpha+1}_{L^{\alpha+1}(\Omega)} &\leq C\left\lVert \rho \right\rVert_{L^{1}(\Omega)}^{\frac{3\gamma+2\alpha-4}{3(\gamma+\alpha)-4}}\left\lVert \rho \right\rVert_{L^{3(\gamma+\alpha-1)}(\Omega)}^{\frac{3\alpha(\gamma+\alpha-1)}{3(\gamma+\alpha)-4}}\\
&\leq C\left\lVert \rho \right\rVert_{L^{3(\gamma+\alpha-1)}(\Omega)}^{\frac{3\alpha(\gamma+\alpha-1)}{3(\gamma+\alpha)-4}}\\
&= C\left\lVert \rho^{\frac{\gamma+\alpha-1}{2}} \right\rVert _{L^{6}(\Omega)}^{\frac{6\alpha}{3(\gamma+\alpha)-4}}\\
&\leq C\Big(\left\lVert \rho^{\frac{\gamma+\alpha-1}{2}} \right\rVert_{L^{2}(\Omega)}+\left\lVert \nabla\rho^{\frac{\gamma+\alpha-1}{2}} \right\rVert_{L^{2}(\Omega)}\Big)^{\frac{6\alpha}{3(\gamma+\alpha)-4}}\\
&\leq C(T)\Big(1+\left\lVert \nabla\rho^{\frac{\gamma+\alpha-1}{2}} \right\rVert_{L^{2}(\Omega)}^{\frac{6\alpha}{3(\gamma+\alpha)-4}}\Big)\\
&\leq\frac{2\gamma\alpha}{(\gamma+\alpha)^{2}}\int^{1}_{0}\left( \left( \rho^{\frac{\gamma+\alpha}{2}} \right)_{x}r^{2} \right)^{2}dx+C(T),
\end{aligned}
\end{equation}
where we used \eqref{plasmas3dalphagamma}.
Inserting \eqref{3rho1010} into \eqref{finallalpha<1both} gives   
    \begin{equation}
        \begin{aligned}
\frac{1}{2}\int^{1}_{0}w^{2}dx+ \frac{1}{2(\gamma-1)}\int^{1}_{0}\rho^{\gamma-1}dx+&  \frac{2\gamma\alpha}{(\gamma+\alpha)^{2}}\int^{1}_{0}\left( \left( \rho^{\frac{\gamma+\alpha}{2}} \right)_{x}r^{2} \right)^{2}dx\\
\leq& C(T)+C\int^{\tau}_{0}\int^{1}_{0}w^{2}dxds.
\end{aligned}
    \end{equation}
Applying Gronwall’s inequality, we finally arrive at \eqref{bdsfinall000}.

 From \eqref{bdsfinall000} and the definition of $w$, we have
 \begin{equation}\label{w4rho1}
         \sup_{0\leq t\leq T}\left\lVert \rho^{\alpha-\frac{1}{2}} r^{\frac{1}{2}+\xi} \right\rVert _{L^\infty(0, R)} \le C(T).
     \end{equation}

\end{proof}
With the BD entropy estimate, we obtain higher integrability of the density. Consequently, we derive the following relation between $x$ and $r$.
\begin{lema}
Under the assumption of \eqref{plasmas3dalphagamma}, there exists a constant $C(T)>0$ such that
    \begin{equation}\label{xrbd3}
   x \leq C(T)\,  {r(x,t)^{3\left(1-\frac{1}{6\alpha-3}\right)}}, \quad (x,\tau) \in [0,1] \times [0,T].
\end{equation}
\end{lema}
\begin{proof}
It follows from \eqref{b100-twocase} and \eqref{bdsfinall000} that
    \begin{equation}\label{est: rho alpha-1/2 H1}
        \begin{aligned}
 \int_{\Omega}\Big|\nabla \rho^{{\alpha-}\frac{1}{2}}\Big|^{2}d\mathbf{x}
\leq C\int^{1}_{0}(w-u)^2dx\leq C\int^{1}_{0}w^2dx+C\int^{1}_{0}u^2dx\leq C(T).
\end{aligned}
    \end{equation}
From the three-dimensional Sobolev embedding inequality and \eqref{est: rho alpha-1/2 H1}, we obtain

\begin{equation}\label{6alpha-3t}
\begin{aligned}
\int_\Omega \rho^{6\alpha-3} d\mathbf{x}
\le C \left( \|\rho^{\alpha-\frac{1}{2}}\|_{L^2(\Omega)} + \|\nabla \rho^{\alpha-\frac{1}{2}}\|_{L^2(\Omega)} \right)^6 \le C(T).
\end{aligned}
\end{equation}
Consequently, the above inequality implies that
\begin{equation}
\begin{aligned}
x &= \int_0^{r} \rho s^2 ds \\
&\le C\left( \int_0^{r} \rho^{6\alpha-3} s^2 ds \right)^{\frac{1}{6\alpha-3}} \left( \int_0^{r} s^2 ds \right)^{\frac{6\alpha-4}{6\alpha-3}} \\
&\le C(T) r^{3\left(1-\frac{1}{6\alpha-3}\right)}.
\end{aligned}
\end{equation}
\end{proof}

\begin{lema}
Under the assumption of \eqref{plasmas3dalphagamma} and the condition $1< n< n_{3}(\alpha)$ with $n\in \mathbb{R}$, the following estimates hold 
\begin{equation}\begin{aligned}\label{3du4}
&\sup_{0\leq t\leq T}\int_{0}^{1}u^{2n}dx + \int_{0}^{T}\int_{0}^{1}\Big(\rho^{\alpha-1}\frac{u^{2n}}{r^{2}} +\rho^{1+\alpha}r^{4}u_x^{2}u^{2n-2}\Big)dxdt \\
\leq& C+C\int_{0}^{T}\int_{0}^{1}\rho^{2n(\gamma-\alpha)+\alpha-1}r^{2n-2}dxdt\\
&+C\int_{0}^{T}\int_{0}^{1}\rho^{(2n-1)(1-\alpha)}r^{6n(1-\frac{1}{6\alpha-3})-2}dxdt. \\
\end{aligned}\end{equation}
\end{lema}
\begin{proof}
   By multiplying \eqref{eq:NSP-lagrangian-rhoalpha} by $u^{2n-1}$ and then integrating by parts, we obtain
   \begin{equation}
        \begin{aligned}\label{3uhigh10}
            \frac{1}{2n} \frac{d}{d\tau}\int^{1}_{0}u^{2n}dx+& (4\alpha-2)\int^{1}_{0} \rho^{\alpha-1}\frac{u^{2n}}{r^{2}}dx+(2n-1)\alpha\int^{1}_{0}\rho^{1+\alpha}u^{2n-2}u_{x}^{2}r^{4}dx\\
            &+4n(\alpha-1)\int_{0}^{1}\rho^{\alpha}u^{2n-1}u_xrdx\\
           =&\int^{1}_{0}\rho^{\gamma}(r^{2}u^{2n-1})_{x}dx -\kappa\int^{1}_{0} \frac{x}{r^{2}}u^{2n-1}dx +\frac{\kappa}{3}\int^{1}_{0}u^{2n-1}rdx\\
           \triangleq& H_1+H_2+H_3.
        \end{aligned}
    \end{equation}
Applying Young's inequality, we obtain
\begin{equation}
\begin{aligned}\label{equ}
&4n(\alpha-1) \rho^\alpha u^{2n-1}  u_x r\\
\ge& -(1-2\varepsilon)(2n-1)\alpha \rho^{1+\alpha} u_x^2 u^{2n-2} r^{4} - \frac{\big(4n(\alpha-1)\big)^2}{4(1-2\varepsilon)(2n-1)\alpha} \rho^{\alpha-1} \frac{u^{2n}}{r^2}\\
\geq& -(1-2\varepsilon)\Big((2n-1)\alpha\rho^{1+\alpha}u_x^2u^{2n-2}r^4+(4\alpha-2)\rho^{\alpha{-1}}\frac{u^{2n}}{r^2}\Big).
\end{aligned}
\end{equation}
provided that \(\varepsilon\) is sufficiently small such that
\begin{equation}\label{2depsilon}
\frac{\big(4n(\alpha-1)\big)^2}{4(1-2\varepsilon)(2n-1)\alpha} < (1-2\varepsilon)(4\alpha-2),
\end{equation}
The estimate \eqref{2depsilon} holds, provided that \( n < n_3(\alpha) \), which guarantees the existence of such \(\varepsilon\).
 Integration by parts combined with H\"older's inequality implies that   
    \begin{equation}\begin{aligned}\label{3uhigh2}
H_{1}&= \int^{1}_{0}\rho^{\gamma}(r^{2}u^{2n-1})_{x}dx\\
&= 2\int^{1}_{0}\rho^{\gamma-1} \frac{u^{2n-1}}{r}dx+(2n-1)\int^{1}_{0}\rho^{\gamma}u^{2n-2}u_{x}r^{2}dx\\
&\leq C\left( \int^{1}_{0}\rho^{\alpha-1} \frac{u^{2n}}{r^{2}}dx \right)^{\frac{2n-1}{2n}}\left( \int^{1}_{0}\rho^{2n(\gamma-\alpha)+\alpha-1}r^{2(n-1)}dx \right)^{\frac{1}{2n}}\\
&\quad+C\left( \int^{1}_{0} \rho^{1+\alpha}u^{2n-2}r^{4}u_{x}^{2}dx \right)^{\frac{1}{2}}\left( \int^{1}_{0} \rho^{\alpha-1}\frac{u^{2n}}{r^{2}}dx \right)^{\frac{n-1}{2n}}\left( \int^{1}_{0}\rho^{2n(\gamma+\alpha)+\alpha-1}r^{2(n-1)} dx\right)^{\frac{1}{2n}}\\
&\leq \frac{\varepsilon(4\alpha-2)}{4}\int^{1}_{0}\rho^{\alpha-1} \frac{u^{2n}}{r^{2}}dx +C(\varepsilon)\int^{1}_{0}\rho^{2n(\gamma-\alpha)+\alpha-1}r^{2(n-1)}dx\\
&\quad+\frac{\varepsilon(2n-1)\alpha}{4} \int^{1}_{0}\rho^{1+\alpha}u^{2n-2}r^{4}u_{x}^{2}dx.
\end{aligned}\end{equation}
Using H\"older's inequality gives
\begin{equation}\begin{aligned}\label{3uhigh2}
H_{2}&= -\kappa\int^{1}_{0} \frac{x}{r^{2}}u^{2n-1}dx\\
&\leq C\left( \int^{1}_{0} \rho^{\alpha-1}\frac{u^{2n}}{r^{2}}dx \right)^{\frac{2n-1}{2n}}\left( \int^{1}_{0} \rho^{(2n-1)(1-\alpha)}\frac{x^{2n}}{r^{2}}dx \right)^{\frac{1}{2n}}\\
&\leq \frac{\varepsilon(4\alpha-2)}{4}\int_{0}^{1}\rho^{\alpha-1}\frac{u^{2n}}{r^{2}}dx+C\int_{0}^{1}\rho^{(2n-1)(1-\alpha)}r^{6n(1-\frac{1}{6\alpha-3})-2}dx.\\
\end{aligned}\end{equation}
In the same manner,
\begin{equation}\begin{aligned}\label{3uhigh30}
H_{3}&\leq C\int^{1}_{0}u^{2n}dx +C.\\
\end{aligned}\end{equation}
Substituting \eqref{equ}–\eqref{3uhigh30} into \eqref{3uhigh10} yields
\begin{equation}\begin{aligned}\label{3du41}
&\frac{d}{d\tau}\int_{0}^{1}u^{2n}dx + \varepsilon\int_{0}^{1}\Big((4\alpha-2)\rho^{\alpha-1}\frac{u^{2n}}{r^{2}} +(2n-1)\alpha\rho^{1+\alpha}r^{4}u_x^{2}u^{2n-2}\Big)dx\\
\leq &C+C\int_{0}^{1}\rho^{2n(\gamma-\alpha)+\alpha-1}r^{2n-2}dx\\
&+C\int_{0}^{1}\rho^{(2n-1)(1-\alpha)}r^{6n(1-\frac{1}{6\alpha-3})-2}dx. \\
\end{aligned}\end{equation}
An application of Gronwall's inequality leads to \eqref{3du4}.
\end{proof}
\begin{lema}
Under  the condition $1< m\in A_{\text{set}}$, the following estimates hold 
\begin{equation}\begin{aligned}\label{3dw4}
\sup_{0\leq t\leq T}\int_{0}^{1}w^{2m}dx + \int_{0}^{T}\int_{0}^{1}\Big(r^2\Big(\rho^{\frac{\gamma+\alpha(2m-1)}{2m}}\Big)_x\Big)^{2m}dx dt \leq& C+C\int_{0}^{T}\int_{0}^{1}\rho^{\gamma-\alpha}u^{2m}dxdt \\
&+C\int_{0}^{T}\int_{0}^{1}r^{2m(1-\frac{1}{2\alpha-1})}dxdt.
\end{aligned}\end{equation}
\end{lema}
\begin{proof}
By multiplying \eqref{bd2_alpha<1} by $w^{2m-1}$ we obtain 
\begin{equation}\begin{aligned}\label{bdw001}
    \frac{1}{2m} \frac{d}{d\tau} \int^{1}_{0}w^{2m}dx=&-\int^{1}_{0} (\rho^{\gamma})_{x}r^{2}w^{2k-1}dx  -\kappa\int^{1}_{0} \frac{x}{r^{2}}w^{2m-1}dx +  \frac{\kappa}{3}\int^{1}_{0}w^{2m-1}rdx\\
    \triangleq& H_1+H_2+H_3.
\end{aligned}\end{equation}
Using Lemma \ref{mdense}, we obtain
\begin{equation}\begin{aligned}\label{bdw100}
H_{1}&= -\gamma \int^{1}_{0}\rho^{\gamma-1}\rho_{x}r^{2}w^{2m-1}dx\\
&= -\int^{1}_{0} r^2(\rho^{\alpha})_x\bigl(u + r^2(\rho^{\alpha})_x\bigr)^{2m-1}\frac{\gamma}{\alpha}\rho^{\gamma-\alpha}dx\\
&\leq -\varepsilon\int^{1}_{0}\rho^{\gamma-\alpha}\bigl(r^2(\rho^{\alpha})_x\bigr)^{2m}dx+C\int^{1}_{0}\rho^{\gamma-\alpha}u^{2m}dx .
\end{aligned}\end{equation}
By virtue of H\"older's inequality, we obtain
\begin{equation}\begin{aligned}\label{2whighintegrability31}
H_{2}&\leq C\left( \int^{1}_{0} w^{(2m-1) \frac{2m}{2m-1}}dx \right)^{\frac{2m-1}{2m}}\left( \int^{1}_{0} \left( \frac{x}{r^2} \right)^{2m}dx \right)^{\frac{1}{2m}}\leq C\int^{1}_{0}w^{2m}dx+C\int_{0}^{1}r^{2m(1-\frac{1}{2\alpha-1})}dx.\\
\end{aligned}\end{equation}
In the same manner
\begin{equation}\begin{aligned}\label{2whighintegrability3}
H_3\leq C\int^{1}_{0}w^{2m}dx+C.
\end{aligned}\end{equation}
Combining \eqref{bdw100}–\eqref{2whighintegrability3} with \eqref{bdw001} yields
\begin{equation}\begin{aligned}
    \frac{1}{2m} \frac{d}{d\tau}\int^{1}_{0}w^{2m}dx+ \varepsilon\int^{1}_{0}\rho^{\gamma-\alpha}\bigl(r^2(\rho^{\alpha})_x\bigr)^{2m}dx&\leq C\int^{1}_{0}w^{2m}dx+C\int^{1}_{0}\rho^{\gamma-1}u^{2m}dx\\
    &\quad+C\int_{0}^{1}r^{2m(1-\frac{1}{2\alpha-1})}dx.\\
\end{aligned}\end{equation}
Applying Gronwall's inequality to the above inequality gives \eqref{3dw4}.
\end{proof}
\begin{lema}\label{2dlemma3}
Let $N = 3$. For any $(\alpha,\gamma)$ satisfying 
 \eqref{plasmas3dalphagamma}, there exists an $\alpha,\gamma$-dependent constant $k\in A_{set}$ with $k > 3$ such that
\begin{equation}\label{3dalphagamma}
\frac{4k+3}{4k+6}< \alpha < 1 - \frac{1}{2k} ,
\end{equation}
and
\begin{equation}\label{3dalphagamma1}
\frac{4}{3} < \gamma < 3\alpha -1 + \frac{\alpha}{2k}. 
\end{equation}
\end{lema}
\begin{proof}
The first inequality in \eqref{3dalphagamma} is equivalent to $k <\frac{6\alpha-3}{4-4\alpha}$. A direct computation shows that
\begin{equation}\label{3dn3limt}
\frac{4n+3}{4n+6} < 1 - \frac{1}{2n} ,\quad \Longleftrightarrow \quad n > 3.
\end{equation}
$\frac{6\alpha-3}{4-4\alpha}$ is strictly increasing on $(\frac{5}{6},1)$. Set $k_0 = \frac{1}{2(1-\alpha)}$. Then \eqref{3dn3limt} yields
\[
\frac{4k_0+3}{4k_0+6} < \alpha = 1 - \frac{1}{2k_0},
\]
and \eqref{plasmas3dalphagamma}, gives
\[
\frac{4}{3}< \gamma < 3\alpha - 1 + \frac{\alpha}{2k_0}.
\]
By continuity of the relevant functions and the density of $A_{\text{set}}$ in $(1,\infty)$, we may select $k \in A_{\text{set}}$ such that $k > k_0$ and $k - k_0$ is sufficiently small, thereby ensuring that both \eqref{3dalphagamma} and \eqref{3dalphagamma1} hold.
\end{proof}
\begin{lema}\label{rhouw3dhigh}
Under the assumption of \eqref{plasmas3dalphagamma}, and let $k$ be defined as in Lemma \ref{2dlemma3}. Then there exists a constant 
\begin{equation}\label{sigma}
\frac{2k+2-(2k+3)\alpha}{k} < \sigma <\frac{1}{2k},
\end{equation}
such that
\begin{equation}\label{gamma}
\alpha - \frac{\alpha}{2k} + \sigma < \gamma < 3\alpha - 1 + \frac{\alpha - 1}{2k} + \sigma .
\end{equation}
Moreover, the following estimates hold:
\begin{equation}\label{rhouhigh}
\sup_{0\leq t\leq T}\int_{0}^{R}\rho u^{2k}r^{2}dr + \int_{0}^{T}\int_{0}^{R}\Big(\rho^{\alpha}u^{2k} + \rho^{\alpha}u_r^{2}u^{2k-2}r^{2}\Big)drdt \leq C(T)R_{T}^{2k\sigma},
\end{equation}
and
\begin{equation}\label{w-uhigh}
\sup_{0\leq t\leq T}\int_{0}^{R}\bigl|\partial_{r}\rho^{\alpha -1 + \frac{1}{2k}}\bigr|^{2k}r^{2}dr + \int_{0}^{T}\int_{0}^{R}\Big|\partial_{r}\rho^{\alpha -1 + \frac{\gamma - \alpha + 1}{2k}}\Big|^{2k}r^{2}drdt \leq C(T)R_{T}^{2k\sigma}.
\end{equation}
\end{lema}
\begin{proof}
Transforming \eqref{3du4} into Eulerian coordinates and then applying \eqref{w4rho1} and \eqref{sigma}, we obtain
\allowdisplaybreaks
\begin{align}
& \sup_{0\leq t\leq T}\int_{0}^{R}\rho u^{2k}r^{2}dr + \int_{0}^{T}\int_{0}^{R}\left(\rho^{\alpha}u^{2k} + \rho^{\alpha}u_r^{2}u^{2k - 2}r^{2}\right)drdt \nonumber\\
&\leq C + C\int_{0}^{T}\int_{0}^{R}\rho^{2k(\gamma -\alpha) + \alpha}r^{2k}drdt \nonumber\\
&\quad + \int_{0}^{T}\int_{0}^{R}\rho^{(2k-1)(1-\alpha)+1}r^{6k\left(1-\frac{1}{6\alpha-3}\right)}drdt \nonumber\\
&\leq C + C\sup_{0\leq t\leq T}\left\|\rho^{\alpha -\frac{1}{2}}r^{\frac{1}{2} +\xi}\right\|_{L^{\infty}(0,R)}^{\frac{2k(\gamma -\alpha) + \alpha -2k\sigma}{\alpha-\frac{1}{2}}}R_{T}^{2k\sigma} \int_{0}^{T}\int_{0}^{R} r^{2k - \left(\frac{1}{2}+\xi\right)\frac{2k(\gamma -\alpha) + \alpha - 2k\sigma}{\alpha-\frac{1}{2}}} \, dr \, dt \nonumber\\
&\quad + C\left\|\rho^{\alpha-\frac{1}{2}}r^{\frac{1}{2}+\xi}\right\|_{L^{\infty}(0,R)}^{\frac{(1-\alpha)(2k-1)+1-2k\sigma}{\alpha-\frac{1}{2}}} R_{T}^{2k\sigma}  \int_{0}^{T}\int_{0}^{R} r^{6k\left(1-\frac{1}{6\alpha-3}\right) - \left(\frac{1}{2}+\xi\right)\frac{(1-\alpha)(2k-1)+1-2k\sigma}{\alpha-\frac{1}{2}}} \, dr \, dt \nonumber\\
&\leq C(T)R_{T}^{2k\sigma},
\end{align}
where $\xi > 0$ is a sufficiently small constant satisfying
\begin{equation}
\left\{
\begin{aligned}
&2k - \frac{2k(\gamma - \alpha) + \alpha - 2k\sigma}{\alpha - 1 / 2}\left(\frac{1}{2} +\xi\right) > -1, \\[6pt]
&6k\left(1-\frac{1}{6\alpha-3}\right) - \left(\frac{1}{2}+\xi\right)\frac{(1-\alpha)(2k-1)+1-2k\sigma}{\alpha-\frac{1}{2}} > -1.
\end{aligned}
\right.
\end{equation}
The existence of such $\xi$ follows from \eqref{sigma} and \eqref{gamma}. We thus obtain \eqref{rhouhigh}.
We now prove \eqref{w-uhigh}. Using \eqref{w4rho1} and \eqref{rhouhigh}, we obtain
\begin{equation}
\begin{aligned}
& \sup_{0\leq t\leq T}\int_{0}^{R}|\partial_{r}\rho^{\alpha -1 + \frac{1}{2k}}|^{2k}r^{2}dr + \int_{0}^{T}\int_{0}^{R}|\partial_{r}\rho^{\alpha -1 + \frac{\gamma - \alpha + 1}{2k}}|^{2k}r^{2}drdt \\
\leq &C\sup_{0\leq t\leq T}\int_{0}^{R}\rho u^{2k}r^{2}dr +C\sup_{0\leq t\leq T}\int_{0}^{R}\rho (u+\rho^{-1}(\rho^{\alpha})_r)^{2k}r^{2}dr\\
&+ C\int_{0}^{T}\int_{0}^{R}|\partial_{r}\rho^{\alpha -1 + \frac{\gamma - \alpha + 1}{2k}}|^{2k}r^{2}drdt  \\
\leq& C(T)R_{T}^{2k\sigma} + C\int_{0}^{T}\int_{0}^{R}\rho^{\alpha}u^{2k}\rho^{\gamma -2\alpha +1}r^{2}drdt \\
&+C\int_{0}^{T}\int_{0}^{R}\rho r^{2k(1-\frac{1}{2\alpha-1})+2}drdt\\
\leq& C(T)R_{T}^{2k\sigma} + C\sup_{0\leq t\leq T}\left\|\rho^{\alpha -\frac{1}{2}}r^{\frac{1}{2} +\xi}\right\|_{L^{\infty}(0,R)}^{\frac{\gamma-2\alpha+1}{\alpha-\frac{1}{2}}}\int_{0}^{T}\int_{0}^{R}\rho^{\alpha}u^{2k}drdt \\
&+C\sup_{0\leq t\leq T}\Big\|\rho^{\alpha-\frac{1}{2}}r^{\frac{1}{2}+\xi}\Big\|_{L^\infty}^{{\frac{1-2k\sigma}{\alpha-\frac{1}{2}}}}R_T^{2k\sigma}\int_{0}^{T}\int_{0}^{R}r^{2k(1-\frac{1}{2\alpha-1})+2-(\frac{1}{2}+\xi)\frac{(1-2k\sigma)}{\alpha-\frac{1}{2}}}drdt\\
\leq& C(T)R_{T}^{2k\sigma},
\end{aligned}
\end{equation}
where we used $\gamma - 2\alpha + 1 > 0$, and $\xi$ is taken sufficiently small so that
\begin{equation}
\left\{
\begin{aligned}
&2k\left(1-\frac{1}{2\alpha-1}\right)+2-\left(\frac{1}{2}+\xi\right)\frac{(1-2k\sigma)}{\alpha-\frac{1}{2}} > -1, \\[6pt]
&\left(\frac{\gamma - 2\alpha + 1}{\alpha - 1 / 2}\right)\left(\frac{1}{2} +\xi\right) < 2.
\end{aligned}
\right.
\end{equation}
The existence of such $\xi$ is guaranteed by \eqref{3dalphagamma} and \eqref{sigma}. This completes the proof of Proposition \ref{rhouw3dhigh}.
\end{proof}
\begin{prop}
\label{prop:3.9}
Under the assumption of \eqref{plasmas3dalphagamma}, there exists a constant $C(T)>0$ such that
\[
R_{T} \leq C(T),
\]
where $R_{T} := \displaystyle\sup_{0\leq t\leq T}\|\rho(t)\|_{L^{\infty}(\Omega)} + 1$.
\end{prop}
\begin{proof}
Choose $\beta$ such that
$$\sigma<\beta<(3\alpha-2)\Big(1-\frac{1}{2k}\Big),$$
where $\sigma$ is given in Proposition \ref{rhouw3dhigh}.
In view of \eqref{3dalphagamma}, one readily verifies that $\beta$ satisfies
\begin{equation}\label{betalimt}
\max \Big\{0,\alpha - 1 + \frac{1}{2k}\Big\} < \beta < \min \Big\{2\alpha - 1,\,(3\alpha - 2)\Big(1 - \frac{1}{k}\Big)\Big\}. 
\end{equation}
By the Sobolev embedding theorem in one dimension, we have
\begin{equation}
    \begin{aligned}
\left\lVert \rho^{\beta} \right\rVert _{L^{\infty}(0,R)}&\leq \int^{R}_{0}\rho^{\beta}dr+\int^{R}_{0}|\partial_{r}\rho^{\beta}|dr:=K_{1}+K_{2}.
\end{aligned}
\end{equation}
Using \eqref{6alpha-3t} and H\"older's inequality yields
\begin{equation}\begin{aligned}\label{rhobeta00}
K_{1}&= \int^{R}_{0}\rho^{\beta}dr=\int^{R}_{0}\Big(\rho^{6\alpha-3}r^{2}\Big)^{\frac{\beta}{6\alpha-3}}r^{- \frac{2\beta}{6\alpha-3}}dr\\
&\leq\left( \int^{R}_{0}\rho^{6\alpha-3}r^{2}dr \right)^{\frac{\beta}{6\alpha-3}}\left( \int^{R}_{0}r^{- \frac{2\beta}{6\alpha-\beta-3}}dr \right)^{\frac{6\alpha-\beta-3}{6\alpha-3}}\\
&\leq C(T),
\end{aligned}\end{equation}
where we use \eqref{betalimt}.

H\"older's inequality and \eqref{w-uhigh} gives that
\begin{equation}\begin{aligned}\label{rhobeta200}
K_{2}&= \int^{R}_{0} |\beta \rho^{\beta-1}\rho_{r}|dr\\
&\leq C\left( \int^{R}_{0}|\left( \rho^{\alpha-1+\frac{1}{2k}} \right)_r|^{2k}r^{2}dr \right)^{\frac{1}{2k}}\left( \int^{R}_{0}\left( \rho^{\beta+1-\alpha-\frac{1}{2k}}r^{- \frac{1}{k}} \right)^{\frac{2k}{2k-1}} dr\right)^{\frac{2k-1}{2k}}\\
&\leq C(T)(R_{T}^{2k \sigma})^{\frac{1}{2k}}\left( \int^{R}_{0}\rho^{(\beta+1-\alpha-\frac{1}{2k})\frac{2k}{2k-1}}r^{- \frac{2}{2k-1}}dr \right)^{\frac{2k-1}{2k}}\\
&\leq C(T)R_{T}^{\sigma}\left\lVert \rho^{\alpha-\frac{1}{2}}r^{\frac{1}{2}+\xi} \right\rVert _{L^{\infty}}^{\frac{\beta-\alpha+\frac{2k-1}{2k}}{\alpha-\frac{1}{2}}}\left( \int^{R}_{0}r^{- \frac{2}{2k-1}- (\frac{1}{2}+\xi) \frac{\frac{2k}{2k-1}(\beta-\alpha)+1}{\alpha-\frac{1}{2}}}dr \right)^{\frac{2k-1}{2k}}\\
&\leq C(T) R_{T}^{\sigma},
\end{aligned}\end{equation} 
where $\xi$ is chosen sufficiently small such that
\begin{equation}
    - \frac{2}{2k-1}- (\frac{1}{2}+\xi)\frac{\frac{2k}{2k-1}(\beta-\alpha)+1}{\alpha-\frac{1}{2}}> -1\implies \xi< \frac{(2k-1)(2k-3)-(2k(\beta-\alpha)2k-1)}{2(2k(\beta-\alpha)+2k-1)}.
\end{equation}
Such a $\xi$ exists due to \eqref{betalimt}.
We deduce from \eqref{rhobeta200} and \eqref{rhobeta00} that
\begin{equation}
    R_{T}^{\beta}\leq C(T)+C(T)R_{T}^{\sigma}\implies R_{T}\leq C(T),\beta>\sigma.
\end{equation}
\end{proof}
\begin{prop}
Under the assumption of \eqref{plasmas3dalphagamma}, there exists a constant $C(T)>0$ such that
\begin{equation}\label{lowerdensity00}
V_{T}\leq C(T). 
\end{equation}
\end{prop}

\begin{proof}
Let $v(x,\tau) = \frac{1}{\rho(x,\tau)}$. By $\eqref{eq:NSP-lagrangian-rhoalpha}_1$, we know that $v$ satisfies the equation $v_{\tau} = (r^{2}u)_{x}$. Integrating this equation over $(0,1)\times(0,\tau)$, we obtain
\begin{equation}\label{3dv1}
\int_{0}^{1}v(x,\tau)dx = \int_{0}^{1}v_{0}(x)dx \leq C. 
\end{equation}
Let $\xi > 0$ be chosen small enough such that $\xi < \frac{2k-3}{4k}$. Applying \eqref{3dv1}, \eqref{w-uhigh} and the one-dimensional Sobolev embedding yields
\begin{equation}
\begin{aligned}
v(x,\tau) &\leq \int_{0}^{1}vdx + \int_{0}^{1}|\partial_{x}v|dx \\
&\leq C + C\int_{0}^{1}|(\rho^{\alpha})_x|v^{\alpha+1}dx \\
&\leq C + C\left(\int_{0}^{1}|(\rho^{\alpha})_x|^{2k}r^{4k}dx\right)^{\frac{1}{2k}}\left(\int_{0}^{1}vr^{-\frac{3}{\xi+1}}dx\right)^{\frac{2+2\xi}{3}}\left(\int_{0}^{1}v^{(\alpha+\frac{1}{3}-\frac{2\xi}{3})\frac{6k}{2k-3-4k\xi}}dx\right)^{\frac{2k-3-4k\xi}{6k}} \\
&\leq C + C(T)\left(\int_{0}^{1}v^{(\alpha+\frac{1}{3}-\frac{2\xi}{3})\frac{6k}{2k-3-4k\xi}-1+1}dx\right)^{\frac{2k-3-4k\xi}{6k}} \\
&\leq C(T)V_{T}^{\alpha+\frac{1}{2k}}.
\end{aligned}
\end{equation}
Taking the supremum of the above expression over $[0,1]\times[0,T]$ and applying Young's inequality together with \eqref{w-uhigh}, we obtain \eqref{lowerdensity00}.
\end{proof}

\subsection{A priori estimates: Higher-Order estimates}
\label{sec-4}
In this subsection, we establish higher-order estimates for the local classical solution $(\rho,u,\Phi)$. We always take $N=2$ or $3$ with $\frac{N}{N-1} < \alpha \leq 1$. 
\begin{lema}
Let $N=2,$ $l>1,$ or $N=3,$ $l\in(1.5,3),$ and assume that
\begin{equation}\label{assumption of lem5.1}
 \begin{aligned} &\sup_{0\leq t\leq T} \left( \|\mathbf{u}\|_{L^{2l}} + \|\rho\|_{L^\infty} + \|\rho^{-1}\|_{L^\infty} + \|\nabla\rho\|_{L^2} + \|\nabla\rho\|_{L^{2l}} \right) \\&\qquad + \int_0^T \left( \|\nabla\rho\|_{L^2}^2 + \|\nabla\rho\|_{L^{2l}}^{2l} + \|\nabla\mathbf{u}\|_{L^2}^2 \right) dt \leq C(T). \end{aligned} 
\end{equation}
Then there exists a constant \(C(T) > 0\) such that
\begin{equation}\label{LOW1}
\sup_{0\leq t\leq T} \|\nabla\mathbf{u}\|_{L^2}^2 + \int_0^T \int_\Omega \rho |\dot{\mathbf{u}}|^2 \, \mathrm{d}\mathbf{x} \, dt \leq C(T).
\end{equation}
\end{lema}

\begin{proof}
In the radially symmetric setting, we have \(\mathcal{D}\mathbf{u} = \nabla\mathbf{u}\). Therefore, we rewrite $\eqref{eq:NSP-original}_2$ as
\begin{equation}\label{rewirteeq:NSP-original}
\rho \dot{\mathbf{u}}+\nabla P=\text{div}\left(\mu(\rho)\nabla \mathbf{u}\right)+\nabla(\lambda(\rho)\text{div}\left(\mathbf{u}\right))-\rho \nabla \Phi.
\end{equation}
Next, we test \eqref{rewirteeq:NSP-original} with \(\dot{\mathbf{u}}\) and integrate the result over \(\Omega\). Integrating by parts leads to
\begin{equation}\label{low0}
\begin{aligned}
\int_\Omega \rho |\dot{\mathbf{u}}|^2 \, \mathrm{d}\mathbf{x} =& -\int_\Omega \nabla P \cdot \dot{\mathbf{u}} \, \mathrm{d}\mathbf{x} - \int_\Omega \mu(\rho )\nabla \mathbf{u} : \nabla \dot{\mathbf{u}} \, \mathrm{d}\mathbf{x}\\
&-\int_{\Omega}\lambda(\rho)\text{div}\mathbf{u} \text{div}\dot{\mathbf{u}} \mathrm{d}\mathbf{x} -\int_\Omega\rho\dot{\mathbf{u}}\cdot{\nabla\Phi}\, \mathrm{d}\mathbf{x}\\
\triangleq& Q_1 + Q_2+Q_3+Q_{4}.
\end{aligned}
\end{equation}
We have
\begin{equation}\label{low1}
\begin{aligned}
Q_1 \leq& C(T) \int_\Omega |\nabla \rho| \sqrt{\rho} |\dot{\mathbf{u}}| \, \mathrm{d}\mathbf{x} \\
\leq& \frac{1}{8} \|\sqrt{\rho} \dot{\mathbf{u}}\|_{L^2}^2 + C(T) \|\nabla \rho\|_{L^2}^2 .
\end{aligned}
\end{equation}
By combining H\"{o}lder's inequality, Young's inequality and \eqref{est of nabla Phi}, we arrive at
\begin{equation}\label{low4}
\begin{aligned}
Q_4\leq& \frac{1}{8} \|\sqrt{\rho} \dot{\mathbf{u}}\|_{L^2}^2+C(T)\|\nabla\Phi\|^2_{L^2}\\
\leq&\frac{1}{8} \|\sqrt{\rho} \dot{\mathbf{u}}\|_{L^2}^2+C(T)\|\rho-1\|_{L^{\frac{6}{5}}}^2\\
\le& \frac{1}{8} \|\sqrt{\rho} \dot{\mathbf{u}}\|_{L^2}^2+C(T).
\end{aligned}
\end{equation}
From \eqref{eq:NSP-original}$_2$, we see that $\mathbf{u}$ satisfies the elliptic system:
\begin{equation}\label{eq: elliptic sys of u}
\begin{aligned}
\begin{cases}
\Delta \mathbf{u}+(\alpha-1)\nabla \text{div}\left(\mathbf{u}\right)= \frac{1}{\rho^{\alpha}}\left( \rho \dot{\mathbf{u}}+\nabla P-\nabla \mu(\rho)\cdot \nabla \mathbf{u}-\nabla\lambda(\rho)\text{div}\left(\mathbf{u}\right)+\rho \nabla \Phi \right) &\text{ in }\Omega, \\[4pt]
\mathbf{u}=0&\text{ on }\partial\Omega.
\end{cases}
\end{aligned}
\end{equation}
Due to elliptic regularity, H\"{o}lder's inequality,
\begin{equation}\label{eq: u H2}
\begin{aligned}
\left\lVert \mathbf{u} \right\rVert _{H^{2}}&\leq C(T)\left( \|\rho\dot{\mathbf{u}}\|_{L^2}+\|\nabla\rho\|_{L^2}+\||\nabla\rho||\nabla \mathbf{u}|\|_{L^2}+\|\nabla\Phi\|_{L^2} \right)  \\
&\leq C(T)\left( \left\lVert \sqrt{ \rho }\dot{\mathbf{u}} \right\rVert _{L^{2}}+ \left\lVert \nabla\rho \right\rVert _{L^{2}}+\left\lVert \nabla \rho \right\rVert_{L^{2l}}\left\lVert \nabla \mathbf{u} \right\rVert _{L^{\frac{2l}{l-1}}}+\left\lVert \nabla \Phi \right\rVert_{L^{2}}  \right) \\
&\leq C(T)\left(\left\lVert \sqrt{ \rho }\dot{\mathbf{u}} \right\rVert _{L^{2}}+ 1+\left\lVert \nabla \mathbf{u} \right\rVert _{L^{2}}^{1- \frac{N}{2l}}\left\lVert \nabla\mathbf{u} \right\rVert_{H^{1}}^{\frac{N}{2l}}   \right) \\
&\leq \frac{1}{2}\left\lVert \mathbf{u} \right\rVert _{H^{2}}+ C(T)\left( \left\lVert \sqrt{ \rho }  \dot{\mathbf{u}}\right\rVert_{L^{2}}+\left\lVert \nabla \mathbf{u} \right\rVert_{L^{2}} +1 \right) .
\end{aligned}
\end{equation}
Employing \eqref{assumption of lem5.1} together with \eqref{eq: u H2}, we deduce that
\allowdisplaybreaks
\begin{align}\label{low2}
Q_{2}&= -\int_{\Omega}\mu(\rho)\nabla \mathbf{u}:\nabla \dot{\mathbf{u}}\mathrm{d}\mathbf{x} \nonumber\\
&=- \frac{1}{2}\frac{d}{dt}\int_{\Omega}\mu(\rho)\left| \nabla \mathbf{u} \right| ^{2}\mathrm{d}\mathbf{x}
-  \frac{1}{2}\int_{\Omega}\text{div}\left(\mu(\rho)\mathbf{u}\right)\left| \nabla \mathbf{u} \right| ^{2}\mathrm{d}\mathbf{x} \nonumber\\
&\quad -  \frac{1}{2}(\alpha-1)\int_{\Omega}\mu(\rho)\text{div}\left(\mathbf{u}\right)\left| \nabla \mathbf{u} \right| ^{2}\mathrm{d}\mathbf{x}
-\int_{\Omega}\mu(\rho)\nabla \mathbf{u}:\nabla(\mathbf{u}\cdot \nabla \mathbf{u})\mathrm{d}\mathbf{x} \nonumber\\
&\leq - \frac{1}{2}\frac{d}{dt}\int_{\Omega}\mu(\rho)\left| \nabla \mathbf{u} \right| ^{2}\mathrm{d}\mathbf{x}
+C(T)\int_{\Omega}|\nabla \rho||\mathbf{u}||\nabla \mathbf{u}|^{2}\mathrm{d}\mathbf{x} \nonumber\\
&\qquad +C(T)\int_{\Omega}|\mathbf{u}|\left| \nabla \mathbf{u} \right| \left| \nabla^{2}\mathbf{u} \right| \mathrm{d}\mathbf{x} \nonumber\\
&\leq - \frac{1}{2}\frac{d}{dt}\int_{\Omega}\mu(\rho)\left| \nabla \mathbf{u} \right| ^{2}\mathrm{d}\mathbf{x}
+C(T)\left\lVert \nabla \rho \right\rVert _{L^{2l}}\left\lVert \mathbf{u} \right\rVert _{L^{2l}}\left\lVert \nabla \mathbf{u} \right\rVert ^{2}_{L^{\frac{2l}{l-1}}} \nonumber\\
&\qquad +C(T)\left\lVert \mathbf{u} \right\rVert _{L^{2l}}\left\lVert \nabla\mathbf{u} \right\rVert _{L^{\frac{2l}{l-1}}}\left\lVert \nabla^{2}\mathbf{u} \right\rVert_{L^{2}} \nonumber\\
&\leq - \frac{1}{2}\frac{d}{dt}\int_{\Omega}\mu(\rho)\left| \nabla \mathbf{u} \right| ^{2}\mathrm{d}\mathbf{x}
+ \frac{1}{8}\left\lVert \sqrt{ \rho }\dot{\mathbf{u}} \right\rVert _{L^{2}}^{2}
+ C(T)\left( \left\lVert \nabla\mathbf{u} \right\rVert^{2}_{L^{2}} +1 \right).
\end{align}
Similarly, we have
\begin{equation}\label{low3}
    \begin{aligned}
Q_{3}&\leq - \frac{1}{2} \frac{d}{dt}\int_{\Omega}\lambda(\rho)(\text{div}\mathbf{u})^{2}\mathrm{d}\mathbf{x}+ \frac{1}{8}\left\lVert \sqrt{ \rho }\dot{\mathbf{u}} \right\rVert^{2}_{L^{2}}+ C(T)\left( \left\lVert \nabla\mathbf{u} \right\rVert^{2}_{L^{2}} +1 \right) .\end{aligned}
\end{equation}
Inserting \eqref{low1}, \eqref{low4}, \eqref{low2} and \eqref{low3} into \eqref{low0} and integrating in time yields 
\begin{equation}
    \int_{\Omega}\mu(\rho)|\nabla \mathbf{u}|^{2}\mathrm{d}\mathbf{x}+ \int_{\Omega}\lambda(\rho)\text{div}\left(\mathbf{u}\right)^{2}\mathrm{d}\mathbf{x}+ \int^{T}_{0}\int_{\Omega}\rho \left| \dot{\mathbf{u}} \right| ^{2}\mathrm{d}\mathbf{x}dt\leq C(T).
\end{equation}
We note that $\left| \text{div}\mathbf{u}\right|\leq \sqrt{ N }\left| \nabla \mathbf{u} \right|$ and $\alpha\in\left( \frac{N-1}{N},1 \right),$ which completes the proof.

\end{proof}

\begin{lema}\label{lemm4.14}
     Under the assumption of \eqref{assumption of lem5.1}, there exists a constant $C(T)>0$ such that
     \begin{equation}\label{LOW2}
\sup_{0\leq t\leq T} \left( \|\sqrt{\rho} \dot{\mathbf{u}}\|_{L^2}^2 + \|\mathbf{u}\|_{H^2}^2 \right) + \int_0^T \int_\Omega |\nabla \dot{\mathbf{u}}|^2 \, \mathrm{d}\mathbf{x} \, dt \leq C(T).
\end{equation}
\end{lema}
\begin{proof}
Following Hoff \cite{HOFF-1995}, we apply $\dot{u}_j[\partial_t + \operatorname{div}(\mathbf{u},\cdot)]$ to \eqref{rewirteeq:NSP-original}, sum over $j$, integrate over $\Omega$, and integrate by parts. This leads to
\allowdisplaybreaks
\begin{align} \label{nfin}
\Big( \frac{1}{2} \int_{\Omega} \rho |\dot{\mathbf{u}}|^{2} \mathrm{d}\mathbf{x} \Big)_{t}
=&- \int_\Omega \dot{\mathbf{u}} \cdot \Big[ \nabla P_t + \text{div}(\mathbf{u} \otimes \nabla P) \Big] \mathrm{d}\mathbf{x}  \nonumber\\
&+ \int_\Omega \dot{\mathbf{u}} \cdot \Big[ \partial_t \big( \text{div}(\mu(\rho)\nabla \mathbf{u}) \big) + \text{div} \Big( \mathbf{u} \otimes \text{div}(\mu(\rho)\nabla \mathbf{u}) \Big) \Big] \mathrm{d}\mathbf{x} \nonumber\\
&+\int_\Omega \dot{\mathbf{u}} \cdot \Big[ \partial_t \big( \nabla(\lambda(\rho)\text{div} \mathbf{u}) \big) + \text{div} \Big( \mathbf{u} \otimes \nabla(\lambda(\rho)\text{div} \mathbf{u}) \Big) \Big] \mathrm{d}\mathbf{x} \nonumber\\
& -\int_{\Omega} \dot{\mathbf{u}}\cdot{[(\rho\nabla\Phi)_t+\text{div}(\mathbf{u}\otimes (\rho\nabla\Phi))]}\mathrm{d}\mathbf{x}\nonumber\\
\triangleq &N_{1} + N_{2}+N_3+N_{4}. 
\end{align}
Integration by parts to the first term yields
\begin{equation}
\begin{aligned}\label{n1}
N_{1} 
&= - \int_\Omega \dot{ \mathbf{u}} \cdot [\nabla P_t + \text{div}( \mathbf{u} \otimes \nabla P)]  \mathrm{d}\mathbf{x} \\
&= \int_{\Omega} \big[ -P'(\rho) \rho \operatorname{div} \mathbf{u} \, \partial_{j} \dot{u}_{j} + \partial_{k} (\partial_{j} \dot{u}_{j} u_{k}) P - P \partial_{j} (\partial_{k} \dot{u}_{j} u_{k}) \big] \mathrm{d}\mathbf{x}\\
&=  \int_{\Omega}\left[ -P'(\rho)\rho \text{div}\left(\mathbf{u}\right)\text{div}\left(\dot{\mathbf{u}}\right)+P\text{div}\left(\dot{\mathbf{u}}\right)\text{div}\left(\mathbf{u}\right)-P\partial_{k}\dot{\mathbf{u}}_{j}\partial_{j}\mathbf{u}_{k} \right]\mathbf{d}\mathbf{x}\\
&\leq C(T) \|\nabla \mathbf{u}\|_{L^{2}} \|\nabla \dot{\mathbf{u}}\|_{L^{2}} \\
&\leq \varepsilon \|\nabla \dot{\mathbf{u}}\|_{L^{2}}^{2} + C(\varepsilon,T) \|\nabla \mathbf{u}\|_{L^{2}}^{2}. 
\end{aligned}
\end{equation}
Applying integration by parts to the second term gives
\begin{equation}\label{n2}
\begin{aligned}
N_{2}=& \int_{\Omega} \dot{u}_{j}[\partial_{it}(\mu(\rho)\partial_{i}u_{j})+\partial_{k}\left(u_{k}\partial_{i}(\mu(\rho)\partial_{i}u_{j}) \right)]\mathrm{d}\mathbf{x}\\
=&- \int_{\Omega}\partial_{i}\dot{u}_{j} \partial_{t}(\mu(\rho) \partial_{i}u_{j})\mathrm{d}\mathbf{x}-  \int_{\Omega}\partial_{k}\dot{u}_{j}u_{k} \partial_{i}(\mu(\rho)\partial_{i}u_{j})\mathrm{d}\mathbf{x} \\
=& -\int_{\Omega}\mu(\rho)|\nabla \dot{\mathbf{u}}|^{2}\mathrm{d}\mathbf{x}- \int_{\Omega}\partial_{i}\dot{u}_{j}\partial_{t}(\mu(\rho))\partial_{i}u_{j}\mathrm{d}\mathbf{x}+ \int_{\Omega}\mu(\rho)\partial_{i}\dot{u}_{j}\partial_{i}(u_{k}\partial_{k}u_{j})\mathrm{d}\mathbf{x}\\
&-\int_{\Omega}\partial_{k}\dot{u}_{j}u_{k}\mu(\rho)\partial_{ii}u_{j}\mathrm{d}\mathbf{x}-\int_{\Omega}\partial_{k}\dot{u}_{j}u_{k}\partial_{i}(\mu(\rho))\partial_{i}u_{j}\mathrm{d}\mathbf{x}\\
\leq&-\int_{\Omega}\mu(\rho)\left| \nabla \dot{\mathbf{u}} \right| ^{2}\mathrm{d}\mathbf{x}+C(T)\int_{\Omega}\left( \left| \nabla \dot{\mathbf{u}} \right|\left| \nabla \mathbf{u} \right| ^{2}+\left| \nabla \dot{\mathbf{u}} \right|  \left| \mathbf{u} \right| \left| \nabla^{2}\mathbf{u} \right| +\left| \nabla \dot{\mathbf{u}} \right| \left| \mathbf{u} \right| \left| \nabla \mathbf{u} \right| \left| \nabla \rho \right|  \right) \mathrm{d}\mathbf{x}\\
\leq &-\int_{\Omega}\mu(\rho)\left| \nabla \dot{\mathbf{u}} \right| ^{2}\mathrm{d}\mathbf{x}+\varepsilon \left\lVert \nabla \dot{\mathbf{u}} \right\rVert _{L^{2}}\\
&+C(\varepsilon,T)\left( \left\lVert \left| \mathbf{u} \right| \left| \nabla^{2}\mathbf{u} \right|  \right\rVert^{2}_{L^{2}}+\left\lVert \nabla \mathbf{u} \right\rVert ^{4}_{L^{4}}+\left\lVert \left| \nabla \rho \right| \left| \mathbf{u} \right| \left| \nabla \mathbf{u} \right|  \right\rVert  _{L^{2}} \right) .
\end{aligned}
\end{equation}
From \eqref{eq: elliptic sys of u} and \eqref{assumption of lem5.1} we obtain 
\begin{equation}\label{est: u W2 2l}
    \begin{aligned}
\left\lVert \mathbf{u} \right\rVert _{W^{2,2l}}&\leq C(T)\left( \|\rho\dot{\mathbf{u}}\|_{L^{2l}}+\|\nabla\rho\|_{L^{2l}}+\||\nabla\rho||\nabla \mathbf{u}|\|_{L^{2l}}+\|\rho\nabla\Phi\|_{L^{2l}} \right) \\
&\leq C(T)\left( \left\lVert \sqrt{ \rho }\dot{\mathbf{u}} \right\rVert_{L^{2l}}+\left\lVert \nabla \rho \right\rVert_{L^{2l}} +\left\lVert \nabla \rho \right\rVert_{L^{2l}}\left\lVert \nabla \mathbf{u} \right\rVert _{L^{\infty}}+\left\lVert \rho \right\rVert _{L^{\infty}}\left\lVert  \Phi \right\rVert _{W^{2,2l}}   \right) \\
&\leq C(T)\left( \left\lVert \sqrt{ \rho } \dot{\mathbf{u}}\right\rVert_{L^{2l}}+\left\lVert \nabla \rho \right\rVert _{L^{2l}}+\left\lVert \nabla \mathbf{u} \right\rVert _{L^{2}}^{\frac{2l-N}{2l+Nl-N}}\left\lVert \nabla \mathbf{u} \right\rVert ^{\frac{Nl}{2l+Nl-N}} _{W^{1,2l}}+\left\lVert \rho-1\right\rVert _{L^{2l}} \right) \\
&\leq \frac{1}{2}\left\lVert \mathbf{u} \right\rVert _{W^{2,2l}}+C(T)\left( \left\lVert \sqrt{ \rho }\dot{\mathbf{u}} \right\rVert _{L^{2l}}+\left\lVert \nabla \rho \right\rVert _{L^{2l}}+\left\lVert \nabla \mathbf{u} \right\rVert _{L^{2}} +1\right) .
\end{aligned}
\end{equation}
Thus we have 
\begin{equation}\label{eq: N2_1}
    \begin{aligned}
C(\varepsilon,T)\left\lVert \left| \mathbf{u} \right| \left| \nabla^{2}\mathbf{u} \right|  \right\rVert ^{2}_{L^{2}}&\leq C(\varepsilon,T)\left\lVert \mathbf{u} \right\rVert ^{2}_{L^{\frac{2l}{l-1}}}\left\lVert \nabla^{2}\mathbf{u} \right\rVert ^{2}_{L^{2l}}\\
&\leq C(\varepsilon,T)\left\lVert \mathbf{u} \right\rVert ^{2}_{L^{\frac{2l}{l-1}}}\left\lVert \mathbf{u} \right\rVert ^{2}_{W^{2,2l}}\\
&\leq C(\varepsilon,T)\left\lVert \mathbf{u} \right\rVert ^{2}_{L^{\frac{2l}{l-1}}}\left( \left\lVert \sqrt{ \rho }\dot{\mathbf{u}} \right\rVert _{L^{2l}}^{2}+\left\lVert \nabla \rho \right\rVert _{L^{2l}}^{2}+\left\lVert \nabla \mathbf{u} \right\rVert _{L^{2}} ^{2}+1\right) \\
&\leq C(\varepsilon,T)\left\lVert \nabla \mathbf{u} \right\rVert ^{2}_{L^{2}}\left( \left\lVert \sqrt{ \rho }\dot{\mathbf{u}} \right\rVert^{\frac{N+2l-Nl}{l}}_{L^{2}}\left\lVert \nabla \dot{\mathbf{u}} \right\rVert^{\frac{Nl-N}{l}}_{L^{2}}+1  \right) \\
&\leq \varepsilon \left\lVert \nabla \dot{\mathbf{u}} \right\rVert ^{2}_{L^{2}}+C(\varepsilon,T)\left( \left\lVert \sqrt{ \rho } \dot{\mathbf{u}}\right\rVert ^{2}_{L^{2}}+1 \right),
\end{aligned}
\end{equation} 
where we use $\|\mathbf{u}\|_{L^{\frac{2l}{l-1}}}\le C\|\nabla \mathbf{u}\|_{L^2}$ for $\mathbf{u}$ with zero boundary condition. From \eqref{assumption of lem5.1}, \eqref{LOW1} and \eqref{eq: u H2} we obtain 
\begin{equation}
    \begin{aligned}
C(\varepsilon,T)\left\lVert \nabla \mathbf{u} \right\rVert ^{4}_{L^{4}}&\leq C(\varepsilon,T)\left\lVert \nabla \mathbf{u} \right\rVert ^{4}_{H^{1}}\\
&\leq C(\varepsilon,T)(\left\lVert \sqrt{ \rho } \dot{\mathbf{u}}\right\rVert ^{4}_{L^{2}}+\left\lVert \nabla \mathbf{u} \right\rVert ^{4}_{L^{2}}+1)\\
&\leq C(\varepsilon,T)(\left\lVert \sqrt{ \rho } \dot{\mathbf{u}}\right\rVert ^{4}_{L^{2}}+\left\lVert \nabla \mathbf{u} \right\rVert ^{2}_{L^{2}}+1).
\end{aligned}
\end{equation}
Analogously to \eqref{eq: N2_1}, it holds that
\begin{equation}\label{eq: N2_3}
    \begin{aligned}
C(\varepsilon,T)\left\lVert \left| \nabla \rho \right| \left| \mathbf{u} \right| \left| \nabla \mathbf{u} \right|  \right\rVert  _{L^{2}} &\leq C(\varepsilon,T)\left\lVert \nabla \rho \right\rVert^{2}_{L^{2l}}\left\lVert \mathbf{u} \right\rVert ^{2}_{L^{\frac{2l}{l-1}}}\left\lVert \nabla \mathbf{u} \right\rVert ^{2}_{L^{\infty}}\\
&\leq C(\varepsilon,T)\left\lVert \mathbf{u} \right\rVert ^{2}_{L^{\frac{2l}{l-1}}}\left\lVert \nabla \mathbf{u} \right\rVert ^{2}_{W^{1,2l}}\\
&\leq \varepsilon \left\lVert \nabla \dot{\mathbf{u}} \right\rVert ^{2}_{L^{2}}+ C(\varepsilon,T)\left( \left\lVert \sqrt{ \rho }\dot{\mathbf{u}} \right\rVert _{L^{2}}^{2}+1\right) .
\end{aligned}
\end{equation}
Substituting \eqref{eq: N2_1}-\eqref{eq: N2_3} into \eqref{n2}, one has 
\begin{equation}
    \begin{aligned}
N_{2}\leq & -\int_{\Omega}\mu(\rho)\left| \nabla \dot{\mathbf{u}} \right| ^{2}\mathrm{d}\mathbf{x}+3\varepsilon \left\lVert \nabla \dot{\mathbf{u}} \right\rVert ^{2}_{L^{2}}+C(\varepsilon,T)\left\lVert \sqrt{ \rho }\dot{\mathbf{u}} \right\rVert ^{4}_{L^{2}}\\&+C(\varepsilon,T)\left( \left\lVert \sqrt{ \rho } \dot{\mathbf{u}}\right\rVert^{2}_{L^{2}}+\left\lVert \nabla \mathbf{u} \right\rVert ^{2}_{L^{2}}+1  \right) .
\end{aligned}
\end{equation}
In a similar way, we arrive at
\begin{equation}
    \begin{aligned}
N_{3}\leq& -\int_{\Omega}\lambda(\rho)(\text{div}\left(\dot{\mathbf{u}}\right))^{2}\mathrm{d}\mathbf{x}+\varepsilon \left\lVert \nabla \dot{\mathbf{u}} \right\rVert _{L^{2}}\mathrm{d}\mathbf{x}\\&+C(\varepsilon,T)\left( \left\lVert \left| \mathbf{u} \right| \left| \nabla^{2}\mathbf{u} \right|  \right\rVert^{2}_{L^{2}}+\left\lVert \nabla \mathbf{u} \right\rVert ^{4}_{L^{4}}+\left\lVert \left| \nabla \rho \right| \left| \mathbf{u} \right| \left| \nabla \mathbf{u} \right|  \right\rVert  _{L^{2}} \right) \\
\leq& -\int_{\Omega}\lambda(\rho)(\text{div}\left(\dot{\mathbf{u}}\right))^{2}\mathrm{d}\mathbf{x}+3\varepsilon \left\lVert \nabla \dot{\mathbf{u}} \right\rVert ^{2}_{L^{2}}\mathrm{d}\mathbf{x}\\&+C(\varepsilon,T)\left\lVert \sqrt{ \rho }\dot{\mathbf{u}} \right\rVert ^{4}_{L^{2}}+C(\varepsilon,T)\left( \left\lVert \sqrt{ \rho } \dot{\mathbf{u}}\right\rVert^{2}_{L^{2}}+\left\lVert \nabla \mathbf{u} \right\rVert ^{2}_{L^{2}}+1  \right) .
\end{aligned}
\end{equation}
Applying H\"{o}lder's inequality, Young's inequality, \eqref{assumption of lem5.1} and \eqref{LOW1} yields
\begin{equation}\label{n4}
\begin{aligned}
N_{4}&= -\int_{\Omega}\dot{u}_{j}\left[ \partial_{t}(\rho \partial_{j}\Phi)+\partial_{k}(u_{k}\rho \partial_{j}\Phi) \right] \mathrm{d}\mathbf{x}\\
&= -\int_{\Omega}\rho \dot{\mathbf{u}}\cdot \left( \nabla \Phi_{t}+\mathbf{u}\cdot \nabla^{2}\Phi \right) \mathrm{d}\mathbf{x}\\
&\leq C(T)\left( \left\lVert \sqrt{ \rho }\dot{\mathbf{u}} \right\rVert _{L^{2}}\left\lVert \nabla \Phi_{t} \right\rVert _{L^{2}}+\left\lVert \nabla^{2}\Phi  \right\rVert _{L^{2}}\left\lVert \dot{\mathbf{u}} \right\rVert _{L^{4}}\left\lVert \mathbf{u} \right\rVert _{L^{4}}\right) \\
&\leq C(T)\left( \left\lVert \sqrt{ \rho } \dot{\mathbf{u}}\right\rVert ^{2}_{L^{2}} +1+\left\lVert \rho-1 \right\rVert_{L^{2}}\left\lVert  \sqrt{ \rho } \dot{\mathbf{u}} \right\rVert_{L^{2}}^{\frac{4-N}{4}}\left\lVert \nabla \dot{\mathbf{u}} \right\rVert _{L^{2}}^{\frac{N}{4}} \left\lVert \nabla \mathbf{u} \right\rVert_{L^{2}} \right)\\
&\leq \varepsilon \left\lVert \nabla \dot{\mathbf{u}} \right\rVert_{L^{2}}^{2} +C(\varepsilon,T)\left( \left\lVert \sqrt{ \rho } \dot{\mathbf{u}}\right\rVert^{2}_{L^{2}}+1  \right) .
\end{aligned}
\end{equation}
Here we note that 
\begin{equation}
\begin{cases}
\Delta \Phi_{t} 
=-\kappa\text{div}\left(\rho \mathbf{u}\right)&\text{ in }\Omega, \\
\frac{\partial \Phi_{t}}{\partial n}=0&\text{ on }\partial\Omega.
\end{cases}
\end{equation}
For $1<q<\infty$, we arrive at
\begin{equation}\label{est: nabla phi_t}
\begin{aligned}
\|\nabla\Phi_t\|_{L^q}\leq C\|\rho\mathbf{u}\|_{L^q}.
\end{aligned}
\end{equation}
Inserting the estimates \eqref{n1}-\eqref{n4} into \eqref{nfin} produces
\begin{equation}\label{na}
 \begin{aligned} &\frac{1}{2}\frac{d}{dt}\int_{\Omega} \rho |\dot{\mathbf{u}}|^2 \mathrm{d}\mathbf{x}+(\left\lVert \rho^{-1} \right\rVert _{L^{\infty}}^{-\alpha}(1+(\alpha-1)N)-8\varepsilon)\left\lVert \nabla \dot{\mathbf{u}} \right\rVert ^{2}_{L^{2}}\\
\leq&\frac{1}{2}\frac{d}{dt}\int_{\Omega} \rho |\dot{\mathbf{u}}|^2 \mathrm{d}\mathbf{x} + \int_{\Omega} \left( \mu(\rho) |\nabla \dot{\mathbf{u}}|^2 + \lambda(\rho) (\mathrm{div}\, \dot{\mathbf{u}})^2 \right) \mathrm{d}\mathbf{x} - 8\varepsilon \|\nabla \dot{\mathbf{u}}\|_{L^2}^2 \\ \leq &C(\varepsilon,T)\left( \|\sqrt{\rho}\dot{\mathbf{u}}\|_{L^2}^2 + \|\nabla \rho\|_{L^2}^2  + \|\nabla \mathbf{u}\|_{L^2}^2 +1\right) + C(\varepsilon,T) \|\sqrt{\rho}\dot{\mathbf{u}}\|_{L^2}^4. \end{aligned} 
\end{equation}
Choose $\varepsilon\in(0,\frac{1}{16}\left\lVert \rho^{-1} \right\rVert _{L^{\infty}}^{-\alpha}(1+(\alpha-1)N))$ sufficiently small. Finally, applying Gronwall's inequality to \eqref{na} yields \eqref{LOW2}, which completes the proof of Lemma \ref{lemm4.14}.
\end{proof}
\begin{lema}
Under the assumption of \eqref{assumption of lem5.1}, there exists a constant $C(T)>0$ such that
\begin{equation}\label{high}
\sup_{0\leq t\leq T} (\|\rho\|_{H^2}^2 +\|\nabla\Phi\|^2_{H^3} ) + \int_0^T \|\mathbf{u}\|_{H^3}^2 \, dt \leq C(T).
\end{equation}
\end{lema}
\begin{proof}
We begin by applying the  operator $\partial_j \partial_i$  to the equation $\eqref{eq:NSP-original}_1$
\begin{equation}
\begin{aligned}
&\partial_j \partial_i 
\rho_t + \partial_j \partial_i \mathbf{u}\cdot \nabla \rho + \mathbf{u} \cdot \nabla \partial_j \partial_i \rho + \partial_j \mathbf{u} \cdot \nabla \partial_i \rho + \partial_i \mathbf{u} \cdot \nabla \partial_j \rho \\
&+ \partial_j \partial_i \rho \div \mathbf{u} + \rho \partial_j \partial_i \div \mathbf{u} + \partial_i \rho \partial_j \div \mathbf{u} + \partial_j \rho \partial_i \div \mathbf{u} = 0.
\end{aligned}
\end{equation}
We then test the equation against $2\partial_j\partial_i \rho$ and integrate over $\Omega$. After performing an integration by parts, this yields
 \allowdisplaybreaks
\begin{align} \label{rho2}
\frac{\mathrm{d}}{\mathrm{d}t} \int_{\Omega} |\partial_j \partial_i \rho|^2 \,\mathrm{d}\mathbf{x}
&= - \int_{\Omega} |\partial_j \partial_i \rho|^2 \, \mathrm{div}\boldsymbol{\mathbf{u}} \,\mathrm{d}\mathbf{x}
   - 2 \int_{\Omega} \partial_j \partial_i \rho \, \partial_i \boldsymbol{\mathbf{u}} \cdot \nabla \partial_j \rho \,\mathrm{d}\mathbf{x} \notag\\
&\quad - 2 \int_{\Omega} \partial_j \partial_i \rho \, \partial_j \boldsymbol{\mathbf{u}} \cdot \nabla \partial_i \rho \,\mathrm{d}\mathbf{x}
   - 2 \int_{\Omega} \rho \, \partial_j \partial_i \rho \, \partial_j \partial_i \mathrm{div}\boldsymbol{\mathbf{u}} \,\mathrm{d}\mathbf{x} \notag\\
&\quad - 2 \int_{\Omega} \partial_i \rho \, \partial_j \partial_i \rho \, \partial_j \mathrm{div}\boldsymbol{\mathbf{u}} \,\mathrm{d}\mathbf{x}
   - 2 \int_{\Omega} \partial_j \rho \, \partial_j \partial_i \rho \, \partial_i \mathrm{div}\boldsymbol{\mathbf{u}} \,\mathrm{d}\mathbf{x} \notag\\
&\quad - 2 \int_{\Omega} \partial_j \partial_i \rho \, \partial_j \partial_i \boldsymbol{\mathbf{u}} \cdot \nabla \rho \,\mathrm{d}\mathbf{x} \notag\\
&\leq C(T)(1 + \|\nabla \mathbf{u}\|_{L^\infty}) \left\|\nabla^2 \rho\right\|_{L^2}^2 + C(T) \left\||\nabla \rho| |\nabla^2 \mathbf{u}|\right\|_{L^2}^{2} + C(T) \left\|\nabla^3 \mathbf{u}\right\|_{L^2}^2 \notag\\ 
&\leq C(T)(1 + \|\nabla \mathbf{u}\|_{L^\infty}) \left\|\nabla^2 \rho\right\|_{L^2}^2 + C(T) \|\nabla \rho\|_{L^{2l}}^2 \|\nabla^2 \mathbf{u}\|_{L^{\frac{2l}{l-1}}}^2 + C(T) \left\|\nabla^3 \mathbf{u}\right\|_{L^2}^2 \notag\\
&\leq C(T)(1 + \|\nabla \mathbf{u}\|_{L^\infty}) \left\|\nabla^2 \rho\right\|_{L^2}^2 + C(T) \left\|\nabla^3 \mathbf{u}\right\|_{L^2}^2+ C(T).
\end{align}
Therefore, we need to handle $\|\nabla \mathbf{u}\|_{L^\infty}$ and $\left\|\nabla^3 \mathbf{u}\right\|_{L^2}$. First, from \eqref{est: u W2 2l} we obtain 
\begin{equation}\label{nabla u L infty}
    \begin{aligned} \|\nabla \mathbf{u}\|_{L^\infty} &\leq C \|\nabla \mathbf{u}\|_{W^{1,2l}} \leq C(T)\left( \|\sqrt{\rho} \dot{\mathbf{u}}\|_{L^{2l}} + \|\nabla \rho\|_{L^{2l}} + \|\nabla \mathbf{u}\|_{L^2} +1\right) \\
 &\leq C(T) + C(T) \|\nabla \dot{\mathbf{u}}\|_{L^2}. \end{aligned}
\end{equation}
Applying the $H^1$ estimate for the elliptic system \eqref{eq: elliptic sys of u} again, we arrive at
\begin{equation}\label{uh3}
    \begin{aligned}
\|\mathbf{u}\|_{H^3}^2 &\leq C \left\| \frac{1}{\rho^{\alpha}}\left( \rho \dot{\mathbf{u}}+\nabla P-\nabla \mu(\rho)\cdot \nabla \mathbf{u}-\nabla\lambda(\rho)\text{div}\left(\mathbf{u}\right)+\rho \nabla \Phi \right)\right\|_{H^1}^2 \\
&\leq C \left( 1 + \left\| |\nabla \rho| \dot{\mathbf{u}} \right\|_{L^2}^2 + \|\nabla \dot{\mathbf{u}}\|_{L^2}^2 + \left\| \nabla^2 \rho \right\|_{L^2}^2 + \|\nabla \rho\|_{L^4}^2 + \left\| |\nabla \rho|^2 \nabla \mathbf{u} \right\|_{L^2}^2 \right. \\
&\quad \left. + \left\| |\nabla^2 \rho| \nabla \mathbf{u} \right\|_{L^2}^2 + \left\| |\nabla \rho| \nabla^2 \mathbf{u} \right\|_{L^2}^2+\left\lVert \left| \nabla \rho \right|  \left| \nabla \Phi \right| \right\rVert_{L^{2}}+\left\lVert \rho \nabla^{2}\Phi \right\rVert _{L^{2}}  \right) \\
&\leq C(T) \left( 1 + \|\nabla \dot{\mathbf{u}}\|_{L^2}^2 \right) \left( 1 + \left\| \nabla^2 \rho \right\|_{L^2}^2 \right) + C(T) \left\| \nabla^2 \mathbf{u} \right\|_{L^{\frac{2l}{l-1}}}^2 \\
&\leq C(T) \left( 1 + \|\nabla \dot{\mathbf{u}}\|_{L^2}^2 \right) \left( 1 + \left\| \nabla^2 \rho \right\|_{L^2}^2 \right) + C(T) \left\| \nabla^2 \mathbf{u} \right\|_{L^2}^{\frac{2l-N}{l}} \left\| \nabla^2 \mathbf{u} \right\|_{H^1}^{\frac{N}{l}} \\
&\leq \frac{1}{2} \|\mathbf{u}\|_{H^3}^2 + C(T) \left( 1 + \|\nabla \dot{\mathbf{u}}\|_{L^2}^2 \right) \left( 1 + \left\| \nabla^2 \rho \right\|_{L^2}^2 \right).
\end{aligned}
\end{equation}
Inserting \eqref{nabla u L infty}, \eqref{uh3} into \eqref{rho2} gives
\begin{equation}
\begin{aligned}
\frac{d}{dt}\left\lVert \nabla^{2}\rho \right\rVert^2 _{L^{2}}\leq C(T)(1+\left\lVert \nabla \dot{\mathbf{u}} \right\rVert _{L^{2}}^{2})\left( 1+\left\lVert \nabla^{2}\rho \right\rVert^{2}_{L^{2}}  \right). 
\end{aligned}
\end{equation}
Consequently, thanks to \eqref{LOW2}, \eqref{psiestimate}, Gronwall's inequality implies \eqref{high}.
 \end{proof}
 \begin{lema}
Under the assumption of \eqref{assumption of lem5.1}, there exists a constant $C(T)>0$ such that
\begin{equation}\label{5.1}
\sup_{0\leq t\leq T} \left( \|\nabla\dot{\mathbf{u}}\|_{L^2}^2 + \|\nabla\mathbf{u}_t\|_{L^2}^2 + \|\mathbf{u}\|_{H^3}^2 \right) + \int_0^T \left( \|\mathbf{u}_{tt}\|_{L^2}^2 + \|\mathbf{u}_t\|_{H^2}^2 \right) dt \leq C(T).
\end{equation}
 \end{lema}
\begin{proof}
Differentiating \eqref{rewirteeq:NSP-original} with respect to \(t\) we arrive at 
\begin{equation}\label{eq: elliptic sys of u_t}
    \begin{aligned}
&-\Delta \mathbf{u}_{t}-(\alpha-1)\nabla \text{div}\left(\mathbf{u}_{t}\right)\\=& -\partial_{t}\left[ \frac{1}{\rho^{\alpha}}\left( \rho \left( \mathbf{u}_{t}+\mathbf{u}\cdot \nabla \mathbf{u} \right) +\nabla P-\nabla \mu(\rho)\cdot \nabla \mathbf{u}-\nabla\lambda(\rho)\text{div}\left(\mathbf{u}\right)+\rho \nabla \Phi \right)  \right].
\end{aligned}
\end{equation}
Multiplying by \(\mathbf{u}_{tt}\), and integrating over \(\Omega\) yields
\begin{equation}\label{high_tmp}
\begin{aligned}
&\frac{1}{2}\frac{d}{dt}\int_{\Omega}\left(|\nabla \mathbf{u}_t|^2+(\alpha-1)(\mathrm{div}\,\mathbf{u}_t)^2\right)\mathrm{d}\mathbf{x}+\int_{\Omega}\rho^{1-\alpha}|\mathbf{u}_{tt}|^2\mathrm{d}\mathbf{x} \\
\leq& C(T)\int_{\Omega}|\mathbf{u}_{tt}|\left(|\rho_t||\mathbf{u}_t|+|\rho_t||\mathbf{u}||\nabla \mathbf{u}|+|\mathbf{u}_t||\nabla \mathbf{u}|+|\mathbf{u}||\nabla \mathbf{u}_t|+|\rho_t||\nabla \rho|+|\nabla \rho_t|\right. \\
&\left.+|\rho_t||\nabla \rho||\nabla \mathbf{u}|+|\nabla \rho_t||\nabla \mathbf{u}|+|\nabla \rho||\nabla \mathbf{u}_t|+\left| \rho_{t} \right| \left| \nabla \Phi \right|+\left| \nabla \Phi_{t} \right|  \right)\mathrm{d}\mathbf{x} \\
\leq &\varepsilon\int_{\Omega}|\mathbf{u}_{tt}|^2\mathrm{d}\mathbf{x}+C(\varepsilon,T)\int_{\Omega}\left(|\rho_t|^2|\mathbf{u}_t|^2+|\rho_t|^2|\mathbf{u}|^2|\nabla \mathbf{u}|^2+|\mathbf{u}_t|^2|\nabla \mathbf{u}|^2+|\mathbf{u}|^2|\nabla \mathbf{u}_t|^2\right. \\
&\left.+|\rho_t|^2|\nabla \rho|^2+|\nabla \rho_t|^2+|\rho_t|^2|\nabla \rho|^2|\nabla \mathbf{u}|^2+|\nabla \rho_t|^2|\nabla \mathbf{u}|^2+|\nabla \rho|^2|\nabla \mathbf{u}_t|^2+\left| \rho_{t} \right| ^{2}\left| \nabla \Phi \right|^{2}+\left| \nabla \Phi_{t} \right| ^{2} \right)\mathrm{d}\mathbf{x} \\
\leq &\varepsilon\int_{\Omega}|\mathbf{u}_{tt}|^2\mathrm{d}\mathbf{x}+C(\varepsilon,T)\left(\|\rho_t\|_{L^6}^2\|\mathbf{u}_t\|_{L^6}^2+\|\rho_t\|_{L^6}^2\|\mathbf{u}\|_{L^\infty}^2\|\nabla \mathbf{u}\|_{L^6}^2+\|\mathbf{u}_t\|_{L^6}^2\|\nabla \mathbf{u}\|_{L^6}^2\right. \\
&+\|\mathbf{u}\|_{L^\infty}^2\|\nabla \mathbf{u}_t\|_{L^2}^2+\|\rho_t\|_{L^6}^2\|\nabla \rho\|_{L^6}^2+\|\nabla \rho_t\|_{L^2}^2+\|\rho_t\|_{L^6}^2\|\nabla \rho\|_{L^6}^2\|\nabla \mathbf{u}\|_{L^6}^2 \\
&\left.+\|\nabla \rho_t\|_{L^2}^2\|\nabla \mathbf{u}\|_{L^\infty}^2+\|\nabla \rho\|_{L^6}^2\|\nabla \mathbf{u}_t\|_{L^3}^2+\left\lVert \rho_{t} \right\rVert^{2}_{L^{6}}\left\lVert \nabla \Phi \right\rVert ^{2}_{L^{3}} +\left\lVert \nabla \Phi_{t} \right\rVert ^{2}_{L^{2}}\right).
\end{aligned}
\end{equation}
By virtue of $\eqref{eq:NSP-original}_1,$ $\eqref{eq:NSP-original}_3,$ \eqref{LOW2}, \eqref{high}, \eqref{est: nabla phi_t}, H\"{o}lder's inequality and Sobolev embedding, we deduce that
\begin{equation}
\begin{aligned}
\|\rho_t\|_{L^6}+\|\nabla\rho_t\|_{L^2}+\|\nabla\Phi\|_{L^3}+\|\nabla\Phi_t\|_{L^2}\leq C(T).
\end{aligned}
\end{equation} 
Since $u_t$ satisfies zero boundary condition,  it follows from the elliptic estimate of \eqref{eq: elliptic sys of u_t} that 
\begin{equation} \label{uth2}
\begin{aligned}
\|\mathbf{u}_t\|_{H^2}^2 \leq& C\left\|\partial_{t}\left[ \frac{1}{\rho^{\alpha}}\left( \rho \left( \mathbf{u}_{t}+\mathbf{u}\cdot \nabla \mathbf{u} \right) +\nabla P-\nabla \mu(\rho)\cdot \nabla \mathbf{u}-\nabla\lambda(\rho)\text{div}\left(\mathbf{u}\right)+\rho \nabla \Phi \right)  \right]\right\|_{L^2}^2 \\
\leq &C(T)\left(\|\rho_t\mathbf{u}_t\|_{L^2}^2+\|\mathbf{u}_{tt}\|_{L^2}^2+\||\rho_t\mathbf{u}||\nabla\mathbf{u}|\|_{L^2}^2+\||\mathbf{u}_t||\nabla\mathbf{u}|\|_{L^2}^2+\||\mathbf{u}||\nabla\mathbf{u}_t|\|_{L^2}^2\right. \\
&\left.+\|\rho_t\nabla\rho\|_{L^2}^2+\|\nabla\rho_t\|_{L^2}^2+\||\rho_t\nabla\rho||\nabla\mathbf{u}|\|_{L^2}^2+\||\nabla\rho_t||\nabla\mathbf{u}|\|_{L^2}^2+\||\nabla\rho||\nabla\mathbf{u}_t|\|_{L^2}^2\right.\\
&+\left.\left\lVert \left| \rho_{t} \right| \left| \nabla \Phi \right|  \right\rVert^{2}_{L^{2}}+\left\lVert \nabla \Phi_{t} \right\rVert^{2}_{L^{2}}  \right) \\
\leq &C(T)\left(\|\rho_t\|_{L^6}^2\|\mathbf{u}_t\|_{L^6}^2+\|\mathbf{u}_{tt}\|_{L^2}^2+\|\rho_t\|_{L^6}^2\|\mathbf{u}\|_{L^\infty}^2\|\nabla\mathbf{u}\|_{L^6}^2+\|\mathbf{u}_t\|_{L^6}^2\|\nabla\mathbf{u}\|_{L^6}^2\right. \\
&+\|\mathbf{u}\|_{L^\infty}^2\|\nabla\mathbf{u}_t\|_{L^2}^2+\|\rho_t\|_{L^6}^2\|\nabla\rho\|_{L^6}^2+\|\nabla\rho_t\|_{L^2}^2+\|\rho_t\|_{L^6}^2\|\nabla\rho\|_{L^6}^2\|\nabla\mathbf{u}\|_{L^6}^2 \\
&+\left.\|\nabla\rho_t\|_{L^2}^2\|\nabla\mathbf{u}\|_{L^\infty}^2+\|\nabla\rho\|_{L^6}^2\|\nabla\mathbf{u}_t\|_{L^3}^2 +\left\lVert \rho_{t} \right\rVert_{L^{6}}^{2}\left\lVert \nabla \Phi \right\rVert^{2}_{L^{3}}+\left\lVert \nabla \Phi_{t} \right\rVert ^{2}_{L^{2}}  \right)\\
\leq &C(T)\left(1+\|\nabla\dot{\mathbf{u}}\|_{L^2}^2+\|\mathbf{u}_{tt}\|_{L^2}^2+\|\nabla\mathbf{u}_t\|_{L^3}^2\right) \\
\leq &C(T)\left(1+\|\nabla\dot{\mathbf{u}}\|_{L^2}^2+\|\mathbf{u}_{tt}\|_{L^2}^2\right)+C(T)\|\nabla\mathbf{u}_t\|_{L^2}^{2-\frac{N}{3}}\|\nabla\mathbf{u}_t\|_{H^1}^{\frac{N}{3}} \\
\leq &\frac{1}{2}\|\nabla\mathbf{u}_t\|_{H^1}^2+C(T)\left(1+\|\nabla\dot{\mathbf{u}}\|_{L^2}^2+\|\mathbf{u}_{tt}\|_{L^2}^2\right),
\end{aligned}
\end{equation}
where we use 
\begin{equation}\label{eq: nabla u_t L2}
    \begin{aligned} \|\nabla \mathbf{u}_t\|_{L^2}^2 \leq C\|\nabla \dot{\mathbf{u}}\|_{L^2}^2 + C\|\nabla(\mathbf{u}\cdot\nabla\mathbf{u})\|_{L^2}^2 \leq C\|\nabla \dot{\mathbf{u}}\|_{L^2}^2 + C(T), \end{aligned}
\end{equation}
 and
 \begin{equation}\label{eq: nabla u dot L2}
     \begin{aligned}
\|\nabla \dot{\mathbf{u}}\|_{L^2}^2 \leq C\|\nabla \mathbf{u}_t\|_{L^2}^2 + C\|\nabla(\mathbf{u}\cdot\nabla\mathbf{u})\|_{L^2}^2 \leq C\|\nabla \mathbf{u}_t\|_{L^2}^2 + C(T).
\end{aligned}
 \end{equation}
substituting \eqref{uth2}, \eqref{eq: nabla u_t L2}, \eqref{eq: nabla u dot L2} into \eqref{high_tmp} yields
\begin{equation}\label{prop5.1 tmp}
\begin{aligned}
&\frac{1}{2}\frac{d}{dt}\int_{\Omega}\left(|\nabla \mathbf{u}_t|^2+(\alpha-1)(\mathrm{div}\,\mathbf{u}_t)^2\right)\mathrm{d}\mathbf{x}+\int_{\Omega}\rho^{1-\alpha}|\mathbf{u}_{tt}|^2\mathrm{d}\mathbf{x}\\
\leq& \varepsilon \left\lVert \mathbf{u}_{tt} \right\rVert ^{2}_{L^{2}}+C(\varepsilon,T)\left( 1+\left\lVert \nabla \dot{\mathbf{u} }\right\rVert  ^{2}_{L^{2}}+\left\lVert \nabla \mathbf{u}_{t} \right\rVert^{2}_{L^{3}} \right)\\
\leq& \varepsilon \left\lVert \mathbf{u}_{tt} \right\rVert  ^{2}_{L^{2}}+C(\varepsilon,T)\left( 1+\left\lVert \nabla \dot{\mathbf{u}} \right\rVert ^{2}_{L^{2}}+\left\lVert \nabla \mathbf{u}_{t} \right\rVert _{L^{2}}^{2- \frac{N}{3}} \left\lVert \nabla \mathbf{u}_{t} \right\rVert _{H^{1}}^{\frac{N}{3}}\right) \\
\leq&2\varepsilon \left\lVert \mathbf{u}_{tt} \right\rVert ^{2}_{L^{2}}+C(\varepsilon,T)\left(\left\lVert \nabla \dot{\mathbf{u}} \right\rVert _{L^{2}}^{2}+1\right).
\end{aligned}
\end{equation}
Integrating \eqref{prop5.1 tmp} with respect to time and choosing $\varepsilon\in\left(0, \frac{1}{4}\|\rho^{-1}\|_{L^{\infty}}^{\alpha-1}\right),$ we derive \eqref{5.1} from \eqref{LOW2}, \eqref{uh3}, \eqref{eq: nabla u_t L2}, \eqref{eq: nabla u dot L2} and \eqref{uth2}.
\end{proof}

\begin{prop}
Under the assumption of \eqref{assumption of lem5.1}, there exists a constant $C(T)>0$ such that
\begin{equation}\label{5.2}
\begin{aligned}
&\sup_{0\leq t\leq T} \left( \|\rho\|_{H^3}^2 + \|\nabla\Phi\|_{H^4}^2+\|\rho_t\|_{H^2}^2 + \|\nabla\Phi_t\|_{H^3}^2 +\|\rho_{tt}\|_{L^2}^2+\|\nabla\Phi_{tt}\|^2_{H^1} \right) \\
&+ \int_0^T \left( \|\mathbf{u}\|_{H^4}^2 + \|\rho_{tt}\|_{H^1}^2 \right) dt \leq C(T).
\end{aligned}
\end{equation}
\end{prop}
\begin{proof}
    By applying the operator $\partial_i \partial_j \partial_k$ to equation $\eqref{eq:NSP-original}_1$, multiplying the resulting expression by $\partial_{ijk}\rho$, summing over $i$, $j$, and $k$, and then integrating by parts, we apply \eqref{high}, \eqref{5.1} H\"{o}lder's inequality, Sobolev embedding and arrive at 
 \begin{equation}\label{prop5.2eq1}
\begin{aligned}
\frac{1}{2}\frac{d}{dt}\int_\Omega |\nabla^3\rho|^2 \mathrm{d}\mathbf{x}
&\le C\int_\Omega |\nabla^3\rho|\big(|\nabla^3\mathbf{u}||\nabla\rho| + |\nabla^2\mathbf{u}||\nabla^2\rho| + |\nabla \mathbf{u}||\nabla^3\rho| + \rho|\nabla^4\mathbf{u}|\big)\mathrm{d}\mathbf{x} \\
&\le C(T)\big(\|\nabla^3\rho\|_{L^2}\|\nabla^3\mathbf{u}\|_{L^3}\|\nabla\rho\|_{L^6} + \|\nabla^3\rho\|_{L^2}\|\nabla^2\mathbf{u}\|_{L^\infty}\|\nabla^2\rho\|_{L^2} \\
&\quad + \|\nabla \mathbf{u}\|_{L^\infty}\|\nabla^3\rho\|_{L^2}^2 + \|\rho\|_{L^\infty}\|\nabla^3\rho\|_{L^2}\|\nabla^4\mathbf{u}\|_{L^2}\big) \\
&\le C(T)\big(\|\nabla^3\rho\|_{L^2}\|\mathbf{u}\|_{H^4}\|\rho\|_{H^2} + \|\nabla^3\rho\|_{L^2}\|\mathbf{u}\|_{H^4}\|\rho\|_{H^2} \\
&\quad + \|\mathbf{u}\|_{H^3}\|\nabla^3\rho\|_{L^2}^2 + \|\nabla^3\rho\|_{L^2}\|\mathbf{u}\|_{H^4}\big) \\
&\leq C(T)(\|\nabla^3\rho\|^2_{L^2}+\|\mathbf{u}\|^2_{H^4}+1).
\end{aligned}
\end{equation}
From \eqref{eq: elliptic sys of u}, \eqref{high}, \eqref{5.1} and \eqref{uth2}, it holds that
\begin{equation}\label{prop5.2eq2}
\begin{aligned} 
\|\mathbf{u}\|^2_{H^4}\leq &C\left\lVert \frac{1}{\rho^{\alpha}}\Big( \rho \left( \mathbf{u}_{t}+\mathbf{u}\cdot \nabla \mathbf{u} \right) +\nabla P-\nabla \mu(\rho)\cdot \nabla \mathbf{u}-\nabla\lambda(\rho)\text{div}\left(\mathbf{u}\right)+\rho \nabla \Phi \right)  \Big\rVert _{H^{2}}^{2}\\
\leq& C(T) \Big( \|\mathbf{u}_t\|_{H^2}^2 + \||\nabla^2 \rho| |\mathbf{u}_t|\|_{L^2}^2 + \||\nabla \rho| |\nabla \mathbf{u}_t|\|_{L^2}^2 \\
&+ \||\nabla^2 \mathbf{u}| |\nabla \mathbf{u}|\|_{L^2}^2 + \||\mathbf{u}| |\nabla^3 \mathbf{u}|\|_{L^2}^2 \\
& + \|\nabla^3 \rho\|_{L^2}^2 + \||\nabla^2 \rho| |\nabla \rho|\|_{L^2}^2 + \||\nabla \rho|^3\|_{L^2}^2 \\
& + \||\nabla^3 \rho| |\nabla \mathbf{u}|\|_{L^2}^2 + \||\nabla^2 \rho| |\nabla^2 \mathbf{u}|\|_{L^2}^2 + \||\nabla \rho| |\nabla^3 \mathbf{u}|\|_{L^2}^2 + \||\nabla \rho|^2 |\nabla^2 \mathbf{u}|\|_{L^2}^2 \\
&+ \|\nabla^3 \Phi\|_{L^2}^2 + \||\nabla \rho| |\nabla^2 \Phi|\|_{L^2}^2 + \||\nabla^2 \rho| |\nabla \Phi|\|_{L^2}^2 +\||\nabla\rho|^2|\nabla\Phi|\|^2_{L^2}\Big) \\
\leq& C(T)\big(1 + \|\nabla^3\rho\|_{L^2}^2 + \|\mathbf{u}_t\|_{H^2}^2\big)\\
\leq&  C(T) \Big( 1 + \left\lVert \nabla^3\rho \right\rVert _{L^2}^2 + \left\lVert \mathbf{u}_{tt} \right\rVert _{L^2}^2 \Big).
\end{aligned}
\end{equation}Substituting \eqref{prop5.2eq2} into \eqref{prop5.2eq1}, applying Gronwall's inequality and \eqref{5.1} yields
\begin{equation}\label{5.2 tmp}
    \sup_{0 \le t \le T} \|\rho\|_{H^3}^2 +\int_0^T \|\mathbf{u}\|_{H^4}^2 dt \le C(T).
\end{equation}
Finally, we estimate the time derivatives of the density. Thus, from the above equation and \eqref{psiestimate}, we obtain an estimate for $\Phi$. By taking the second-order spatial derivative  of the continuity equation $\eqref{eq:NSP-original}_1$ and utilizing Hölder's inequality, Sobolev embedding, \eqref{high}, \eqref{5.1} along with \eqref{5.2 tmp}, we obtain
\begin{equation}\label{5.2 tmp2}
    \begin{aligned}
\|\nabla^2 \rho_t\|_{L^2} &\le C\|\nabla^2(\mathbf{u} \cdot \nabla \rho)\|_{L^2} + C\|\nabla^2(\rho \text{div} \mathbf{u})\|_{L^2} \\
&\le C\|\nabla^2 \mathbf{u} \cdot \nabla \rho\|_{L^2} + C\|\nabla \mathbf{u} \cdot \nabla^2 \rho\|_{L^2} +C \|\mathbf{u} \cdot \nabla^3 \rho\|_{L^2} \\
&\quad +C\|\nabla^2 \rho \text{div} \mathbf{u}\|_{L^2} + C\|\nabla \rho \cdot \nabla\text{div} \mathbf{u}\|_{L^2} + C\|\rho \nabla^2\text{div} \mathbf{u}\|_{L^2} \\
&\le C(T)\|\nabla^2 \mathbf{u}\|_{L^2}\|\nabla \rho\|_{L^\infty} + C(T)\|\nabla \mathbf{u}\|_{L^\infty}\|\nabla^2 \rho\|_{L^2} + C(T)\|\mathbf{u}\|_{L^\infty}\|\nabla^3 \rho\|_{L^2} \\
&\quad +C(T) \|\nabla^2 \rho\|_{L^2}\|\nabla \mathbf{u}\|_{L^\infty} + C(T)\|\nabla \rho\|_{L^\infty}\|\nabla^2 \mathbf{u}\|_{L^2} + C(T)\|\rho\|_{L^\infty}\|\nabla^3 \mathbf{u}\|_{L^2} \\
&\le C(T)\|\mathbf{u}\|_{H^3}\|\rho\|_{H^3}\le C(T).
\end{aligned}
\end{equation}
Differentiating the continuity equation $\eqref{eq:NSP-original}_1$ with respect to $t$ gives
\begin{equation}\label{eq:rho_tt_eq}
\rho_{tt} = - \rho_t \text{div} \mathbf{u} - \rho \text{div} \mathbf{u}_t - \mathbf{u}_t \cdot \nabla \rho - \mathbf{u} \cdot \nabla \rho_t.
\end{equation}
Taking the $L^2$-norm yields
\begin{equation}
\begin{aligned}
\|\rho_{tt}\|_{L^2} &\le C\big(\|\rho_t\|_{L^3}\|\nabla \mathbf{u}\|_{L^6} + \|\rho\|_{L^\infty}\|\nabla \mathbf{u}_t\|_{L^2} + \|\mathbf{u}_t\|_{L^6}\|\nabla \rho\|_{L^3} + \|\mathbf{u}\|_{L^\infty}\|\nabla \rho_t\|_{L^2}\big) \\
&\leq C\left( \left\lVert \rho_{t} \right\rVert _{H^{1}}\left\lVert \mathbf{u} \right\rVert _{H^{2}}+\left\lVert \rho \right\rVert _{H^{2}}\left\lVert \nabla \mathbf{u}_{t} \right\rVert _{L^{2}}\right) \le C(T),
\end{aligned}
\end{equation}
where we use \eqref{5.1}, \eqref{5.2 tmp2}. Take the spatial gradient on \eqref{eq:rho_tt_eq} and estimate its $L^2$-norm as follows
\begin{equation}
\begin{aligned}
\|\nabla \rho_{tt}\|_{L^2} &\le C\Big(\|\nabla\rho_t \text{div} \mathbf{u}\|_{L^2} + \|\rho_t \nabla\text{div} \mathbf{u}\|_{L^2} + \|\nabla\rho \text{div} \mathbf{u}_t\|_{L^2} + \|\rho \nabla\text{div} \mathbf{u}_t\|_{L^2} \\
&\quad + \|\nabla \mathbf{u}_t \cdot \nabla \rho\|_{L^2} + \|\mathbf{u}_t \cdot \nabla^2 \rho\|_{L^2} + \|\nabla \mathbf{u} \cdot \nabla \rho_t\|_{L^2} + \|\mathbf{u} \cdot \nabla^2 \rho_t\|_{L^2} \Big)\\
&\le C(T)\big(\|\nabla\rho_t\|_{L^3}\|\nabla \mathbf{u}\|_{L^6} + \|\rho_t\|_{L^\infty}\|\nabla^2 \mathbf{u}\|_{L^2} + \|\nabla\rho\|_{L^\infty}\|\nabla \mathbf{u}_t\|_{L^2} + \|\rho\|_{L^\infty}\|\nabla^2 \mathbf{u}_t\|_{L^2} \\
&\quad + \|\nabla \mathbf{u}_t\|_{L^2}\|\nabla \rho\|_{L^\infty} + \|\mathbf{u}_t\|_{L^\infty}\|\nabla^2 \rho\|_{L^2} + \|\nabla \mathbf{u}\|_{L^\infty}\|\nabla \rho_t\|_{L^2} + \|\mathbf{u}\|_{L^\infty}\|\nabla^2 \rho_t\|_{L^2}\big)\\
&\le C(T)\big(1 + \|\nabla^2 \mathbf{u}_t\|_{L^2} + \|\mathbf{u}_t\|_{H^2}\big) \\
&\le C(T)\big(1 + \|\mathbf{u}_t\|_{H^2}\big).
\end{aligned}
\end{equation}
Integrating this over time and recalling \eqref{5.1} yield
\begin{equation}
    \int_0^T \|\rho_{tt}\|_{H^1}^2 dt \le C(T) \int_0^T \big(1 + \|\mathbf{u}_t\|_{H^2}^2\big) dt \le C(T).
\end{equation}
From \eqref{eq:NSP-original} and \eqref{psiestimate}, we have
\begin{equation}
\Delta \Phi_t = \rho_t, \quad \Delta \Phi_{tt} = -\operatorname{div}(\rho_t \mathbf{u} + \rho \mathbf{u}_t),
\end{equation}
then we complete the proof of \eqref{5.2}.
\end{proof}
\subsection{Proof of the Main Results}

\subsubsection{Proof of the Theorem \ref{thm:global-classical-2dalpha1} and \ref{thm:global-classical-3dalpha1}}

For \( N = 2\) and \( N = 3\), based on lower-order estimates

\begin{equation}
\begin{aligned}
&\sup_{0 \leq t \leq T} \left( \| \mathbf{u} \|_{L^4} + \|\rho\|_{L^\infty} + \|\rho^{-1}\|_{L^\infty} + \|\nabla\rho\|_{L^2} + \|\nabla\rho\|_{L^4} \right)\\ 
&+ \int_0^T \left( \|\nabla\rho\|_{L^2}^2 + \|\nabla\rho\|_{L^4}^4 + \|\nabla\mathbf{u}\|_{L^2}^2 \right) dt \leq C(T),
\end{aligned}\end{equation}
we derive the higher-order estimates
\begin{equation}
\begin{aligned}
\sup_{0\leq t\leq T} &\left( \|\rho\|_{H^3}^2 + \|\nabla\Phi\|_{H^4}^2+\|\rho_t\|_{H^2}^2 + \|\nabla\Phi_t\|_{H^3}^2 +\|\rho_{tt}\|_{L^2}^2+\|\nabla\Phi_{tt}\|^2_{H^1} \right) \\
&+ \int_0^T \left( \|\mathbf{u}\|_{H^4}^2 + \|\rho_{tt}\|_{H^1}^2 \right) dt \leq C(T).
\end{aligned}
\end{equation}

Now that we have obtained the global a priori estimates, we proceed to prove that the solution can be extended to a global one.

\subsubsection{Regularity}

First, we verify the continuity of the solution, i.e.,

\begin{equation}\label{6.3}
(\rho,\nabla\Phi) \in C([0, T]; H^3), \quad u \in C([0, T]; H^3).
\end{equation}

The regularity \(\rho \in L^\infty(0, T; H^3)\) together with \(\rho_t \in L^\infty(0, T; H^2)\) yields \(\rho \in C([0, T]; H^2)\). The strong continuity in \(H^3\) follows from standard arguments based on weak and norm convergence. For the velocity field, the condition \(\mathbf{u} \in L^2(0, T; H^4) \cap H^1(0, T; H^2)\) gives \(\mathbf{u} \in C([0, T]; H^3)\) by standard interpolation theory.

\subsubsection{Global Existence}

To show that \(T^* = \infty\), we proceed by contradiction. Assume \(T^* < \infty\). The estimates \eqref{eq: u H2} remain valid as \(t \to T^*\); the continuity claim \eqref{6.3} allows us to define the limiting values \((\rho(\mathbf{x},T^*), \nabla\Phi(\mathbf{x},T^*), \mathbf{u}(\mathbf{x},T^*)) \in H^3\). The density retains a positive lower and upper bound. These limits can be taken as new initial data. By Lemma \ref{lem:local-existence}, this implies that \((\rho, u, \Phi)\) is indeed the unique classical solution defined on \(\Omega \times (0, T]\) for any \(0 < T < T^* = \infty\). The proof of Theorem \ref{thm:global-classical-2dalpha1} and \ref{thm:global-classical-3dalpha1} is thus complete.
\section{Global classical solutions away from vacuum ($\alpha=1$)}
\subsection{A priori estimates for $N=2$: lower-order estimates}
\begin{lema}Under the assumption of\eqref{gamma1alpha=1}, there exists a constant $C(T) > 0$ such that 
   \begin{equation}\label{basic2}
     \sup_{0\leq t\leq T}  \int_{0}^{1} \left( u^{2}+\rho^{\gamma-1}  \right) dx + \int_{0}^{T} \int_{0}^{1} \left( \frac{u^{2}}{r^{2}} +\rho^2  r^2u_x^2\right) dx dt \le C (T).
   \end{equation} 
\end{lema}
\begin{proof}
We take the $L^2$ inner product of $\eqref{eq:NSP-lagrangian-rhoalpha}_2$ with $u$ on $[0,1]$. Then, integration by parts and the use of $\eqref{eq:NSP-lagrangian-rhoalpha}_1$ give
\begin{equation}\begin{aligned}
    \label{u1}
\frac{d}{d\tau} \int_0^1 \left( \frac{1}{2}u^2 + \frac{1}{\gamma-1} \rho^{\gamma-1} \right) dx 
&+  \int_0^1 \frac{u^2}{r^2} dx + \int_0^1 \rho^2  r^2u_x^2 dx \\
&\leq \frac{1}{8}\int_{0}^{1}\frac{u^2}{r^2}dx+C\int_{0}^{1}x^2dx+C\int_{0}^{1}r^4dx\\
&\leq \frac{1}{8}\int_{0}^{1}\frac{u^2}{r^2}dx+C.
\end{aligned}\end{equation}
Integrating over $(0,\tau)$, we obtain \eqref{basic2}.
\end{proof}

\begin{lema}   
Under the assumption \eqref{gamma1alpha=1}, there exists a constant $C(T) > 0$ such that 
\begin{equation}
x \leq C(T)\, r(x,\tau)^{\frac{2(\gamma-1)}{\gamma}}, \quad (x,\tau) \in [0,1] \times [0,T].
\end{equation}
\end{lema} 
\begin{proof}
For any $x\in[0,1]$ and \eqref{basic2}
 we have,
\begin{equation}
\begin{aligned}
x = \int_{0}^{r(x,\tau)} \rho s\, ds \leq& \left( \int_{0}^{r(x,\tau)} \rho^\gamma s\, ds \right)^{\frac{1}{\gamma}} \left( \int_{0}^{r(x,\tau)} s \,ds \right)^{\frac{\gamma-1}{\gamma}}\\
\leq& C(T) r^{\frac{2(\gamma-1)}{\gamma}}.
\end{aligned}\end{equation}
\end{proof}
\begin{lema}
There exists a positive constant $C$, independent of $T$, such that for any $0<\xi\ll 1$,
\begin{equation}\label{weighted2}
\sup_{0\leq t\leq T}\Big\|\rho^{\frac{1}{2}}r^{\xi}\Big\|_{L^\infty(0,R)}\leq C(1+\Big\|\nabla\rho^{\frac{1}{2}}\Big\|_{L^2(\Omega)}).
\end{equation}
\end{lema}
\begin{proof}
Similar to the previous case where $\alpha < 1$, by combining the one-dimensional Sobolev embedding inequality with H\"{o}lder's inequality, we obtain
\begin{equation}
\begin{aligned}
\Big\| \rho^{\frac{1}{2}}r^{\xi}\Big\|_{L^{\infty}(0,R)} 
&\leq C\left(1 + \Big\| \nabla \rho^{\frac{1}{2}}\Big\|_{L^{2}(\Omega)}+\Big\| \rho^{\frac{1}{2}}\Big\|_{L^{2}(\Omega)}\right)\\
&\leq C\left(1 + \Big\| \nabla \rho^{\frac{1}{2}}\Big\|_{L^{2}(\Omega)}\right).
\end{aligned}
\end{equation}
\end{proof}
\begin{lema}
Under the assumption of\eqref{gamma1alpha=1}, there exists a constant $C(T) > 0$ such that 
    \begin{equation}\label{2w2}
\sup_{0\leq t\leq T}\int^{1}_{0}w^{2}dx+\int^{T}_{0}\int^{1}_{0}\left( \left( \rho^{\frac{\gamma+1}{2}} \right)_{x}r \right)^{2}dxdt\leq C(T).
    \end{equation}
\end{lema}
\begin{proof}
Multiplying \eqref{bd2_alpha<1} by $w$ and integrating the resulting equation over $[0, 1]$, followed by an integration by parts, yields
\begin{equation}
\begin{aligned}\label{w1}
&\frac{d}{d\tau}  \left(\frac{1}{2}\int_0^1  w^2 dx +  \frac{1}{\gamma - 1} \int_0^1 \rho^{\gamma - 1} dx \right)  + \frac{4\gamma}{(\gamma + 1)^2} \int_0^1 \left( (\rho^{\frac{\gamma + 1}{2}})_x r \right)^2 dx \\
=&  -\kappa\int_0^1 \frac{x}{r}wdx+\frac{\kappa}{2}\int_0^1 rwdx
\triangleq R_1+R_2.
\end{aligned}
\end{equation}
By H\"{o}lder's inequality and \eqref{weighted2}, we obtain
\begin{equation}
\begin{aligned}\label{r1}
R_1\leq& \int_{0}^{1}w^2dx+C\int_{0}^{1}\frac{x^2}{r^2}dx\\
\leq&\int_{0}^{1}w^2dx+C\int_{0}^{1}r^{2-\frac{2}{\gamma}}dx\\
\leq&\int_{0}^{1}w^2dx+C\int_{0}^{R}\rho r^{3-\frac{4}{\gamma}}dr\\
\leq&\int_{0}^{1}w^2dx+C\Big\|\rho^{\frac{1}{2}}r^{\xi}\Big\|^2_{L^\infty(0,R)}\int_{0}^{R} r^{3-\frac{4}{\gamma}-2\xi}dr\\
\leq&\int_{0}^{1}w^2dx+C(1+\Big\|\nabla\rho^{\frac{1}{2}}\Big\|^2_{L^2(\Omega)})\\
\leq&C\Big(1+\Big\|u\Big\|^2_{L^2(0,1)}+\Big\|w\Big\|^2_{L^2(0,1)}\Big),
\end{aligned}
\end{equation}
where we need 
$$\xi<\frac{2(\gamma-1)}{\gamma}.$$
Similarly, we obtain
\begin{equation}
\begin{aligned}\label{w2}
R_2\leq&C\|w\|^2_{L^2(0,1)}+C.
\end{aligned}
\end{equation}
Combining \eqref{r1}-\eqref{w2} and \eqref{w1}, we obtain

\begin{equation}
\begin{aligned}
&\frac{d}{d\tau}\int_0^1\left(\frac{1}{2} u^2 + \frac{1}{2} w^2 + \frac{\rho^{\gamma - 1}}{\gamma - 1} \right)dx 
+ \int_0^1 r^2 \rho^{\gamma - 1} \rho_x^2dx \\
\leq &C(\xi)\left(1 + \| u\|_{L^2(0,1)}^2 + \| w\|_{L^2(0,1)}^2\right).
\end{aligned}
\end{equation}
By applying Gronwall's inequality, we finally prove equation \eqref{2w2}.

Thus,
\begin{equation}\label{1weighted2}
\sup_{0\leq t\leq T}\|\rho^{\frac{1}{2}}r^{\xi}\|_{L^\infty(0,R)}\leq C(T).
\end{equation}
\end{proof}
    \begin{lema}   
Under the assumption of\eqref{gamma1alpha=1}. For any $1 \leq q < \infty$, there exists a constant $C(T,q) > 0$ such that
\begin{equation}\label{xr1}
x \leq C(T,q)\, r(x,\tau)^{\frac{2(q-1)}{q}}, \quad (x,\tau) \in [0,1] \times [0,T].
\end{equation}
\end{lema} 
\begin{proof}
For any $x\in[0,1]$ and \eqref{2w2}
 we have,
\begin{equation}
\begin{aligned}
x = \int_{0}^{r(x,\tau)} \rho s\, ds  \leq& \left( \int_{0}^{r(x,\tau)} \rho^q s\, ds \right)^{\frac{1}{q}} \left( \int_{0}^{r(x,\tau)} s \,ds \right)^{\frac{q-1}{q}}\\
\leq& C(T,q) r^{\frac{2(q-1)}{q}}.
\end{aligned}\end{equation}
\end{proof}
\begin{lema}
Under the assumption of\eqref{gamma1alpha=1}, there exists a constant $C(T) > 0$ such that  
\begin{equation}\label{2u4}
\sup_{0\leq \tau \leq T}\int_{0}^{1}u^{4}dx + \int_{0}^{T}\int_{0}^{1}\left(\frac{u^{4}}{r^{2}} +\rho^{2}r^{2}u_x^{2}u^{2}\right)dx dt \leq C(T),
\end{equation}\begin{equation}\label{2w4}
\sup_{0\leq \tau \leq T}\int_{0}^{1}w^{4}dx + \int_{0}^{T}\int_{0}^{1}\rho^{\gamma -1}(r\rho_x)^{4}dx dt \leq C(T).
\end{equation}
\end{lema}

\begin{proof}
We multiply $\eqref{eq:NSP-lagrangian-rhoalpha}_2$ by $u^3$ and integrate over $[0, 1]$. Performing an integration by parts on the resulting equation, we obtain
\begin{equation}\begin{aligned}\label{finallp}
\frac{1}{4}\frac{d}{d\tau}\int_{0}^{1}u^4dx+\int_{0}^{1}\frac{u^4}{r^2}dx+3\int_{0}^{1}u^2u_x^2\rho^2r^2dx=&-\int_{0}^{1}(\rho^{\gamma})_xru^3dx-\kappa\int_{0}^{1}\frac{x}{r}u^3dx\\
&+\frac{\kappa}{2}\int_{0}^{1}u^3rdx\\
\triangleq& P_1+P_2+P_3.
\end{aligned}\end{equation}
Integrating by parts and applying H\"{o}lder's inequality, we have
\begin{equation}\begin{aligned}\label{p1}
P_1=&\int_{0}^{1}\rho^{\gamma}(ru^3)_xdx\\
=&\int_{0}^{1}\rho^{\gamma-1}\frac{u^3}{r}dx+3\int_{0}^{1}\rho^{\gamma}u^2u_xrdx\\
\leq&\frac{1}{4}\int_{0}^{1}\frac{u^4}{r^2}dx+\frac{1}{4}\int_{0}^{1}u^2u_x^2\rho^2r^2dx+C\int_{0}^{1}
\rho^{2\gamma-2}u^2dx\\
\leq&\frac{1}{4}\int_{0}^{1}\frac{u^4}{r^2}dx+\frac{1}{4}\int_{0}^{1}u^2u_x^2\rho^2r^2dx+
\Big\|\rho^{\frac{1}{2}}r^{\xi}\Big\|^{4(\gamma-1)}_{L^\infty}\int_{0}^{1}\frac{u^2}{r^2}dx,
\end{aligned}\end{equation}
where $\xi< \frac{1}{2(\gamma-1)}.$
Applying Young's inequality together with \eqref{xr1}, \eqref{1weighted2} yields
\allowdisplaybreaks
\begin{align}
P_2\leq&\frac{1}{4}\int_{0}^{1}\frac{u^4}{r^2}dx+C\int_{0}^{1}x^4r^2dx\nonumber\\
\leq&\frac{1}{4}\int_{0}^{1}\frac{u^4}{r^2}dx+C\int_{0}^{1}r^{10-\frac{8}{q}}dx\nonumber\\
\leq&\frac{1}{4}\int_{0}^{1}\frac{u^4}{r^2}dx+C\int_{0}^{R}\rho r^{11-\frac{8}{q}}dr\nonumber\\
\leq&\frac{1}{4}\int_{0}^{1}\frac{u^4}{r^2}dx+C\Big\|\rho^{\frac{1}{2}}r^{\xi}\Big\|^2_{L^\infty}\int_{0}^{R} r^{11-\frac{8}{q}-2\xi}dr\nonumber\\
\leq&\frac{1}{4}\int_{0}^{1}\frac{u^4}{r^2}dx+C(T).
\end{align}
By a similar argument,
\begin{equation}\begin{aligned}\label{p3}
P_3\leq&\frac{1}{4}\int_{0}^{1}\frac{u^4}{r^2}dx+C\|\rho^{\frac{1}{2}}r^{\xi}\|^2_{L^\infty}\int_{0}^{R} r^{11-2\xi}dr\\
\leq&\frac{1}{4}\int_{0}^{1}\frac{u^4}{r^2}dx+C(T).
\end{aligned}\end{equation}
Integrating \eqref{p1}--\eqref{p3} and \eqref{finallp} over $(0,\tau)$ yields \eqref{2u4}.

After multiplying \eqref{bd2_alpha<1} by $w^3$ and integrating over $[0, 1]$, an integration by parts leads to
\begin{equation}\begin{aligned}\label{h1}
\frac{1}{4}\frac{d}{d\tau}\int_{0}^{1}w^4dx=&-\int_{0}^{1}(\rho^{\gamma})_xrw^3dx-\kappa\int_{0}^{1}\frac{x}{r}w^3dx+\frac{\kappa}{2}\int_{0}^{1}w^3rdx\\
\triangleq& E_1+E_2+E_3.
\end{aligned}\end{equation}
By the definition of effective velocity, we have
\begin{equation}\begin{aligned}
E_1=&-\gamma\int_{0}^{1}r\rho_x\rho^{\gamma-1}w^3dx\\
=&-\gamma\int_{0}^{1}(w-u)\rho^{\gamma-1}w^3dx\\
=&-\gamma\int_{0}^{1}w^4\rho^{\gamma-1}dx+\gamma\int_{0}^{1}\rho^{\gamma-1}uw^3dx\\
\triangleq& E_{11}+E_{12},
\end{aligned}\end{equation}
By \eqref{1weighted2} we have,
\begin{equation}\begin{aligned}
E_{12}\leq&\frac{\gamma}{2}\int_{0}^{1}\rho^{\gamma-1}w^4dx+\int_{0}^{1}\rho^{\gamma-1}u^4dx\\
\leq&\frac{\gamma}{2}\int_{0}^{1}\rho^{\gamma-1}w^4dx+C\Big\|\rho^{\frac{1}{2}}r^{\xi}\Big\|^{2(\gamma-1)}_{L^\infty}\int_{0}^{1}\frac{u^4}{r^2}dx\\
\leq&\frac{\gamma}{2}\int_{0}^{1}\rho^{\gamma-1}w^4dx+C(T)\int_{0}^{1}\frac{u^4}{r^2}dx.
\end{aligned}\end{equation}
From Young's inequality and \eqref{xr1}, it follows that
\begin{equation}\begin{aligned}
E_2\leq& \int_{0}^{1}w^4dx+C\int_{0}^{1}(\frac{x}{r})^4dx\\
\leq&\int_{0}^{1}w^4dx+C\int_{0}^{1}(r^{1-\frac{2}{q}})^4dx\\
\leq&\int_{0}^{1}w^4dx+C\Big\|\rho^{\frac{1}{2}}r^{\xi}\Big\|^2_{L^\infty}\int_{0}^{R}r^{5-\frac{8}{q}-2\xi}dr\\
\leq&\int_{0}^{1}w^4dx+C(T).
\end{aligned}\end{equation}
In a similar manner,
\begin{equation}\begin{aligned}\label{h3}
E_3\leq&\int_{0}^{1}w^4dx+C(T).
\end{aligned}\end{equation}
A combination of \eqref{h1}-\eqref{h3} yields
\begin{equation}\begin{aligned}
\frac{d}{d\tau}\int_{0}^{1}w^4dx+2{\gamma}\int_{0}^{1}\rho^{\gamma-1}w^4dx\leq C(T)\left(1+\int_{0}^{1}\frac{u^4}{r^2}dx+\int_{0}^{1}w^4dx\right).
\end{aligned}\end{equation}
From Gronwall's inequality and \eqref{2u4}, it follows that \eqref{2w4} holds.
\end{proof}

Based on the $L^4$ estimates for the velocity field and the effective velocity, we are able to further derive the upper bound for the density.
\begin{prop}
\label{prop:density_upper}
Under the assumption of\eqref{gamma1alpha=1}, there exists a constant $C(T) > 0$ such that 
\begin{equation}
R_T\leq C(T). 
\end{equation}
\end{prop}
\begin{proof}
In view of \eqref{2u4} and \eqref{2w4}, we obtain
\begin{equation}\begin{aligned}
\sup_{0\leq t\leq T}\int_{0}^{1}(r\rho_x)^4dx\leq& C\sup_{0\leq t\leq T}\int_{0}^{1}(w-u)^4dx\\
\leq&C\sup_{0\leq t\leq T}\int_{0}^{1}w^4dx+C\sup_{0\leq t\leq T}\int_{0}^{1}u^4dx\\
\leq& C(T).
\end{aligned}\end{equation}
Transforming back to Eulerian coordinates, we obtain
\begin{equation}
\sup_{0\leq t\leq T}\int_{0}^{R}\left|\partial_{r}\rho^{\frac{1}{4}}\right|^{4}r dr\leq C(T).
\end{equation}
By the two-dimensional Sobolev embedding inequality, we deduce
\begin{equation}
\sup_{0\leq t\leq T}\Big\| \rho^{\frac{1}{4}}\Big\|_{L^{\infty}(\Omega)}\leq C\Big\| \rho^{\frac{1}{4}}\Big\|_{L^{\infty}L^{4}(\Omega)} + C\Big\| \nabla \rho^{\frac{1}{4}}\Big\|_{L^{\infty}L^{4}(\Omega)}\leq C(T).
\end{equation}
Therefore, we get the upper bound of the density.
\end{proof}
\begin{lema}
\label{lem3.9}
Under the assumption of\eqref{gamma1alpha=1}, there exists a constant $C(T) > 0$ such that 
\begin{equation}\label{2uinfty}
\sup_{0\leq \tau \leq T}\| u(\cdot,\tau)\|_{L^{\infty}(0,1)}\leq C(T),
\end{equation}
and
\begin{equation}\label{2winfty}
\sup_{0\leq \tau \leq T}\| w(\cdot,\tau)\|_{L^{\infty}(0,1)}\leq C(T).
\end{equation}
\end{lema}
\begin{proof}    
Similar to the proof for the case $\alpha < 1$, multiplying $\eqref{eq:NSP-lagrangian-rhoalpha}_2$ by $u^{2n-1}$ and integrating over $[0,1]$, we obtain after integration by parts
    \begin{equation}
        \begin{aligned}\label{2uhighintegrabilityall}
           \frac{1}{2n} \frac{d}{d\tau}\int^{1}_{0}u^{2n}dx+&\int^{1}_{0} \frac{u^{2n}}{r^{2}}dx+(2n-1)\int^{1}_{0}\rho^{2}u^{2n-2}u_{x}^{2}r^{2}dx\\
          =&\gamma\int_{0}^{1}r\rho_x\rho^{\gamma-1}u^{2n-1}dx
            { -\kappa\int^{1}_{0} \frac{x}{r}u^{2n-1}dx }+{ \frac{\kappa}{2}\int^{1}_{0}u^{2n-1}rdx }\\
            \leq&C(T)\Big(\|w\|^{2n}_{L^{2n}}+\|u\|^{2n}_{L^{2n}}\Big)+ \varepsilon\int^{1}_{0} \frac{u^{2n}}{r^{2}}dx +C(\varepsilon,T).
        \end{aligned}
    \end{equation}
Choosing $\varepsilon$ small enough, we get
\begin{equation}
    \begin{aligned}\label{uhighintegrabilityestimate}
        \frac{1}{2n} \frac{d}{d\tau}\int^{1}_{0}u^{2n}dx+&\frac{1}{2} \int^{1}_{0} \frac{u^{2n}}{r^{2}}dx+ (2n-1)\int^{1}_{0}u^{2n-2}(r^{2}u_{x})^{2}\rho^{2}dx\\
        &\leq C(\varepsilon,T)+C\Big(\int^{1}_{0}u^{2n}dx+\int^{1}_{0}w^{2n}dx\Big).
    \end{aligned}
\end{equation}
Multiplying \eqref{bd2_alpha<1} by $w^{2n-1}$ yields 
\begin{equation}\begin{aligned}
\label{2whighintegrabilityall}
    \frac{1}{2n} \frac{d}{d\tau} \int^{1}_{0}w^{2n}dx+& \frac{\gamma}{2}\int_{0}^{1}\rho^{\gamma -1}w^{2n}dx\\
    \leq& C\left( \int w^{(2n-1) \frac{2n}{2n-1}}dx \right)^{\frac{2n-1}{2n}}\left( \int^{1}_{0} \left( \frac{x}{r} \right)^{2n}dx \right)^{\frac{1}{2n}}\\
    &+C\left( \int^{1}_{0}w^{2n} dx\right)^{\frac{2n-1}{2n}}\left( \int^{1}_{0}r^{2n}dx \right)^{\frac{1}{2n}}\\
    \leq& C\int^{1}_{0}w^{2n} dx+C(T)^{2n},
\end{aligned}\end{equation}
where we used \eqref{xr1}.
From \eqref{uhighintegrabilityestimate} together with \eqref{2whighintegrabilityall}, we derive a mixed-type inequality
\begin{equation}\begin{aligned}
\frac{1}{2n}\frac{d}{d\tau}\left(\int^{1}_{0}u^{2n}dx+\int^{1}_{0}w^{2n}dx\right) +& \frac{\gamma}{2}\int_{0}^{1}\rho^{\gamma -1}w^{2n}dx\\
+&\int_{0}^{1}\left(\frac{1}{2}\frac{u^{2n}}{r^{2}} +(2n - 1)\rho^{2}u_{x}^{2}u^{2n - 2}r^{2}\right)dx\\
\leq& C(T)\left( \| w\|_{L^{2n}}^{2n} + \| u\|_{L^{2n}}^{2n}\right)+C(T)^{2n}.
\end{aligned}\end{equation}
A direct application of Grönwall's inequality leads to
\begin{equation}\begin{aligned}
\sup_{0\leq t\leq T}(\|w\|^{2n}_{L^{2n}(0,1)}+\|u\|^{2n}_{L^{2n}(0,1)})\leq e^{2C(T)nT}\left(\|u_0\|_{L^{2n}}^{2n}+\|w_0\|_{L^{2n}}^{2n}+2nTC(T)^{2n}\right).
\end{aligned}\end{equation}
Taking the $\frac{1}{2n}$ power on both sides and passing to the limit as $n \to \infty$ yields \eqref{2uinfty}-\eqref{2winfty}, which completes the proof.
\end{proof}
\begin{prop}
\label{prop3.5}
Under the assumption of\eqref{gamma1alpha=1}, there exists a constant $C(T) > 0$ such that 
\begin{equation}\begin{aligned}\label{lower bound of density}
V_T\leq C(T). 
\end{aligned}\end{equation}
\end{prop}
\begin{proof}
Using \eqref{2uinfty}, \eqref{2winfty} together with the definition of the effective velocity  yields
\begin{equation}\begin{aligned}
\sup_{0\leq t\leq T}\| r\partial_{x}\rho \|_{L^{\infty}(0,1)}\leq C(T). 
\end{aligned}\end{equation}
Define the specific volume $v = 1/\rho$. 

Utilizing $\eqref{eq:NSP-lagrangian-rhoalpha}_2$, a direct computation yields the evolution equation for $v$:
\begin{equation}\begin{aligned}
\partial_{\tau}v =\partial_{x}(r u),
\end{aligned}\end{equation}
hence, we get
\begin{equation}\begin{aligned}
\int_{0}^{1}v(x,\tau)\,dx = \int_{0}^{1}v_{0}(x)\,dx.
\end{aligned}\end{equation}

To derive the \(L^{\infty}\) bound for \(v\), we employ one-dimensional Sobolev embedding \(W^{1,1}(0,R) \hookrightarrow L^{\infty}(0,R)\). For a sufficiently small parameter \(\beta \in (0,1)\) to be chosen later, the following estimate holds for any \((x,\tau) \in (0,1)\times(0,T)\):
\begin{equation}\begin{aligned}
v^{\beta}(x,\tau) &\leq \int_{0}^{1}v^{\beta}dx + \beta \int_{0}^{1}v^{\beta -1}|\partial_{x}v|\,dx \\
&\leq \left(\int_{0}^{1}v\,dx\right)^{\beta} + \beta \int_{0}^{1}|\partial_{x}\rho |\,v^{\beta +1}\,dx\\
& \leq \left(\int_{0}^{1}v\,dx\right)^{\beta} + C\beta\|r\partial_x\rho\|_{L^\infty} \left(\int_{0}^{1} v r^{-3/2}\,dx\right)^{\frac{2}{3}} \left(\int_{0}^{1} v^{3\beta+1}\,dx\right)^{\frac{1}{3}}\\
& \leq \left(\int_{0}^{1}v\,dx\right)^{\beta} + C(T)\beta\left( \sup_{(x,\tau)} v(x,\tau) \right)^{\beta}.
\end{aligned}\end{equation}
We have
\[
V_T^{\beta} \leq C(\beta) + C(T)\beta V_T^{\beta}.
\]
We can now select \(\beta \min\{1/(2C(T)), 1/2\}\)  yielding \eqref{lower bound of density}.
\end{proof}

\subsection{A priori estimates for 
$N=3$: lower-order estimates}

\begin{lema}Under the assumption of \eqref{assumptiongamma1}, consider either of the following cases:
\begin{itemize}
    \item $\kappa = -1$ (plasmas) with $\gamma > 1$;
    \item $\kappa = 1$ (gaseous stars) with $\gamma > \frac{4}{3}$.
\end{itemize}
Then there exists a constant $C(T)>0$ such that 
   \begin{equation}\label{b1}
       \sup_{0\leq t\leq T}\int_{0}^{1} \left( u^{2}+\rho^{\gamma-1}  \right) dx + \int_{0}^{T} \int_{0}^{1} \left( \frac{u^{2}}{r^{2}} +\rho^2(r^2u_x)^2 \right) dx dt \le C (T).
   \end{equation} 
\end{lema}
\begin{proof}
The proof is similar for $\alpha < 1$. Note that because $\alpha = 1$, there is no cross term.
\end{proof}
\begin{lema}   
Under the assumption of \eqref{assumptiongamma1}, there exists a constant $C(T) > 0$ such that
\begin{equation}\label{x13}
x \leq C(T)\, r(x,\tau)^{\frac{3(\gamma-1)}{\gamma}}, \quad (x,\tau) \in [0,1] \times [0,T].
\end{equation}
\end{lema} 
\begin{proof}
For any $x\in[0,1]$ and \eqref{b1}
 we have,
\begin{equation}
\begin{aligned}
x = \int_{0}^{r(x,\tau)} \rho s^2\, ds \leq& \left( \int_{0}^{r(x,\tau)} \rho^\gamma s^2\, ds \right)^{\frac{1}{\gamma}} \left( \int_{0}^{r(x,\tau)} s^2 \,ds \right)^{\frac{\gamma-1}{\gamma}}\\
\leq& C(T) r^{\frac{3(\gamma-1)}{\gamma}}.
\end{aligned}\end{equation}
\end{proof}
\begin{lema}There exists a positive constant $C$, independent of $T$, such that for any $0<\xi\ll 1$,
     \begin{equation}\label{w4rho}
         \sup_{\tau\in [0, T]}\left\lVert \rho^{\frac{1}{2}} r^{\frac{1}{2}+\xi} \right\rVert _{L^\infty(0, R)} \le C\left(\|\nabla\rho^{\frac{1}{2}}\|_{L^2(\Omega}+1\right).
     \end{equation}
\end{lema}
\begin{proof}
Similar to the case $\alpha < 1$, by the one-dimensional Sobolev inequality, we have
    \begin{equation}
        \begin{aligned}\label{w3rho}
\left\lVert \rho^{\frac{1}{2}}r^{\frac{1}{2}+\xi} \right\rVert _{L^{\infty}(0,R)}&\leq \int^{R}_{0}\rho^{\frac{1}{2}}r^{\frac{1}{2}+\xi}dr+\int^{R}_{0} \Big|\partial_{r}(\rho^{\frac{1}{2}}r^{\frac{1}{2}+\xi})\Big|dr\\
& \leq C\left(\|\nabla\rho^{\frac{1}{2}}\|_{L^2(\Omega)}+1\right).
\end{aligned}
    \end{equation}
\end{proof}

\begin{lema}
Under the assumption of \eqref{assumptiongamma1}, there exists a constant $C(T) > 0$ such that, for both the plasma case ($\kappa = -1$) and the gaseous star case ($\kappa = 1$) with $\gamma > \frac{4}{3}$, the following estimate holds
Then it holds that
    \begin{equation}\label{bdsfinall}
        \sup_{0\leq \tau\leq T}\int^{1}_{0}w^{2}dx+\int^{T}_{0}\int^{1}_{0}\left( \left( \rho^{\frac{\gamma+1}{2}} \right)_{x}r^{2} \right)^{2}dxds\leq C(T).
    \end{equation}
\end{lema}
\begin{proof}
We multiply the equation \eqref{bd2_alpha<1} by $w = u + r^2 \rho_x$ to obtain
    \begin{equation}\begin{aligned}\label{bd3}
        \frac{1}{2} \frac{d}{d\tau}\int^{1}_{0}(u+r^{2}\rho_{x})^{2}dx=&-\int^{1}_{0}(\rho^{\gamma})_{x}r^{2}(u+r^{2}\rho_{x})dx\\
        & -\kappa\int^{1}_{0} \frac{x}{r^{2}}(u+r^{2}\rho_{x})dx+ \frac{{\kappa}}{3}\int^{1}_{0}r(u+r^{2}\rho_{x})dx.\\
     \end{aligned}\end{equation}
Similar to the case $\alpha < 1$, we have
    \begin{equation}
        \begin{aligned}\label{finallalpha<1both00}
&\frac{1}{2}\int^{1}_{0}w^{2}dx+ \frac{1}{\gamma-1}\int^{1}_{0}\rho^{\gamma-1}dx+ \frac{4\gamma}{(\gamma+\alpha)^{2}}\int^{\tau}_{0}\int^{1}_{0}\left( \left( \rho^{\frac{\gamma+1}{2}} \right)_{x}r^{2} \right)^{2}dxds\\
\leq&\frac{1}{2}\int^{1}_{0}w_{0}^{2}dx+ \frac{1}{\gamma-1}\int^{1}_{0}\rho_{0}^{\gamma-1}dx + \frac{1}{2(\gamma-1)}\int^{1}_{0}\rho^{\gamma-1}dx+C(T)\\
&+C\int^{\tau}_{0}\int^{1}_{0}\rho^\alpha dxds+C\int^{\tau}_{0}\int^{1}_{0}w^{2}dxds\\
\leq& C(T)+C\int^{\tau}_{0}\int^{1}_{0}\rho dxds+C\int^{\tau}_{0}\int^{1}_{0}w^{2}dxds.
\end{aligned}
    \end{equation}
   
Then it remains to estimate the term $\int^{\tau}_{0}\int^{1}_{0}\rho dxds.$

It is worth noting that we have
\begin{equation}
    \begin{aligned}
\int^{1}_{0}\rho dx&= \int^{R}_{0}\rho^{2}r^{2}dr=\frac{1}{4\pi}\int_{\Omega}\rho^{2}d\mathbf{x}.
\end{aligned}
\end{equation}
Applying the interpolation inequality gives
    \begin{equation}
        \begin{aligned}\label{3rho1}
\left\lVert \rho \right\rVert _{L^{1}(0,1)}=  \frac{1}{4\pi}\left\lVert \rho \right\rVert^{2}_{L^{2}(\Omega)} &\leq C\left\lVert \rho \right\rVert_{L^{1}}^{\frac{3\gamma-2}{3\gamma-1}}\left\lVert \rho \right\rVert_{L^{3\gamma}}^{\frac{3\gamma}{3\gamma-1}}\leq C\left\lVert \rho \right\rVert_{L^{3\gamma}}^{\frac{3\gamma}{3\gamma-1}}\\
&\leq C\left\lVert \rho^{\frac{\gamma}{2}} \right\rVert _{L^{6}(\Omega)}^{\frac{6}{3\gamma-1}}\leq C\Big(\left\lVert \rho^{\frac{\gamma}{2}} \right\rVert_{L^{2}(\Omega)}+\left\lVert \nabla\rho^{\frac{\gamma}{2}} \right\rVert_{L^{2}(\Omega)}\Big)^{\frac{6}{3\gamma-1}}\\
&\leq C(T)\Big(1+\left\lVert \nabla\rho^{\frac{\gamma}{2}} \right\rVert_{L^{2}(\Omega)}^{\frac{6}{3\gamma-1}}\Big)\\
&\leq\frac{2\gamma}{(\gamma+1)^{2}}\int^{1}_{0}\left( \left( \rho^{\frac{\gamma+1}{2}} \right)_{x}r^{2} \right)^{2}dx+C(T),
\end{aligned}
    \end{equation}
where we require $\gamma>\frac{4}{3}.$
Applying Gronwall’s inequality, we finally arrive at \eqref{bdsfinall}.
\end{proof}
Based on the BD entropy estimate, we now obtain a relation between $x$ and $r$. This relation eliminates the singularity of the Poisson term in the higher integrability estimates of the velocity $u$ and the effective velocity $w$.
\begin{lema}
Under the assumption of \eqref{assumptiongamma1},  there exists a constant $C(T) > 0$ such that
    \begin{equation}\label{xrbd3}
   x \leq C(T)\, r(x,\tau)^2, \quad (x,\tau) \in [0,1] \times [0,T].
\end{equation}
\end{lema}
\begin{proof}
From the three-dimensional Sobolev embedding inequality and \eqref{bdsfinall}, we obtain
\begin{equation}
    \begin{aligned}\label{rho3t}
 \Big(\int^{R}_{0}\rho ^{3}r^{2}dr\Big)^{\frac{1}{3}} = &\frac{1}{(4\pi)^{1/3}}\left\lVert \rho \right\rVert_{L^{3}(\Omega)}=\frac{1}{(4\pi)^{1/3}}\left\lVert \rho^{\frac{1}{2}} \right\rVert _{L^{6}}^{2} \\
\leq& \frac{1}{(4\pi)^{1/3}}(\left\lVert \rho^{\frac{1}{2}} \right\rVert_{L^{2}}^{2}+\left\lVert \nabla \rho^{\frac{1}{2}} \right\rVert_{L^{2}(\Omega)} ^{2} )\leq C(T).
\end{aligned}
\end{equation}
Consequently, the above inequality implies that
\begin{equation}
    \begin{aligned}
x&= \int^{r}_{0}\rho s^{2}ds\leq \left( \int^{r}_{0}\rho^{3}s^{2}ds \right)^{\frac{1}{3}}\left( \int^{r}_{0}s^{2}ds \right)^{\frac{2}{3}}\\
&\leq C(T)r^{2}.
\end{aligned}
\end{equation}
\end{proof}

\begin{prop}
Under the assumption of \eqref{assumptiongamma1},  there exists a constant $C(T) > 0$ such that
    \begin{equation}\label{rhoupperbound3d}
    R_T\leq C(T).
\end{equation}
\end{prop}
\begin{proof}
Let $n\geq 4$ be fixed such that $\gamma<3-\frac{1}{n}$, and let $\sigma$ be a constant with $0\leq\sigma<1-\frac{1}{n}$ and satisfying

\begin{equation}\label{gammasigma}
1-\frac{1}{2n}+\sigma<\gamma<2+\sigma.
\end{equation}
   By multiplying $\eqref{eq:NSP-lagrangian-rhoalpha}_2$ by $u^{2n-1}$ and then integrating by parts, we obtain
   \begin{equation}
        \begin{aligned}\label{3uhigh1}
            \frac{1}{2n} \frac{d}{d\tau}\int^{1}_{0}u^{2n}dx+& 2\int^{1}_{0} \frac{u^{2n}}{r^{2}}dx+(2n-1)\int^{1}_{0}\rho^{2}u^{2n-2}u_{x}^{2}r^{4}dx\\
           &=-\int^{1}_{0}\rho^{\gamma}(r^{2}u^{2n-1})_{x}dx -\kappa\int^{1}_{0} \frac{x}{r^{2}}u^{2n-1}dx +\frac{\kappa}{3}\int^{1}_{0}u^{2n-1}rdx\\
           &\leq C\int^{R}_{0}\rho^{2n(\gamma-1)+1}r^{2n}dr+\frac{1}{2}\int_{0}^{1}\frac{u^{2n}}{r^2}dx+C(T)\\
&\leq C\left\lVert \rho^{\frac{1}{2}}r^{\frac{1}{2}+\xi} \right\rVert ^{2(2n(\gamma-1)+1-2n\sigma)}_{L^{\infty}}R_{T}^{2n\sigma}\int^{R}_{0}r^{2n-2(\frac{1}{2}+\xi)[2n(\gamma-1)+1-2n\sigma]}dr\\
&\quad+\frac{1}{2}\int_{0}^{1}\frac{u^{2n}}{r^2}dx+C(T)\\
&\leq CR_{T}^{2n\sigma}\int^{R}_{0}r^{2n-(1+2\xi)[2n(\gamma-1)+1-2n\sigma]}dr+C(T)\\
&\leq C(T)R_{T}^{2n\sigma}+\frac{1}{2}\int_{0}^{1}\frac{u^{2n}}{r^2}dx,
        \end{aligned}
    \end{equation}
 where $\xi$ is taken sufficiently small so that
\begin{equation}
    2n-(1+2\xi)[2n(\gamma-1)+1-2n\sigma]>-1\implies \xi<\frac{n(2-\gamma+\sigma)}{[2n(\gamma-1)+1-2n\sigma]}.
\end{equation}
The existence of $\xi$ is ensured by \eqref{gammasigma}.
Integrating \eqref{3uhigh1} over $[0,\tau]$ gives
\begin{equation}\begin{aligned}
\frac{1}{2n}\int^{1}_{0}u^{2n}dx+& \int^{\tau}_{0}\int^{1}_{0} \frac{u^{2n}}{r^{2}}dxds+ \frac{2n-1}{2}\int^{\tau}_{0}\int^{1}_{0} \rho^{2}u^{2n-2}(r^{2}u_{x})^{2}dxds\\
&\leq \frac{1}{2n}\int^{R}_{0}\rho u_{0}^{2n}r^{2}dr+ C(T)R_{T}^{2n\sigma}.
\end{aligned}\end{equation}
In Eulerian coordinates, we have
\begin{equation}\begin{aligned}\label{uhighintegratibility}
&\int^{R}_{0}\rho u^{2n}r^{2}dr+2n\int^{\tau}_{0}\int^{R}_{0}\rho u^{2n}drds+ \frac{(2n-1)(2n)}{2}\int^{\tau}_{0}\int^{R}_{0}\rho u^{2n-2}(ru_{r})^{2}drds\\
\leq& \int^{R}_{0}\rho u_{0}^{2n}r^{2}dr+2nC(T)R_{T}^{2n\sigma}\leq 2nC(T)R_{T}^{2n\sigma}.
\end{aligned}\end{equation}
 By multiplying \eqref{bd2_alpha<1} by $w^{2n-1}$ we obtain 
\begin{equation}\begin{aligned}\label{bdw}
    \frac{1}{2n} \frac{d}{d\tau} \int^{1}_{0}w^{2n}dx=&- \int^{1}_{0} (\rho^{\gamma})_{x}r^{2}w^{2n-1}dx  -\kappa\int^{1}_{0} \frac{x}{r^{2}}w^{2n-1}dx +  \frac{1}{3}\int^{1}_{0}w^{2n-1}rdx.\\
\end{aligned}\end{equation}
Similar to the estimates for the case $\alpha < 1$, we have
\begin{equation}\begin{aligned}
    \frac{1}{2n} \frac{d}{d\tau}\int^{1}_{0}w^{2n}dx+ \frac{\gamma}{2}\int^{1}_{0}\rho^{\gamma-1}w^{2n}dx&\leq C\int^{1}_{0}w^{2n}dx+C(T)+C\int^{1}_{0}\rho^{\gamma-1}u^{2n}dx\\
&\leq C\int^{1}_{0}w^{2n}dx+C(T)+C\int^{R}_{0}\rho^{\gamma}u^{2n}r^{2}dr \\
&\leq C\int^{1}_{0}w^{2n}dx+C(T)+C\left\lVert \rho^{\frac{1}{2}}r^{\frac{1}{2}+\xi} \right\rVert _{L^{\infty}}^{2(\gamma-1)}\int^{R}_{0}\rho u^{2n}dr\\
&\leq C\int^{1}_{0}w^{2n}dx+C(T)+C(T)\int^{R}_{0}\rho u^{2n}dr,
\end{aligned}\end{equation}
where taking $\xi$ small enough such that
\begin{equation}
    2- 2\left( \frac{1}{2}+\xi \right)(\gamma-1)>0\implies \xi< \frac{3-\gamma}{2(\gamma-1)}.
\end{equation}
The condition $\gamma < 3$ guarantees the existence of $\xi$.
We apply Gronwall's inequality and  \eqref{uhighintegratibility} to deduce that
\begin{equation}
    \begin{aligned}\label{whighintegratibility}
\sup_{0\leq t\leq T}\int^{R}_{0}\rho w^{2n}r^{2}dr+ \int^{T}_{0}\int^{1}_{0}\rho^{\gamma}w^{2n}r^{2}drds&\leq 2nC(T)+2nC(T)\int^{T}_{0}\int^{R}_{0}\rho u^{2n}drds\\
&\leq 2nC(T)+2nC(T)R_{T}^{2n\sigma}\\
&\leq 2nC(T)R_{T}^{2n\sigma}.
\end{aligned}
\end{equation}
We obtain from \eqref{uhighintegratibility} and \eqref{whighintegratibility} that
\begin{equation}
    \begin{aligned}\label{rrhoxn}
\int^{1}_{0}(\rho_{x}r^{2})^{2n}dx&= \int^{R}_{0}\rho\left( \frac{\rho_{r}}{\rho} \right)^{2n}r^{2}dr\\
&= \int^{R}_{0}\left({2n} \partial_{r}\rho^{\frac{1}{2n}} \right)^{2n}r^{2}dr. \\
&\leq (2n)^{2n} C(T)R_{T}^{2n\sigma}.
\end{aligned}
\end{equation}
By the Sobolev embedding theorem in one dimension, we have
\begin{equation}
    \begin{aligned}
\left\lVert \rho^{\beta} \right\rVert _{L^{\infty}(0,R)}&\leq \int^{R}_{0}\rho^{\beta}dr+\int^{R}_{0}\partial_{r}\rho^{\beta}dr:=G_{1}+G_{2}.
\end{aligned}
\end{equation}
Using \eqref{rho3t} and H\"older's inequality yields
\begin{equation}\begin{aligned}\label{rhobeta}
G_{1}&= \int^{R}_{0}\rho^{\beta}dr=\int^{R}_{0}(\rho^{3}r^{2})^{\frac{\beta}{3}}r^{- \frac{2}{3}\beta}dr\\
&\leq\left( \int^{R}_{0}(\rho^{3}r^{2})^{\frac{\beta}{3} \frac{3}{\beta}}dr \right)^{\frac{\beta}{3}}\left( \int^{R}_{0}r^{- \frac{2\beta}{3} \frac{3}{3-\beta}}dr \right)^{\frac{3-\beta}{3}}\\
&\leq C(T)\left( \int^{R}_{0}r^{- \frac{2\beta}{3-\beta}}dr \right)^{\frac{3-\beta}{3}}\\
&\leq C(T),
\end{aligned}\end{equation}
where we need
\begin{equation}
    - \frac{2\beta}{3-\beta}>-1\implies \frac{3\beta-3}{3-\beta}<0\implies \beta<1.
\end{equation}
H\"older's inequality and \eqref{rrhoxn} gives that
\begin{equation}\begin{aligned}\label{rhobeta2}
G_{2}&= \int^{R}_{0} \beta \rho^{\beta-1}\rho_{r}dr\leq \beta\left( \int^{R}_{0}\rho^{\frac{1}{2n}-1}\rho_{r}\rho^{1- \frac{1}{2n}}\rho^{\beta-1} \right)dr\\
&\leq \beta \left( \int^{R}_{0}|2n\partial_{r}\left( \rho^{\frac{1}{2n}} \right)|^{2n}r^{2}dr \right)^{\frac{1}{2n}}\left( \int^{R}_{0}\left( \rho^{\beta- \frac{1}{2n}}r^{- \frac{1}{n}} \right)^{\frac{2n}{2n-1}} \right)^{\frac{2n-1}{2n}}\\
&\leq \beta C(T)(R_{T}^{2n \sigma})^{\frac{1}{2n}}\left( \int^{R}_{0}\rho^{\frac{2n\beta-1}{2n-1}}r^{- \frac{2}{2n-1}}dr \right)^{\frac{2n-1}{2n}}\\
&\leq \beta C(T)R_{T}^{\sigma}\left\lVert \rho^{\frac{1}{2}}r^{\frac{1}{2}+\xi} \right\rVert _{L^{\infty}}^{\frac{2(2n\beta-1)}{2n-1}}\left( \int^{R}_{0}r^{- \frac{2}{2n-1}- (1+2\xi) \frac{2n\beta-1}{2n-1}}dr \right)^{\frac{2n-1}{2n}}\\
&\leq \beta C(T) R_{T}^{\sigma}\left( \int^{R}_{0}r^{- \frac{2}{2n-1}- (1+2\xi) \frac{2n\beta-1}{2n-1}}dr \right)^{\frac{2n-1}{2n}},
\end{aligned}\end{equation} 
under which we require
\begin{equation}
    - \frac{2}{2n-1}- (1+2\xi) \frac{2n\beta-1}{2n-1}> -1\implies \xi< \frac{n-1- n\beta}{2n\beta-1},
\end{equation}
that is,
\begin{equation}
    \begin{cases}
n-1-n\beta>0\implies\beta< 1- \frac{1}{n}. \\
2n\beta-1\implies \beta> \frac{1}{2n}.
\end{cases}
\end{equation}
We deduce from \eqref{rhobeta2} and \eqref{rhobeta} that
\begin{equation}
    R_{T}^{\beta}\leq C(T)+C(T)R_{T}^{\sigma}\implies R_{T}\leq C(T),\beta>\sigma.
\end{equation}
\end{proof}
\begin{lema}
\label{lem3.90}

Under the assumption of \eqref{assumptiongamma1}, there exists a constant $C(T) > 0$ such that 
\begin{equation}\label{2uinfty0}
\sup_{0\leq \tau \leq T}\| u(\cdot,\tau)\|_{L^{\infty}(0,1)}\leq C(T),
\end{equation}
and
\begin{equation}\label{2winfty}
\sup_{0\leq \tau \leq T}\| w(\cdot,\tau)\|_{L^{\infty}(0,1)}\leq C(T).
\end{equation}
\end{lema}
\begin{prop}
\label{prop3.50}
Under the assumption of  \eqref{assumptiongamma1}, there exists a constant $C(T) > 0$ such that
\begin{equation}\begin{aligned}\label{lower bound of density00}
V_T\leq C(T). 
\end{aligned}\end{equation}
\end{prop}
\begin{proof}
Using \eqref{rrhoxn}, \eqref{rhoupperbound3d} yields
\begin{equation}\begin{aligned}
\sup_{0\leq t\leq T}\| r^2\rho_x \|_{L^{\infty}(0,1)}\leq C(T). 
\end{aligned}\end{equation}
Define the specific volume $v = 1/\rho$. 

Utilizing $\eqref{eq:NSP-lagrangian-rhoalpha}_2$, a direct computation yields the evolution equation for $v$:
\begin{equation}\begin{aligned}
\partial_{\tau}v =\partial_{x}(r u),
\end{aligned}\end{equation}
hence, we get
\begin{equation}\begin{aligned}
\int_{0}^{1}v(x,\tau)\,dx = \int_{0}^{1}v_{0}(x)\,dx.
\end{aligned}\end{equation}

To derive the \(L^{\infty}\) bound for \(v\), we employ one-dimensional Sobolev embedding \(W^{1,1}(0,R) \hookrightarrow L^{\infty}(0,R)\). For a sufficiently small parameter \(\beta \in (0,1)\) to be chosen later, the following estimate holds for any \((x,\tau) \in (0,1)\times(0,T)\):
\begin{equation}\begin{aligned}
v^{\beta}(x,\tau) &\leq \int_{0}^{1}v^{\beta}dx + \beta \int_{0}^{1}v^{\beta -1}|v_x|\,dx \\
&\leq \left(\int_{0}^{1}v\,dx\right)^{\beta} + \beta \int_{0}^{1}|\rho_x |\,v^{\beta +1}\,dx\\
& \leq \left(\int_{0}^{1}v\,dx\right)^{\beta} + C\beta\|r^2\rho_x\|_{L^\infty(0,1)} \left(\int_{0}^{1} v r^{-\frac{8}{3}}\,dx\right)^{\frac{3}{4}} \left(\int_{0}^{1} v^{4\beta+1}\,dx\right)^{\frac{1}{4}}\\
& \leq \left(\int_{0}^{1}v\,dx\right)^{\beta} + C(T)\beta\left( \sup_{(x,\tau)} v(x,\tau) \right)^{\beta}.
\end{aligned}\end{equation}
We have
\[V_T^{\beta} \leq C(\beta) + C(T)\beta V_T^{\beta}.
\]
We can now select \(\beta= \min\{1/(2C(T)), 1/2\}\)  yielding \eqref{lower bound of density00}.
\end{proof}
\subsection{A priori estimates: Higher-Order estimates}
The higher-order estimates can be obtained in a manner similar to that for $\alpha < 1$.

\subsection{Proof of the Theorem \ref {thm:global-classical-2d} and \ref{thm:global-classical-3d}}
The proof for the case $\alpha=1$ is similar to that for $\alpha<1.$




\vspace{1cm}
\noindent\textbf{Data availability statement.} Data sharing is not applicable to this article.

\vspace{0.3cm}
\noindent\textbf{Conflict of interest.} The authors declare that they have no conflict of interest.

\section*{Ackonwledgments}
X Huang is partially supported by Chinese Academy of Sciences Project for Young Scientists in Basic Research (Grant No. YSBR-031), National Natural Science Foundation of China (Grant Nos. 12494542, 11688101). J Fan was supported by the China Postdoctoral Science Foundation (Grant 2025M783153).

\end{document}